\documentclass[a4paper, 10pt, notitlepage]{article}

\usepackage{amsthm, amsmath, amssymb, latexsym}
\usepackage{mathrsfs,cite}

\theoremstyle{plain}
\newtheorem{theorem}{Theorem}

\newtheorem{corollary}{Corollary}
\newtheorem{proposition}{Proposition}

\theoremstyle{definition}
\newtheorem{definition}{Definition}
\newtheorem{example}{Example}

\theoremstyle{remark}
\newtheorem{remark}{Remark}

\usepackage[a4paper]{geometry}	

\DeclareMathOperator{\dist}{dist}
\DeclareMathOperator{\interior}{int}
\DeclareMathOperator{\relint}{relint}
\DeclareMathOperator*{\esssup}{ess\,sup}
\DeclareMathOperator*{\linhull}{span}
\DeclareMathOperator{\image}{Im}
\DeclareMathOperator{\core}{core}
\DeclareMathOperator{\cl}{cl}
\DeclareMathOperator{\co}{co}
\DeclareMathOperator{\cone}{cone}
\DeclareMathOperator{\sign}{sign}

\author{M.V. Dolgopolik}
\title{Exact Penalty Functions for Optimal Control Problems {I}{I}: 
Exact Penalisation of Terminal and Pointwise State Constraints}

\begin{document}

\maketitle

\begin{abstract}
The second part of our study is devoted to an analysis of the exactness of penalty functions for
optimal control problems with terminal and pointwise state constraints. We demonstrate that with the use of the exact
penalty function method one can reduce fixed-endpoint problems for linear time-varying systems and linear evolution
equations with convex constraints on the control inputs to completely equivalent free-endpoint optimal control problems,
if the terminal state belongs to the relative interior of the reachable set. In the nonlinear case, we prove that a
local reduction of fixed-endpoint and variable-endpoint problems to equivalent free-endpoint ones is possible under the
assumption that the linearised system is completely controllable, and point out some general properties of nonlinear
systems under which a global reduction to equivalent free-endpoint problems can be achieved. In the case of problems
with pointwise state inequality constraints, we prove that such problems for linear time-varying systems and linear
evolution equations with convex state constraints can be reduced to equivalent problems without state constraints,
provided one uses the $L^{\infty}$ penalty term, and Slater's condition holds true, while for nonlinear systems a local
reduction is possible, if a natural constraint qualification is satisfied. Finally, we show that the exact
$L^p$-penalisation of state constraints with finite $p$ is possible for convex problems, if Lagrange multipliers
corresponding to the state constraints belong to $L^{p'}$, where $p'$ is the conjugate exponent of $p$, and for general
nonlinear problems, if the cost functional does not depend on the control inputs explicitly.
\end{abstract}

\section{Introduction}

The exact penalty method is an important tool for solving constrained optimisation problems. Many publications have
been devoted to its analysis from various perspectives (see, e.g. References
\cite{EvansGouldTolle,HanMangasarian,DiPilloGrippo86,DiPilloGrippo88,DiPilloGrippo89,Burke91,DiPillo94,ExactBarrierFunc,
Zaslavski,Dolgopolik_ExPen_I, Dolgopolik_ExPen_II,Strekalovsky2019}). The main idea of this method consist in the
reduction of a constrained optimisation problem, say
\begin{equation} \label{MathProgram_Intro}
  \min_{x \in \mathbb{R}^d} f(x) \quad \text{subject to} \quad g_i(x) \le 0, \quad i \in \{ 1, \ldots, n \}
\end{equation}
to the unconstrained optimisation problem of minimising the nonsmooth penalty function: 
$$
  \min_{x \in \mathbb{R}^d} \Phi_{\lambda}(x) = f(x) + \lambda \sum_{i = 1}^n \max\{ g_i(x), 0 \}.
$$
Under some natural assumptions such as the coercivity of $\Phi_{\lambda}$ and the validity of a suitable constraint
qualification one can prove that for any sufficiently large (but finite) value of the penalty parameter $\lambda$ the
penalised problem is equivalent to the original problem in the sense that these problems have the same optimal value and
the same globally optimal solutions. In this case the penalty function $\Phi_{\lambda}$ is called \textit{exact}. Under
some additional assumptions not only globally optimal solutions, but also locally optimal solutions and stationary
(critical) points of these problems coincide. In this case the penalty function $\Phi_{\lambda}$ is called
\textit{completely exact}. Finally, if a given locally optimal solution of problem \eqref{MathProgram_Intro} is a point
of local minimum of $\Phi_{\lambda}$ for any sufficiently large $\lambda$, then $\Phi_{\lambda}$ is said to be
\textit{locally exact} at this solution.

Thus, the exactness property of a penalty function allows one to reduce (locally or globally) a constrained optimisation
problem to the equivalent unconstrained problem of minimising a penalty function and, as a result, apply numerical
methods of unconstrained optimisation to constrained problems. However, note that exact penalty functions not depending
on the derivatives of the objective function and constraints are inherently nonsmooth (see, e.g. 
\cite[Remark~3]{Dolgopolik_ExPen_I}), and one has to either utilise general methods of nonsmooth optimisation to
minimise exact penalty functions or develop specific methods for minimising such functions that take into account their
structure. See~the works of M\"{a}kel\"{a} et al.\cite{Makela2002,KarmitsaBagriovMakela2012,BagirovKarmitsaMakela_book}
for a survey and comparative analysis of modern nonsmooth optimisation methods and software.

Numerical methods for solving optimal control problems based on exact penalty functions were developed by Maratos
\cite{MaratosPHD} and in a series of papers by Mayne et al.
\cite{MaynePolak80,MayneSmith83,MaynePolak85,MaynePolak87,SmithMayne88} (see also monograph
\cite{Polak_book}). An exact penalty method for optimal control problems with delay was proposed by Wong and Teo
\cite{WongTeo91}, and such method for some nonsmooth optimal control problems was studied in the works of Outrata
et al. \cite{Outrata83,OutrataSchindler,Outrata88}. A continuous numerical method for optimal control problems
based on the direct minimisation of an exact penalty function was considered in recent paper
\cite{FominyhKarelin2018}. Finally, closely related methods based on Huyer and Neumaier's exact penalty
function\cite{HuyerNeumaier,WangMaZhou,Dolgopolik_ExPen_II} were developed for optimal control problems with state
inequality constraints\cite{LiYu2011,JiangLin2012} and optimal feedback control problems\cite{LinLoxton2014}.

Despite the abundance of publications on exact penalty methods for optimal control problems, relatively little
attention has been paid to an actual analysis of the exactness of penalty functions for such problems. To the best of
author's knowledge, the possibility of the exact penalisation of pointwise state constraints for optimal control
problems was first mentioned by Luenberger \cite{Luenberger70}; however, no particular conditions ensuring exact
penalisation were given in this paper. Lasserre \cite{Lasserre} proved that a stationary point of an optimal control
problem with endpoint equality and state inequality constraints is also a stationary point of a nonsmooth penalty
function for this problem (this result is closely related to the local exactness). The local exactness of a penalty
function for problems with state inequality constraints was proved by Xing et al. \cite{Xing89,Xing94} under the
assumption that certain second order sufficient optimality conditions are satisfied. First results on the
\textit{global} exactness of penalty functions for optimal control problems were probably obtained by Demyanov et al.
for a problem of finding optimal parameters in a system described by ordinary differential equations 
\cite{DemyanovKarelin98}, free-endpoint optimal control problems 
\cite{DemyanovKarelin2000_InCollect,DemyanovKarelin2000,Karelin}, and certain optimal control problems for implicit 
control systems \cite{DemyanovTamasyan2005}. However, the main results of these papers are based on the assumptions
that the penalty function attains a global minimum in the space of piecewise continuous functions for any sufficiently
large value of the penalty parameter, and the cost functional is Lipschitz continuous on a possibly unbounded and rather
complicated set. It is unclear how to verify these assumptions in particular cases, which makes it very difficult to
apply the main results of papers
\cite{DemyanovKarelin98,DemyanovKarelin2000_InCollect,DemyanovKarelin2000,Karelin,DemyanovTamasyan2005} to real
problems. To the best of author's knowledge, the only verifiable sufficient conditions for the global exactness of
penalty functions for optimal control problems were obtained by Gugat\cite{Gugat} for an optimal control of the wave
equation, by Gugat and Zuazua\cite{Zuazua} for optimal control problems for general linear evolution equations, and by
Jayswal and Preeti\cite{Jayswal} for a PDE constrained optimal control problem with state inequality constraints. In
papers\cite{Gugat,Zuazua} only the exact penalisation of terminal constraint was analysed, while in
article\cite{Jayswal} the exact penalisation of state inequality constraints and system dynamics was proved under
some restrictive convexity assumptions.

The main goal of this two-part study is to develop a general theory of exact penalty functions for optimal control
problems containing sufficient conditions for the complete or local exactness of penalty functions that can be 
readily verified in various particular cases. In the first part of our study\cite{DolgopolikFominyh} we obtained simple
sufficient conditions for the exactness of penalty functions for free-endpoint optimal control problems. This result
allows one to apply numerical method for solving variational problems to free-endpoint optimal control problems.

In the second part of our study we analyse the exactness of penalty functions for problems with terminal and pointwise
state constraints. In the first half of this paper, we study when a penalisation of the terminal constraint is exact,
i.e. when the fixed-endpoint problem
\begin{align*}
  &\min_{(x, u)} \: \mathcal{I}(x, u) = \int_0^T \theta(x(t), u(t), t) \, dt \\
  &\text{subject to} \quad \dot{x}(t) = f(x(t), u(t), t), \quad t \in [0, T], \quad 
  x(0) = x_0, \quad x(T) = x_T, \quad u \in U
\end{align*}
is equivalent to the penalised free-endpoint one
\begin{align*}
  &\min_{(x, u)} \: \Phi_{\lambda}(x, u) = \int_0^T \theta(x(t), u(t), t) \, dt + \lambda | x(T) - x_T | \\
  &\text{subject to} \quad \dot{x}(t) = f(x(t), u(t), t), \quad t \in [0, T], \quad 
  x(0) = x_0, \quad u \in U.
\end{align*}
We prove that fixed-endpoint problems for linear time-varying systems and linear evolution equations in Hilbert spaces
with convex constraints on the control inputs (i.e. the set $U$ is convex) are equivalent to the corresponding penalised
free-endpoint problems, if the terminal state $x_T$ belongs to the relative interior of the reachable set. This result
significantly generalises the one of Gugat and Zuazua \cite{Zuazua} (see Remark~\ref{Remark_ComparisonZuazua} for a
detailed discussion). In the case of nonlinear problems, we show that the penalty function $\Phi_{\lambda}$ is locally
exact at a given locally optimal solution, if the corresponding linearised system is completely controllable, and point
our some general assumption on the system $\dot{x} = f(x, u, t)$ that ensure the complete exactness of
$\Phi_{\lambda}$. We also present an extension of these results to the case of variable-endpoint problems.

In the second half of this paper, we study the exact penalisation of pointwise state constraints, i.e. we study when
the optimal control problem with pointwise state inequality constraints
\begin{align*}
  &\min_{(x, u)} \: \mathcal{I}(x, u) = \int_0^T \theta(x(t), u(t), t) \, dt \quad
  \text{subject to} \quad \dot{x}(t) = f(x(t), u(t), t), \quad t \in [0, T], \\
  &x(0) = x_0, \quad x(T) = x_T, \quad u \in U, \quad 
  g_j(x(t), t) \le 0 \quad \forall t \in [0, T], \quad j \in \{ 1, \ldots, l \}
\end{align*}
is equivalent to the penalised problem without state constraints
\begin{align*}
  &\min_{(x, u)} \: \Phi_{\lambda}(x, u) = \int_0^T \theta(x(t), u(t), t) \, dt + 
  \lambda \| \max\{ g_1(x(\cdot), \cdot), \ldots, g_l(\cdot), \cdot), 0 \} \|_p \\
  &\text{subject to} \quad \dot{x}(t) = f(x(t), u(t), t), \quad t \in [0, T], \quad
  x(0) = x_0, \quad x(T) = x_T, \quad u \in U
\end{align*}
for some $1 \le p \le + \infty$ (here $\| \cdot \|_p$ is the standard norm in $L^p(0, T)$). In the case of problems for
linear time-varying systems and linear evolution equation with convex state constraints, we prove that the penalisation
of state constraints is exact, if $p = + \infty$, and Slater's condition holds true, i.e. there exists a feasible point
$(x, u)$ such that $g_j(x(t), t) < 0$ for all $t \in [0, T]$ and $j \in \{ 1, \ldots, l \}$. In the nonlinear case, we
prove the local exactness of $\Phi_{\lambda}$ with $p = + \infty$ under the assumption that a suitable constraint
qualification is satisfied. Finally, we demonstrate that under some additional assumptions the exact $L^p$ penalisation
of state constraints with finite $p$ is possible for convex problems, if Lagrange multipliers corresponding to state
constraints belong to $L^{p'}(0, T)$, and for nonlinear problems, if the cost functional $\mathcal{I}$ does not depend
on the control inputs $u$ explicitly.

The paper is organised as follows. Some basic definitions and results from the general theory of exact penalty
functions for optimisation problems in metric spaces are collected in Section~\ref{Sect_ExactPenaltyFunctions}, so that
the second part of the paper can be read independently of the first one. Section~\ref{Sect_ExactPen_TerminalConstraint}
is devoted to the analysis of exact penalty functions for fixed-endpoint and variable-endpoint problems, while exact
penalty functions for optimal control problems with state constraints are considered in
Section~\ref{Sect_ExactPen_StateConstraint}. Finally, a proof of a general theorem on completely exact penalty function
from Section~\ref{Sect_ExactPenaltyFunctions} is given in Appendix~A, while
Appendix~B contains some useful results on Nemytskii operators that are utilised throughout the
paper.

\section{Exact Penalty Functions in Metric Spaces}
\label{Sect_ExactPenaltyFunctions}

In this section we recall some basic definitions and results from the theory of exact penalty functions that will be
utilised throughout the article (see papers\cite{Dolgopolik_ExPen_I,Dolgopolik_ExPen_II} for more details). Let
$(X, d)$ be a metric space, $M, A \subseteq X$ be nonempty sets such that $M \cap A \ne \emptyset$, and 
$\mathcal{I} \colon X \to \mathbb{R} \cup \{ + \infty \}$ be a given function. Consider the following optimisation
problem:
$$
  \min_{x \in X} \: \mathcal{I}(x) \quad \text{subject to} \quad x \in M \cap A. \eqno{(\mathcal{P})}
$$
Here the sets $M$ and $A$ respresent two different types of constraints, e.g. pointwise and terminal constraints or
linear and nonlinear constraints, etc. In what follows, we suppose that there exists a globally optimal solution $x^*$
of the problem $(\mathcal{P})$ such that $\mathcal{I}(x^*) < + \infty$, i.e. the optimal value of this problem is finite
and is attained.

Our aim is to ``get rid'' of the constraint $x \in M$ without losing any essential information about (locally or
globally) optimal solutions of the problem $(\mathcal{P})$. To this end, we apply the exact penalty function technique.
Let $\varphi \colon X \to [0, + \infty]$ be a function such that $\varphi(x) = 0$ iff $x \in M$. For example, if $M$ is
closed, one can put $\varphi(x) = \dist(x, M) = \inf_{y \in M} d(x, y)$. For any $\lambda \ge 0$ define 
$\Phi_{\lambda}(x) = \mathcal{I}(x) + \lambda \varphi(x)$. The function $\Phi_{\lambda}$ is called \textit{a penalty
function} for the problem $(\mathcal{P})$ (corresponding to the constraint $x \in M$), $\lambda$ is called 
\textit{a penalty parameter}, and $\varphi$ is called \textit{a penalty term} for the constraint $x \in M$. 

Observe that the function $\Phi_{\lambda}(x)$ is non-decreasing in $\lambda$, and $\Phi_{\lambda}(x) \ge \mathcal{I}(x)$
for all $x \in X$ and $\lambda \ge 0$. Furthermore, for any $\lambda > 0$ one has $\Phi_{\lambda}(x) = \mathcal{I}(x)$
iff $x \in M$. Therefore, it is natural to consider \textit{the penalised problem}
\begin{equation} \label{PenalizedProblem}
  \min_{x \in X} \Phi_{\lambda}(x) \quad \text{subject to} \quad x \in A.
\end{equation}
Note that this problem has only one constraint ($x \in A$), while the constraint $x \in M$ is incorporated into 
the new objective function $\Phi_{\lambda}$. We would like to know when this problem is, in some sense, equivalent to
the problem $(\mathcal{P})$, i.e. when the penalty function $\Phi_{\lambda}$ is \textit{exact}. 

\begin{definition}
The penalty function $\Phi_{\lambda}$ is called (globally) \textit{exact}, if there exists $\lambda^* \ge 0$ such that
for any $\lambda \ge \lambda^*$ the set of globally optimal solutions of the penalised problem \eqref{PenalizedProblem}
coinsides with the set of globally optimal solutions of the problem $(\mathcal{P})$.
\end{definition}

From the fact that $\Phi_{\lambda}(x) = \mathcal{I}(x)$ for any feasible point $x$ of the problem $(\mathcal{P})$ it
follows that if $\Phi_{\lambda}$ is globally exact, then the optimal values of the problems $(\mathcal{P})$ and
\eqref{PenalizedProblem} coincide. Thus, the penalty function $\Phi_{\lambda}$ is globally exact iff 
the problems $(\mathcal{P})$ and \eqref{PenalizedProblem} are equivalent in the sense that they have the same globally
optimal solutions and the same optimal value. However, optimisation methods often can find only locally optimal
solutions (or even only stationary/critical points) of an optimisation problem. Therefore, the concept of the global
exactness of the penalty function $\Phi_{\lambda}$ is not entirely satisfactory for practical applications. One needs
to ensure that not only globally optimal solutions, but also local minimisers and stationary points of the problems
$(\mathcal{P})$ and \eqref{PenalizedProblem} coincide. To provide conditions under which such \textit{complete
exactness} takes place we need to recall the definitions of the \textit{rate of steepest
descent}\cite{Demyanov2000,Demyanov2010,Uderzo} and \textit{inf-stationary point}\cite{Demyanov2000,Demyanov2010} of a
function defined on a metric space.

Let $K \subset X$ and $f \colon X \to \mathbb{R} \cup \{ + \infty \}$ be given, and $x \in K$ be such that 
$f(x) < + \infty$. The quantity
$$
  f^{\downarrow}_K(x) = \liminf_{y \to x, y \in K} \frac{f(y) - f(x)}{d(y, x)}
$$
is called \text{the rate of steepest descent} of $f$ with respect to the set $K$ at the point $x$. If $x$ is an
isolated point of $K$, then by definition $f^{\downarrow}_K(x) = + \infty$. It should be noted that the rate of steepest
descent of $f$ at $x$ is closely connected to the so-called strong slope $|\nabla f|(x)$ of $f$ at $x$. See
papers\cite{Dolgopolik_ExPen_II,Aze,Kruger} for some calculus rules for strong slope/rate of steepest descent, and the
ways one can estimate them in various particular cases.

Let $x^* \in K$ be such that $f(x^*) < + \infty$. The point $x^*$ is called an \textit{inf-stationary} point of $f$ on
the set $K$ if $f^{\downarrow}_K(x^*) \ge 0$. Observe that the inequality $f^{\downarrow}_K(x^*) \ge 0$ is a necessary
optimality condition for the problem
$$
  \min_{x \in X} \: f(x) \quad \text{subject to} \quad x \in K.
$$
In the case when $X$ is a normed space, $K$ is convex, and $f$ is Fr\'{e}chet differentiable at $x^*$ the inequality
$f^{\downarrow}_K(x^*) \ge 0$ is reduced to the standard optimality condition: $f'(x^*)[x - x^*] \ge 0$ for all $x \in
K$, where $f'(x^*)$ is the Fr\'{e}chet derivative of $f$ at $x^*$.

Now we can formulate sufficient conditions for the complete exactness of the penalty function $\Phi_{\lambda}$.
For any $\lambda \ge 0$ and $c \in \mathbb{R}$ denote $S_{\lambda}(c) = \{ x \in A \mid \Phi_{\lambda}(x) < c \}$. Let
also $\Omega = M \cap A$ be the feasible region of $(\mathcal{P})$, and for any $\delta > 0$ define
$\Omega_{\delta} = \{ x \in A \mid \varphi(x) < \delta \}$.

\begin{theorem} \label{Theorem_CompleteExactness}
Let $X$ be a complete metric space, $A$ be closed, $\mathcal{I}$ and $\varphi$ be lower semicontinuous on $A$, and
$\varphi$ be continuous at every point of the set $\Omega$. Suppose also that there exist 
$c > \mathcal{I}^* = \inf_{x \in \Omega} \mathcal{I}(x)$, $\lambda_0 > 0$, and $\delta > 0$ such that
\begin{enumerate}
\item{there exists an open set $V$ such that $S_{\lambda_0}(c) \cap \Omega_{\delta} \subset V$ and the functional
$\mathcal{I}$ is Lipschitz continuous on $V$;
}

\item{there exists $a > 0$ such that $\varphi^{\downarrow}_A(x) \le - a$ for all 
$x \in S_{\lambda_0}(c) \cap (\Omega_{\delta} \setminus \Omega)$;
\label{NegativeDescentRateAssumpt}}

\item{$\Phi_{\lambda_0}$ is bounded below on $A$.
}
\end{enumerate}
Then there exists $\lambda^* \ge 0$ such that for any $\lambda \ge \lambda^*$ the following statements hold true:
\begin{enumerate}
\item{the optimal values of the problems $(\mathcal{P})$ and \eqref{PenalizedProblem} coincide;
}

\item{globally optimal solutions of the problems $(\mathcal{P})$ and \eqref{PenalizedProblem} coincide;
}

\item{$x^* \in S_{\lambda}(c)$ is a locally optimal solution of the penalised problem \eqref{PenalizedProblem} iff 
$x^* \in \Omega$, and it is a locally optimal solution of the problem $(\mathcal{P})$;
}

\item{$x^* \in S_{\lambda}(c)$ is an inf-stationary point of $\Phi_{\lambda}$ on $A$ iff $x^* \in \Omega$, and it is an
inf-stationary point of $\mathcal{I}$ on $\Omega$.
}
\end{enumerate}
\end{theorem}

If the penalty function $\Phi_{\lambda}$ satisfies the four statements of this theorem, then it is said to be
\textit{completely exact} on the set $S_{\lambda}(c)$. The proof of Theorem~\ref{Theorem_CompleteExactness} is given in
the first part of our study\cite{DolgopolikFominyh}.

\begin{remark} \label{Remark_OmegaDeltaEmpty}
Let us note that Theorem~\ref{Theorem_CompleteExactness} is valid even in the case when the set 
$\Omega_{\delta} \setminus \Omega$ is empty. Moreover, if $\Omega_{\delta} \setminus \Omega = \emptyset$ for some
$\delta > 0$, then the penalty function $\Phi_{\lambda}$ is completely exact on $S_{\lambda}(c)$ for any 
$c > \mathcal{I}^*$, provided there exists $\lambda_0 \ge 0$ such that $\Phi_{\lambda_0}$ is bounded below on $A$.
Indeed, in this case for any $\lambda \ge \lambda_0$ and $x \notin \Omega_{\delta}$ one has
$$
  \Phi_{\lambda}(x) = \Phi_{\lambda_0}(x) + (\lambda - \lambda_0) \varphi(x)
  \ge \eta + (\lambda - \lambda_0) \delta \ge c \quad \forall \lambda \ge \lambda^* = \lambda_0 + (c - \eta) / \delta,
$$
where $\eta = \inf_{x \in A} \Phi_{\lambda_0}(x)$, which implies that $S_{\lambda}(c) \subseteq \Omega$ for any 
$\lambda \ge \lambda^*$. Hence taking into account the fact that $\Phi_{\lambda}(x) = \mathcal{I}(x)$ for any 
$x \in \Omega$ one obtains that the first two statements of Theorem~\ref{Theorem_CompleteExactness} hold true, and if
$x^* \in S_{\lambda}(c)$ is a local minimiser/inf-stationary point of $\Phi_{\lambda}$ on $A$, then $x^* \in \Omega$ and
it is a local minimiser/inf-stationary point of $\mathcal{I}$ on $\Omega$, provided $\lambda \ge \lambda^*$. On the
other hand, if $\lambda \ge \lambda^*$ and $x^* \in S_{\lambda}(c)$ is a locally optimal solution of the problem
$(\mathcal{P})$, then for any $x$ in a neighbourhood of $x^*$, either $x \in \Omega$ and 
$\Phi_{\lambda}(x) = \mathcal{I}(x) \ge \mathcal{I}(x^*) = \Phi_{\lambda}(x^*)$ or $x \notin \Omega$ and
$\Phi_{\lambda}(x) \ge c > \Phi_{\lambda}(x^*)$, i.e. $x^*$ is a locally optimal solution of the penalised problem
\eqref{PenalizedProblem}. The analogous statement for inf-stationary points is proved in a similar way.
\end{remark}

Under the assumptions of Theorem~\ref{Theorem_CompleteExactness} nothing can be said about locally optimal solutions of
the penalised problem and inf-stationary points of $\Phi_{\lambda}$ on $A$ that do not belong to the sublevel set
$S_{\lambda}(c)$. If a numerical method for minimising the penalty function $\Phi_{\lambda}$ finds a point
$x^* \notin S_{\lambda}(c)$, then this point might even be infisible for the original problem (in this case, usually,
either constraints are degenerate in some sense at $x^*$ or $\mathcal{I}$ is not Lipschitz continuous near this point).
Under more restrictive assumptions one can exclude such possibility, i.e. prove that the penalty function
$\Phi_{\lambda}$ is \textit{completely exact on} $A$, i.e. on $S_{\lambda}(c)$ with $c = + \infty$. Namely, the
following theorem holds true.\footnote{This result as well as its applications in the following sections
were inspired by a question raised by one of the reviewers of the first part of our study. The author wishes to
express his gratitude to the reviewer for raising this question.} Its proof is given in Appendix~A.

\begin{theorem} \label{THEOREM_COMPLETEEXACTNESS_GLOBAL}
Let $X$ be a complete metric space, $A$ be closed, $\mathcal{I}$ be Lipschitz continuous on $A$, and $\varphi$ be lower
semicontinuous on $A$ and continuous at every point of the set $\Omega$. Suppose also that there exists $a > 0$ such
that $\varphi^{\downarrow}_A(x) \le - a$ for all $x \in A \setminus \Omega$, and the function $\Phi_{\lambda_0}$ is
bounded below on $A$ for some $\lambda_0 \ge 0$. Then the penalty function $\Phi_{\lambda}$ is completely exact on $A$.
\end{theorem}

In some important cases it might be very difficult (if at all possible) to verify the assumptions of
Theorems~\ref{Theorem_CompleteExactness} and \ref{THEOREM_COMPLETEEXACTNESS_GLOBAL} and prove the complete exactness of
the penalty function $\Phi_{\lambda}$. In these cases one can try to check whether $\Phi_{\lambda}$ is at least
\textit{locally} exact.

\begin{definition}
Let $x^*$ be a locally optimal solution of the problem $(\mathcal{P})$. The penalty function $\Phi_{\lambda}$ is said
to be \textit{locally exact} at $x^*$, if there exists $\lambda^*(x^*) \ge 0$ such that $x^*$ is a point of local
minimum of the penalised problem \eqref{PenalizedProblem} for any $\lambda \ge \lambda^*(x^*)$.
\end{definition}

Thus, if the penalty function $\Phi_{\lambda}$ is locally exact at a locally optimal solution $x^*$, then one can ``get
rid'' of the constraint $x \in M$ in a neighbourhood of $x^*$ with the use of the penalty function $\Phi_{\lambda}$,
since by definition $x^*$ is a local minimiser of $\Phi_{\lambda}$ on $A$ for any sufficiently large $\lambda$.
The following theorem, which is a particular case of \cite[Theorem~2.4 and Proposition~2.7]{Dolgopolik_ExPen_I},
contains simple sufficient conditions for the local exactness. Let $B(x, r) = \{ y \in X \mid d(x, y) \le r \}$ for any
$x \in X$ and $r > 0$.

\begin{theorem} \label{Theorem_LocalExactness}
Let $x^*$ be a locally optimal solution of the problem $(\mathcal{P})$. Suppose also that $\mathcal{I}$ is Lipschitz
continuous near $x^*$ with Lipschitz constant $L > 0$, and there exist $r > 0$ and $a > 0$ such
that
\begin{equation} \label{PenaltyTerm_LocalErrorBound}
  \varphi(x) \ge a \dist(x, \Omega) \quad \forall x \in B(x^*, r) \cap A.
\end{equation}
Then the penalty function $\Phi_{\lambda}$ is locally exact at $x^*$ with $\lambda^*(x^*) \le L / a$.
\end{theorem}

Let us also point out a useful result \cite[Corollary~2.2]{Cominetti}) that allows one to easily verify 
inequality \eqref{PenaltyTerm_LocalErrorBound} for a large class of optimisation and optimal control problems.

\begin{theorem} \label{Theorem_LocalErrorBound}
Let $X$ and $Y$ be Banach spaces, $C \subseteq X$ and $K \subset Y$ be closed convex sets, and $F \colon X \to Y$ be a
given mapping. Suppose that $F$ is strictly differentiable at a point $x^* \in C$ such that $F(x^*) \in K$, 
$D F(x^*)$ is its Fr\'{e}chet derivative at $x^*$, and
\begin{equation} \label{MetricRegCond}
  0 \in \core\Big[ DF(x^*)(C - x^*) - (K - F(x^*)) \Big],
\end{equation}
where ``$\core$'' is the algebraic interior. Then there exist $r > 0$ and $a > 0$ such that
$$
  \dist(F(x), K) \ge a \dist( x, F^{-1}(K) \cap C) \quad \forall x \in B(x^*, r) \cap C.
$$
\end{theorem}

\begin{remark}
Let $C$, $K$, and $F$ be as in the previous theorem. Suppose that $A = C$ and $M = \{ x \in X \mid F(x) \in K \}$.
Then $\Omega = F^{-1}(K) \cap C$, and one can define $\varphi(\cdot) = \dist(F(\cdot), K)$. In this case under 
the assumptions of Theorem~\ref{Theorem_LocalErrorBound} constraint qualification \eqref{MetricRegCond} guarantees that
$\varphi(x) \ge a \dist( x, F^{-1}(K) \cap C) = a \dist( x, \Omega )$ for all $x \in B(x^*, r) \cap A$,
i.e. \eqref{PenaltyTerm_LocalErrorBound} holds true.
\end{remark}

In the linear case, the following nonlocal version of Robinson-Ursescu's theorem due to Robinson 
\cite[Theorems~1 and 2]{Robinson76} (see also~\cite{Cominetti,Ioffe}) is very helpful for verifying inequality
\eqref{PenaltyTerm_LocalErrorBound} and the exactness of penalty functions.

\begin{theorem}[Robinson] \label{Theorem_Robinson_Ursescu}
Let $X$ and $Y$ be Banach spaces, $\mathcal{T} \colon X \to Y$ be a bounded linear operator, and $C \subset X$ be a
closed convex set. Suppose that $x^* \in C$ is such that the point $y^* = \mathcal{T} x^*$ belongs to the iterior
$\interior(\mathcal{T}(C))$ of the set $\mathcal{T}(C)$. Then there exist $r > 0$ and $\kappa > 0$ such that
$$
  \dist( x, \mathcal{T}^{-1}(y) \cap C ) \le \kappa \big( 1 + \| x - x^* \| \big) \| \mathcal{T} x - y \|
  \qquad \forall x \in C \quad \forall y \in B(y^*, r).
$$
\end{theorem}

In the following sections we employ Theorems~\ref{Theorem_CompleteExactness}--\ref{Theorem_Robinson_Ursescu} to verify
complete or local exactness of penalty function for optimal control problems with terminal and state constraints.

\begin{remark}
In our exposition of the theory of exact penalty functions we mainly followed
papers\cite{Dolgopolik_ExPen_I,Dolgopolik_ExPen_II}. A completely different approach to an analysis of 
the \textit{global} exactness of exact penalty functions based on the Palais-Smale condition was developed
by Zaslavski\cite{Zaslavski}. It seems possible to apply the main results of monograph\cite{Zaslavski} to obtain
sufficient conditions for the global exactness of penalty functions for some optimal control problems that
significantly differ from the ones obtained in this article. A derivation of such conditions lies beyond the scope of
this article, and we leave it as an interesting open problem for future research.
\end{remark}

\section{Exact Penalisation of Terminal Constraints}
\label{Sect_ExactPen_TerminalConstraint}

In this section we analyse exact penalty functions for fixed-endpoint optimal control problems, including such problems
for linear evolution equations in Hilbert spaces. Our aim is to convert a fixed-endpoint problem into a free-endpoint
one by penalising the terminal constraint and obtain conditions under which the penalised free-endpoint problem is
equivalent (locally or globally) to the original one. The main results of this section allow one to apply methods
for solving free-endpoint optimal control problems to fixed-endpoint problems.

\subsection{Notation}

Let us introduce notation first. Denote by $L_q^m(0, T)$ the Cartesian product of $m$ copies of $L^q(0, T)$, and let
$W_{1, p}^d(0, T)$ be the Cartesian product of $d$ copies of the Sobolev space $W^{1, p}(0, T)$. Here 
$1 \le q, p \le + \infty$. As usual (see, e.g. \cite{Leoni}), we identify the Sobolev 
space $W^{1, p}(0, T)$ with the space consisting of all those absolutely continuous functions 
$x \colon [0, T] \to \mathbb{R}$ for which $\dot{x} \in L^p(0, T)$. The space $L_q^m(0, T)$ is equipped with the norm 
$\| u \|_q = ( \int_0^T |u(t)|^q \, dt)^{1/q}$, when $1 \le q < + \infty$ (here $| \cdot |$ is the Euclidean norm),
while the space $L_{\infty}^m (0, T)$ is equipped with the norm $\| u \|_{\infty} = \esssup_{t \in [0, T]}|u(t)|$. The
Sobolev space $W_{1, p}^d(0, T)$ is endowed with the norm $\| x \|_{1, p} = \| x \|_p + \| \dot{x} \|_p$. Let us note
that by the Sobolev imbedding theorem (see, e.g. \cite[Theorem~5.4]{Adams}) for any $p \in [1, + \infty]$
there exists $C_p > 0$ such that 
\begin{equation} \label{SobolevImbedding}
  \| x \|_{\infty} \le C_p \| x \|_{1, p} \quad \forall x \in W^d_{1, p}(0, T),
\end{equation}
which, in particular, implies that any bounded set in $W^d_{1, p}(0, T)$ is also bounded in $L_{\infty}^d(0, T)$.
In what follows we suppose that the Cartesian product $X \times Y$ of normed spaces $X$ and $Y$ is endowed with the
norm $\| (x, y) \| = \| x \|_X + \| y \|_Y$. For any $r \in [1, + \infty]$ denote by $r' \in [1, + \infty]$ the
\textit{conjugate exponent} of $r$, i.e. $1 / r + 1 / r' = 1$.

Let $g \colon \mathbb{R}^d \times \mathbb{R}^m \times [0, T] \to \mathbb{R}^k$ be a given function. We say that $g$
satisfies \textit{the growth condition} of order $(l, s)$ with $0 \le l < + \infty$ and $1 \le s \le + \infty$, if
for any $R > 0$ there exist $C_R > 0$ and an a.e. nonnegative function $\omega_R \in L^s(0, T)$ such that 
$|g(x, u, t)| \le C_R |u|^l + \omega_R(t)$ for a.e. $t \in [0, T]$ and for all 
$(x, u) \in \mathbb{R}^d \times \mathbb{R}^m$ with $|x| \le R$.

Finally, if the function $g = g(x, u, t)$ is differentiable, then the gradient of the function $x \mapsto g(x, u, t)$ is
denoted by $\nabla_x g(x, u, t)$, and a similar notation is used for the gradient of the function 
$u \mapsto g(x, u, t)$.

\subsection{Linear Time-Varying Systems}

We start our analysis with the linear case, since in this case the complete exactness of the penalty function can be
obtained without any assumptions on the controllability of the system. Consider the following fixed-endpoint optimal
control problem:
\begin{equation} \label{LinearFixedEndPointProblem}
\begin{split}
  {}&\min \: \mathcal{I}(x, u) = \int_0^T \theta(x(t), u(t), t) \, dt \\
  {}&\text{subject to } \dot{x}(t) = A(t) x(t) + B(t) u(t), \quad t \in [0, T], \quad u \in U, \quad
  x(0) = x_0, \quad x(T) = x_T.
\end{split}
\end{equation}
Here $x(t) \in \mathbb{R}^d$ is the system state at time $t$, $u(\cdot)$ is a control input, 
$\theta \colon \mathbb{R}^d \times \mathbb{R}^m \times [0, T] \to \mathbb{R}$, 
$A \colon [0, T] \to \mathbb{R}^{d \times d}$, and $B \colon [0, T] \to \mathbb{R}^{d \times m}$ are given functions, 
$T > 0$ and $x_0, x_T \in \mathbb{R}^d$ are fixed. We suppose that $x \in W^d_{1,p}(0, T)$, while the control inputs $u$
belong to a closed convex subset $U$ of the space $L^m_q(0, T)$ (here $1 \le p, q \le + \infty$). 

Let us introduce a penalty function for problem \eqref{LinearFixedEndPointProblem}. We will penalise only the terminal
constraint $x(T) = x_T$. Define $X = W_{1, p}^d(0, T) \times L_q^m(0, T)$, $M = \{ (x, u) \in X \mid x(T) = x_T \}$, and
\begin{equation} \label{AddConstr_LinearCase}
  A = \Big\{ (x, u) \in X \Bigm| x(0) = x_0, \: u \in U, \:
  \dot{x}(t) = A(t) x(t) + B(t) u(t) \text{ for a.e. } t \in [0, T] \Big\}.
\end{equation}
Then problem \eqref{LinearFixedEndPointProblem} can be rewritten as the problem of minimising $\mathcal{I}(x, u)$
subject to $(x, u) \in M \cap A$. Define $\varphi(x, u) = |x(T) - x_T|$. Then 
$M = \{ (x, u) \in X \mid \varphi(x, u) = 0 \}$, and one can consider the penalised problem of minimising the penalty
function $\Phi_{\lambda}(x, u) = \mathcal{I}(x, u) + \lambda \varphi(x, u)$ subject to $(x, u) \in A$. Note that this is
a free-endpoint problem of the form:
\begin{equation} \label{FreeEndPointProblem_withPenalty}
\begin{split}
  {}&\min_{(x, u) \in X} \Phi_{\lambda}(x, u) = \int_0^T \theta(x(t), u(t), t) \, dt + \lambda \big| x(T) - x_T \big| \\
  {}&\text{subject to } \dot{x}(t) = A(t) x(t) + B(t) u(t), \quad 
  t \in [0, T], \quad u \in U, \quad x(0) = x_0.
\end{split}
\end{equation}
Our aim is to show that under some natural assumptions the penalty function $\Phi_{\lambda}$ is completely exact, i.e.
that free-endpoint problem \eqref{FreeEndPointProblem_withPenalty} is equivalent to fixed-endpoint problem
\eqref{LinearFixedEndPointProblem} for any sufficiently large $\lambda \ge 0$. 

Let $\mathcal{I}^*$ be the optimal value of problem \eqref{LinearFixedEndPointProblem}. Recall that 
$S_{\lambda}(c) = \{ (x, u) \in A \mid \Phi_{\lambda}(x, u) < c \}$ for any $c \in \mathbb{R}$ and 
$\Omega_{\delta} = \{ (x, u) \in A \mid \varphi(x, u) < \delta \}$ for any $\delta > 0$. In our case the set
$\Omega_{\delta}$ consists of all those $(x, u) \in W_{1, p}^d(0, T) \times L_q^m(0, T)$ for which $u \in U$,
\begin{equation} \label{LinearTimeVaryingSystems}
  \dot{x}(t) = A(t) x(t) + B(t) u(t) \quad \text{for a.e. } t \in [0, T], \quad x(0) = x_0,
\end{equation}
and $|x(T) - x_T| < \delta$. Finally, denote by $\mathcal{R}(x_0, T)$ the set that is reachable in time $T$, i.e. the
set of all those $\xi \in \mathbb{R}^d$ for which there exists $u \in U$ such that $x(T) = \xi$, where $x(\cdot)$ is a
solution of \eqref{LinearTimeVaryingSystems}. Observe that the reachable set $\mathcal{R}(x_0, T)$ is convex due to the
convexity of the set $U$ and the linearity of the system. Finally, recall that the \textit{relative interior} of a
convex set $C \subset \mathbb{R}^d$, denoted $\relint C$, is the interior of $C$ relative to the affine hull of $C$.

The following theorem on the complete exactness of the penalty function $\Phi_{\lambda}$ for
problem \eqref{LinearFixedEndPointProblem} can be proved with the use of state-transition matrix for
\eqref{LinearTimeVaryingSystems}. Here, we present a different and more instructive (although slightly longer) proof of
this result, since it contains several important ideas related to penalty functions for optimal control problems, which
will be utilised in the following sections.

\begin{theorem} \label{Theorem_FixedEndPointProblem_Linear}
Let $q \ge p$, and the following assumptions be valid:
\begin{enumerate}
\item{$A(\cdot) \in L_{\infty}^{d \times d}(0, T)$ and $B(\cdot) \in L_{\infty}^{d \times m}(0, T)$;
\label{Assumpt_LTI_BoundedCoef}}

\item{the function $\theta = \theta(x, u, t)$ is continuous, differentiable in $x$ and $u$, and the functions
$\nabla_x \theta$ and $\nabla_u \theta$ are continuous;
}

\item{either $q = + \infty$ or the functions $\theta$ and $\nabla_x \theta$ satisfy the growth condition of order 
$(q, 1)$, while the function $\nabla_u \theta$ satisfies the growth condition of order $(q - 1, q')$;
\label{Assumpt_LTI_DerivGrowthCond}}

\item{there exists a globally optimal solution of problem \eqref{LinearFixedEndPointProblem}, and $x_T$ belongs to 
the relative interior of the reachable set $\mathcal{R}(x_0, T)$ (in the case $U = L_q^m(0, T)$ this assumption holds
true automatically);
\label{Assumpt_LTI_EndpointRelInt}
}

\item{there exist $\lambda_0 > 0$, $c > \mathcal{I}^*$ and $\delta > 0$ such that the set 
$S_{\lambda_0}(c) \cap \Omega_{\delta}$ is bounded in $W^d_{1, p}(0, T) \times L_q^m(0, T)$, and the function
$\Phi_{\lambda_0}(x, u)$ is bounded below on $A$.
\label{Assumpt_LTI_SublevelBounded}}
\end{enumerate}
Then there exists $\lambda^* \ge 0$ such that for any $\lambda \ge \lambda^*$ the penalty function $\Phi_{\lambda}$ for
problem \eqref{LinearFixedEndPointProblem} is completely exact on $S_{\lambda}(c)$.
\end{theorem}

\begin{proof}
Our aim is to employ Theorem~\ref{Theorem_CompleteExactness}. To this end, note that from the essential boundedness of
$A(\cdot)$ and $B(\cdot)$, and the fact that $p \le q$ it follows that the function
$(x, u) \mapsto \dot{x}(\cdot) - A(\cdot) x(\cdot) - B(\cdot) u(\cdot)$ continuously maps $X$ to $L^d_p(0, T)$. Hence
taking into account \eqref{SobolevImbedding} and the fact that $U$ is closed by our assumptions one obtains that the set
$A$ is closed (see \eqref{AddConstr_LinearCase}). By applying \eqref{SobolevImbedding} one gets that
$$
  \big| \varphi(x, u) - \varphi(y, v) \big| = \big| |x(T) - x_T| - |y(T) - x_T| \big| 
  \le |x(T) - y(T)| \le C_p \| x - y \|_{1, p}
  \quad \forall (x, u), (y, v) \in X,
$$
i.e. the function $\varphi$ is continuous. By \cite[Theorem~7.3]{FonsecaLeoni} the growth condition on the function
$\theta$ guarantees that the functional $\mathcal{I}(x, u)$ is correctly defined and finite for any
$(x, u) \in X$, while by \cite[Proposition~4]{DolgopolikFominyh} the growth conditions on $\nabla_x \theta$ and
$\nabla_u \theta$ ensure that the functional $\mathcal{I}(x, u)$ is Lipschitz continuous on any bounded subset of $X$.
Hence, in particular, it is Lipschitz continuous on any bounded open set containing the set 
$S_{\lambda_0}(c) \cap \Omega_{\delta}$ (such \textit{bounded} open set exists, since 
$S_{\lambda_0}(c) \cap \Omega_{\delta}$ is bounded by our assumption). Thus, by Theorem~\ref{Theorem_CompleteExactness}
it remains to check that there exists $a > 0$ such that
$\varphi^{\downarrow}_A(x, u) \le - a$ for any $(x, u) \in S_{\lambda_0}(c) \cap (\Omega_{\delta} \setminus \Omega)$.

Let $(x, u) \in S_{\lambda_0}(c) \cap \Omega_{\delta}$ be such that $\varphi(x, u) > 0$, i.e. $x(T) \ne x_T$. Choose any
$(\widehat{x}, \widehat{u}) \in \Omega = M \cap A$ (recall that $\Omega$ is not empty, since by our assumption 
problem \eqref{LinearFixedEndPointProblem} has a globally optimal solution). By definition $\widehat{x}(T) = x_T$.
Put $\Delta x = ( \widehat{x} - x ) / \sigma$ and $\Delta u = ( \widehat{u} - u ) / \sigma$, where
$\sigma = \| \widehat{x} - x \|_{1, p} + \| \widehat{u} - u \|_q > 0$. 
Then $\| (\Delta x, \Delta u) \|_X = \| \Delta x \|_{1, p} + \| \Delta u \|_q = 1$. From the linearity of the system and
the convexity of the set $U$ it follows that for any $\alpha \in [0, \sigma]$ one has 
$(x + \alpha \Delta x, u + \alpha \Delta x) \in A$. Furthermore, note
that $(x + \alpha \Delta x)(T) = x(T) + \alpha \sigma^{-1} (x_T - x(T))$. Hence
\begin{align*}
  \varphi^{\downarrow}_A(x, u) &\le \lim_{\alpha \to +0} 
  \frac{\varphi(x + \alpha \Delta x, u + \alpha \Delta u) - \varphi(x, u)}{\alpha \| (\Delta x, \Delta u) \|_X} \\
  &= \lim_{\alpha \to +0} \frac{(1 - \alpha \sigma^{-1}) |x(T) - x_T| - |x(T) - x_T|}{\alpha}
  = - \frac{1}{\sigma} |x(T) - x_T|.
\end{align*}
Therefore, it remains to check that there exists $C > 0$ such that for any 
$(x, u) \in S_{\lambda_0}(c) \cap (\Omega_{\delta} \setminus \Omega)$ one can find 
$(\widehat{x}, \widehat{u}) \in \Omega$ satisfying the inequality
\begin{equation} \label{ErrorBound_TerminalConstraint}
  \| x - \widehat{x} \|_{1, p} + \| u - \widehat{u} \|_q \le C |x(T) - x_T|.
\end{equation}
Then $\varphi^{\downarrow}_A(x, u) \le - 1 / C$ for any 
$(x, u) \in S_{\lambda_0}(c) \cap (\Omega_{\delta} \setminus \Omega)$, and the proof is complete.

Firstly, let us check that \eqref{ErrorBound_TerminalConstraint} follows from a seemingly weaker inequality, which is
easier to prove. Let $(x_1, u_1) \in A$ and $(x_2, u_2) \in A$. Then for any $t \in [0, T]$ one has
$x_1(t) - x_2(t) = \int_0^t ( A(\tau) (x_1(\tau) - x_2(\tau)) + B(\tau) (u_1(\tau) - u_2(\tau)) ) d \tau$.
By applying H\"{o}lder's inequality one gets that for any $t \in [0, T]$
$$
  |x_1(t) - x_2(t)| \le \| B(\cdot) \|_{\infty} T^{1/q'} \| u_1 - u_2 \|_q 
  + \| A(\cdot) \|_{\infty} \int_0^t |x_1(\tau) - x_2(\tau)| \, d \tau.
$$
Hence by the Gr\"{o}nwall-Bellman inequality one obtains that 
$\| x_1 - x_2 \|_{\infty} \le L_0 \| u_1 - u_2 \|_q$ for all $(x_1, u_1) \in A$ and $(x_2, u_2) \in A$, where
$L_0 = \| B(\cdot) \|_{\infty} T^{1/q'} ( 1 + T \| A(\cdot) \|_{\infty} e^{T \| A(\cdot) \|_{\infty}} )$.
Consequently, by applying the equality
$$
  \dot{x}_1(t) - \dot{x}_2(t) = A(t) \Big( x_1(t) - x_2(t) \big) + B(t) \big( u_1(t) - u_2(t) \Big),
$$
H\"{o}lder's inequality, and the fact that $q \ge p$ one obtains
$$
  \| \dot{x}_1 - \dot{x}_2 \|_p \le T^{1/p} \| A(\cdot) \|_{\infty} \| x_1 - x_2 \|_{\infty}
  + \| B(\cdot) \|_{\infty} T^{\frac{q - p}{qp}} \| u_1 - u_2 \|_q
  \le \Big( T^{1/p} \| A(\cdot) \|_{\infty} L_0 + T^{\frac{q - p}{qp}} B(\cdot) \|_{\infty} \Big) \| u_1 - u_2 \|_q,
$$
i.e. $\| x_1 - x_2 \|_{1, p} \le L \| u_1 - u_2 \|_q$ for some $L > 0$ depending only on $A(\cdot)$, $B(\cdot)$, $T$,
$p$, and $q$. Therefore, it is sufficient to check that there exists $C > 0$ such that for any 
$(x, u) \in S_{\lambda_0}(c) \cap (\Omega_{\delta} \setminus \Omega)$ one can find 
$(\widehat{x}, \widehat{u}) \in \Omega$ satisfying the inequality 
\begin{equation} \label{LinearSystem_SensitivityCond}
  \| u - \widehat{u} \|_q \le C |x(T) - x_T|
\end{equation}
(cf.~\eqref{ErrorBound_TerminalConstraint}). Let us prove inequality \eqref{LinearSystem_SensitivityCond} with the use
of Robinson's theorem (Theorem~\ref{Theorem_Robinson_Ursescu}).

Introduce the linear operator $\mathcal{T} \colon L_q^m(0, T) \to \mathbb{R}^d$, $\mathcal{T} v = h(T)$, where 
$h \in W_{1, p}^d(0, T)$ is a solution of 
\begin{equation} \label{TimeVaryingSystem_FromZero}
  \dot{h}(t) = A(t) h(t) + B(t) v(t), \quad h(0) = 0.
\end{equation}
For any $v \in L_q^m(0, T)$ a unique absolutely continuous solution $h$ of this equation defined on $[0, T]$ exists by
\cite[Theorem~1.1.3]{Filippov}. By applying H\"{o}lder's inequality and the fact that $q \ge p$ one gets
that $\| \dot{h} \|_p \le T^{1/p} \| A(\cdot) \|_{\infty} \| h \|_{\infty} + \| B(\cdot) \|_{\infty} T^{(q - p)/qp}
\| v \|_q$, which implies that $h \in W^d_{1, p}(0, T)$, and the linear operator $\mathcal{T}$ is correctly defined. Let
us check that it is bounded. Indeed, fix any $v \in L_q^m(0, T)$ and the corresponding solution $h$ of
\eqref{TimeVaryingSystem_FromZero}. For all $t \in [0, T]$ one has
$$
  |h(t)| = \left| \int_0^t \big( A(\tau) h(\tau) + B(\tau) v(\tau) \big) \, d \tau \right|
  \le \| B(\cdot) \|_{\infty} T^{1/q'} \| v \|_q + \| A(\cdot) \| \int_0^t |h(\tau)| \, d \tau,
$$
which with the use of the Gr\"{o}nwall-Bellman inequality implies that $|h(T)| \le L_0 \| v \|_q$, i.e. the operator
$\mathcal{T}$ is bounded.

Fix any feasible point $(x_*, u_*) \in \Omega$ of problem~\eqref{LinearFixedEndPointProblem}, i.e. 
$\dot{x}_*(t) = A(t) x_*(t) + B(t) u_*(t)$ for a.e. $t \in [0, T]$, $u_* \in U$, $x(0) = x_0$, and $x(T) = x_T$.
Observe that for any $(x, u) \in A$ one has $x(0) - x_*(0) = 0$ and
$\dot{x}(t) - \dot{x}_*(t) = A(t) \big( x(t) - x_*(t) \big) + B(t) \big( u(t) - u_*(t) \big)$ 
for a.e. $t \in [0, T]$, which implies that $x(T) = (x(T) - x_T) + x_T = \mathcal{T}(u - u_*) + x_T$ 
(see \eqref{AddConstr_LinearCase} and \eqref{TimeVaryingSystem_FromZero}). Consequently, one has
\begin{equation} \label{RechableSet_AsShiftedImage}
  \mathcal{R}(x_0, T) = x_T + \mathcal{T}(U - u_*).
\end{equation}
Define $X_0 = \cl \linhull(U - u_*)$ and $Y_0 = \linhull \mathcal{T}(U - u_*)$. Note that $Y_0$ is closed as a subspace
of the finite dimensional space $\mathbb{R}^d$. Moreover, $\mathcal{T}(X_0) = Y_0$. Indeed,
it is clear that the operator $\mathcal{T}$ maps $\linhull(U - u_*)$ onto $\linhull \mathcal{T}(U - u_*)$. If 
$u \in X_0$, then there exists a sequence $\{ u_n \} \subset \linhull(U - u_*)$ converging to $u$. From the boundedness
of the operator $\mathcal{T}$ it follows that $\mathcal{T}(u_n) \to \mathcal{T}(u)$ as $n \to \infty$, which implies
that $\mathcal{T}(u) \in Y_0$ due to the closedness of $Y_0$ and the fact that $\{ \mathcal{T}(u_n) \} \subset Y_0$ by
definition. Thus, $\mathcal{T}(X_0) = Y_0$.

Finally, introduce the operator $\mathcal{T}_0 \colon X_0 \to Y_0$, $\mathcal{T}_0(u) = \mathcal{T}(u)$ for all 
$u \in X_0$. Clearly, $\mathcal{T}_0$ is a bounded linear operator between Banach spaces. Recall that by our assumption
$x_T \in \relint \mathcal{R}(x_0, T)$. By the definition of relative
interior it means that $0 \in \interior \mathcal{T}_0(U - u_*)$ (see~\eqref{RechableSet_AsShiftedImage}). Therefore by
Robinson's theorem (Theorem~\ref{Theorem_Robinson_Ursescu} with $C = U - u_*$, $x^* = 0$, and $y = 0$) there
exists $\kappa > 0$
\begin{equation} \label{Robinson_Ursescu_LTVS}
  \dist\big( u - u_*, \mathcal{T}_0^{-1}(0) \cap (U - u_*) \big) \le
  \kappa \big( 1 + \| u - u_* \|_q \big) \big| \mathcal{T}_0(u - u_*) \big|
  \quad \forall u \in U.
\end{equation}
With the use of this inequality we can easily prove \eqref{LinearSystem_SensitivityCond}. Indeed, fix any 
$(x, u) \in S_{\lambda_0}(c) \cap (\Omega_{\delta} \setminus \Omega)$. Note that 
$\mathcal{T}_0(u - u_*) = x(T) - x_T \ne 0$, since $(x, u) \notin \Omega$. By inequality \eqref{Robinson_Ursescu_LTVS}
there exists $v \in U - u_*$ such that $\mathcal{T}_0(v) = 0$ and
\begin{equation} \label{Robinson_Ursescu_LTVS_mod}
  \big\| u - u_* - v \big\|_q \le 2 \kappa \big( 1 + \| u - u_* \|_q \big) |x(T) - x_T|.
\end{equation}
Define $\widehat{u} = u_* + v$, and let $\widehat{x}$ be the corresponding solution of original system
\eqref{LinearTimeVaryingSystems}. Then $\widehat{x}(T) = x_T$, since $(x_*, u_*) \in \Omega$ by definition and
$\mathcal{T}(v) = 0$, which yields $(\widehat{x}, \widehat{u}) \in \Omega$. Furthermore, by
inequality~\eqref{Robinson_Ursescu_LTVS_mod} one has
$$
   \| u - \widehat{u} \|_q \le 2 \kappa \big( 1 + \| u - u_* \|_q \big) |x(T) - \widehat{x}(T)|.
$$
By our assumption the set $S_{\lambda_0}(c) \cap (\Omega_{\delta} \setminus \Omega)$ is bounded, which implies that
there exists $C > 0$ such that $2 \kappa (1 + \| u - u_* \|_q) \le C$ for all 
$(x, u) \in S_{\lambda_0}(c) \cap (\Omega_{\delta} \setminus \Omega)$. Thus, for all such $(x, u)$ there exists
$(\widehat{x}, \widehat{u}) \in \Omega$ satisfying the inequality 
$\| u - \widehat{u} \|_q \le C |x(T) - \widehat{x}(T)|$, i.e. \eqref{LinearSystem_SensitivityCond} holds true,
and the proof is complete.
\end{proof}

\begin{remark}
Let $1 < p \le q < + \infty$, and the function $\theta(x, u, t)$ be convex in $u$ for all $x \in \mathbb{R}^d$ and
$t \in [0, T]$. Then under assumptions \ref{Assumpt_LTI_BoundedCoef}--\ref{Assumpt_LTI_DerivGrowthCond} and
\ref{Assumpt_LTI_SublevelBounded} of Theorem~\ref{Theorem_FixedEndPointProblem_Linear} a globally optimal solution of
problem \eqref{LinearFixedEndPointProblem} exists iff $x_T \in \mathcal{R}(x_0, T)$. Indeed, if 
$x_T \in \mathcal{R}(x_0, T)$, then the sublevel set 
$\{ x \in \Omega \mid \mathcal{I}(x, u) < c \} \subset S_{\lambda_0}(c) \cap \Omega_{\delta}$ is nonempty and bounded
due to the fact that $c > \mathcal{I}^*$. Therefore there exists a bounded sequence $\{ (x_n, u_n) \} \subset \Omega$
such that $\mathcal{I}(x_n, u_n) \to \mathcal{I}^*$ as $n \to \infty$. From the fact that the spaces
$W^d_{1, p}(0, T)$ and $L^m_q(0, T)$ are reflexive, provided $1 < p, q < + \infty$, it follows that there exists a
subsequence $\{ (x_{n_k}, u_{n_k}) \}$ weakly converging to some $(x^*, u^*) \in X$. Since the imbedding of 
$W^d_{1, p}(0, T)$ into $(C[0, T])^d$ is compact (see, e.g. \cite[Theorem~6.2]{Adams}), without loss of
generality one can suppose that $x_{n_k}$ converges to $x^*$ uniformly on $[0, T]$. Utilising this result, as well as
the facts that the system is linear and the set $U$ of admissible control inputs is convex and closed, one can readily
verify that $(x^*, u^*) \in \Omega$. Furthermore, the convexity of the function $u \mapsto \theta(x, u, t)$ ensures that
$\mathcal{I}(x^*, u^*) \le \lim_{k \to \infty} \mathcal{I}(x_{n_k}, u_{n_k}) = \mathcal{I}^*$
(see~\cite[Section~7.3.2]{FonsecaLeoni} and \cite{Ioffe77}), which implies that $(x^*, u^*)$
is a globally optimal solution of \eqref{LinearFixedEndPointProblem}.
\end{remark}

\begin{remark}
Let us note that assumption~\ref{Assumpt_LTI_SublevelBounded} of Theorem~\ref{Theorem_FixedEndPointProblem_Linear} is
satisfied, in particular, if the set $U$ is bounded or there exist $C > 0$ and $\omega \in L^1(0, T)$ such that 
$\theta(x, u, t) \ge C |u|^q + \omega(t)$ for all $x \in \mathbb{R}^d$, $u \in \mathbb{R}^m$, and a.e. $t \in (0, T)$.
Indeed, with the use of this inequality one can check that for any $c > \mathcal{I}^*$ there exists $K > 0$ such that
for all $(x, u) \in S_{\lambda}(c)$ one has $\| u \|_q \le K$ (if $U$ is bounded, then this inequality is satisfied by
definition). Then by applying the Gr\"{o}nwall-Bellman inequality one can easily check that the set $S_{\lambda}(c)$ is
bounded for all $c > \mathcal{I}^*$, provided $q \ge p$. Moreover, with the use of the boundedness of the set
$S_{\lambda}(c)$ and the growth condition of order $(q, 1)$ on the function $\theta$ one can easily check that the
penalty function $\Phi_{\lambda_0}$ is bounded below on $A$ for all $\lambda_0 \ge 0$.
\end{remark}

The following example demonstrates that in the general case Theorem~\ref{Theorem_FixedEndPointProblem_Linear} is no
longer true, if the assumption that $x_T$ belongs to the relative interior of the reachable set $\mathcal{R}(x_0, T)$
is dropped.

\begin{example} \label{Example_EndPoint_NotRelInt}
Let $d = m = 2$, $p = q = 2$, and $T = 1$. Define 
$U = \{ u \in L_2^2(0, 1) \mid u(t) \in Q \text{ for a.e. } t \in (0, 1) \}$, where
$Q = \{ u = (u^1, u^2)^T \in \mathbb{R}^2 \mid u^1 + u^2 \le 1, \: (u^1 - u^2)^2 \le u^1 + u^2 \}$. Note that the set
$U$ of admissible control inputs is closed and convex, since, as is easy to see, $Q$ is a closed convex set. Consider
the
following optimal control problem:
\begin{equation} \label{Problem_NoRint_Endpoint}
  \min \: \mathcal{I}(x, u) = \int_0^1 \big( u^2(t) - u^1(t) \big) \, dt \quad
  \text{s.t.} \quad 
  \begin{cases}
  \dot{x}^1 = 0 \\
  \dot{x}^2 = u^1 + u^2
  \end{cases} 
  \quad t \in [0, 1], \quad u \in U, \quad  x(0) = x(1) = 0.
\end{equation}
Let us show at first that in this case $\mathcal{R}(x_0, T) = \{ x \in \mathbb{R}^2 \mid x^1 = 0, \: x^2 \in [0, 1] \}$
(note that $x_0 = 0$ and $T = 1$), which implies that 
$x_T = 0 \notin \relint \mathcal{R}(x_0, T) = \{ x \in \mathbb{R}^2 \mid x^1 = 0, \: x^2 \in (0, 1) \}$. Indeed, by the
definitions of the sets $U$ and $Q$ for any $u \in U$ one has
$$
  x^2(1) = \int_0^1 (u^1(t) + u^2(t)) \, dt \le \int_0^1 dt = 1, \quad
  x^2(1) = \int_0^1 (u^1(t) + u^2(t)) dt \ge \int_0^1 (u^1(t) - u^2(t))^2 \, dt \ge 0,
$$
i.e. $x^2(1) \in [0, 1]$. Furthermore, for any $s \in [0, 1]$ one has $x^2(1) = s$ for
$u_s^1(t) \equiv (s + \sqrt{s}) / 2$ and $u_s^2(t) \equiv (s - \sqrt{s}) / 2$ (note that $u_s \in U$). Thus, 
$\mathcal{R}(x_0, T) = \{ 0 \} \times [0, 1]$, and $x_T \notin \relint \mathcal{R}(x_0, T)$. Note that all other
assumptions of Theorem~\ref{Theorem_FixedEndPointProblem_Linear} are satisfied. 

Let us check that the penalty function $\Phi_{\lambda}(x, u) = \mathcal{I}(x, u) + \lambda |x(1)|$ for problem
\eqref{Problem_NoRint_Endpoint} is not globally exact. Firstly, note that the only feasible point of problem
\eqref{Problem_NoRint_Endpoint} is $(x^*, u^*)$ with $x^*(t) \equiv 0$ and $u^*(t) = 0$ for a.e. $t \in [0, 1]$.
Indeed, fix any feasible point $(x, u) \in \Omega$. From the terminal constraint $x(1) = 0$ and the definition of $Q$ it
follows that
$$
  0 = x^2(1) = \int_0^1 (u^1(t) + u^2(t)) \, dt \ge \int_0^1 (u^1(t) - u^2(t))^2 \, dt \ge 0,
$$
which implies that $u^1(t) = u^2(t)$ for a.e. $t \in [0, 1]$. Furthermore, by the definition of $Q$ one has 
$u^1(t) + u^2(t) \ge 0$ for a.e. $t \in (0, T)$, which yields $u^1(t) = - u^2(t)$ for a.e. $t \in [0, 1]$. Therefore
$u(t) = 0$ for a.e. $t \in [0, 1]$, $x(t) \equiv 0$, and $\Omega = \{ (x^*, u^*) \}$.

Arguing by reductio ad absurdum, suppose that the penalty function 
$\Phi_{\lambda}(x, u) = \mathcal{I}(x, u) + \lambda |x(1)|$ for problem \eqref{Problem_NoRint_Endpoint} is globally
exact. Then there exists $\lambda > 0$ such that $(x^*, u^*)$ is a globally optimal solution of the problem
$$
  \min \: \Phi_{\lambda}(x, u) = \int_0^1 \big( u^2(t) - u^1(t) \big) \, dt + \lambda |x(1)| \quad
  \text{s.t.} \quad 
  \begin{cases}
  \dot{x}^1 = 0 \\
  \dot{x}^2 = u^1 + u^2
  \end{cases}  
  \quad t \in [0, 1], \quad u \in U, \quad x(0) = 0.
$$
Fix any $s \in (0, 1)$, and define $u = u_s \in U$ (recall that $u_s^1(t) \equiv (s + \sqrt{s}) / 2$ and 
$u_s^2(t) \equiv (s - \sqrt{s}) / 2$). For the corresponding solution $x_s(t)$ one has $x_s^2(1) = s$, and
$$
  \Phi_{\lambda}(x_s, u_s) = - \sqrt{s} + \lambda s \ge 0 = \Phi_{\lambda}(x^*, u^*),
$$
which is impossible for any sufficiently small $s \in (0, 1)$. Thus, the penalty function $\Phi_{\lambda}$ is not
globally exact.
\end{example}

\begin{remark}
It should be noted that the assumption $x_T \in \relint \mathcal{R}(x_0, T)$ is \textit{not} necessary for the
exactness of the penalty function $\Phi_{\lambda}$ for problem \eqref{LinearFixedEndPointProblem}. For instance, 
the interested reader can check that if in Example~\ref{Example_EndPoint_NotRelInt} the system has the form 
$\dot{x}^1 = u^1$ and $\dot{x}^2 = u^2$, then the penalty function 
$\Phi_{\lambda}(x, u) = \int_0^1 (u^2(t) - u^1(t)) \, dt + \lambda |x(1)| = x^2(1) - x^1(1) + \lambda |x(1)|$
is completely exact, despite the fact that in this case $\mathcal{R}(x_0, T) = Q$ and 
$x_T = 0 \notin \relint \mathcal{R}(x_0, T)$. We pose an interesting open problem to find \textit{necessary and
sufficient} conditions for the complete exactness of the penalty function $\Phi_{\lambda}$ for problem
\eqref{LinearFixedEndPointProblem} (at least in the time-invariant case). In particular, it seems that in the case when
$U = \{ u \in L_q^m(0, T) \mid u(t) \in Q \text{ for a.e. } t \in [0, T] \}$ and $Q$ is a convex polytope, 
the assumption $x_T \in \relint \mathcal{R}(x_0, T)$ in Theorem~\ref{Theorem_FixedEndPointProblem_Linear} can be
dropped.
\end{remark}

Let us finally note that in the case when the set $U$ of admissible control inputs is bounded, one can prove the
complete exactness of the penalty function $\Phi_{\lambda}$ for problem \eqref{LinearFixedEndPointProblem} on $A$.
In other words, one can prove that free-endpoint problem \eqref{FreeEndPointProblem_withPenalty} is completely
equivalent to fixed-endpoint problem \eqref{LinearFixedEndPointProblem} in the sense that these problems have the same
optimal value, the same globally/locally optimal solutions, and the same inf-stationary points.

\begin{theorem} \label{Theorem_FixedEndPointProblem_Linear_Global}
Let $q \ge p$, assumptions \ref{Assumpt_LTI_BoundedCoef}--\ref{Assumpt_LTI_EndpointRelInt} of
Theorem~\ref{Theorem_FixedEndPointProblem_Linear} be valid, and suppose that the set $U$ is bounded. Then the penalty
function $\Phi_{\lambda}$ for problem \eqref{LinearFixedEndPointProblem} is completely exact on $A$.
\end{theorem}

\begin{proof}
By our assumption there exists $K > 0$ such that $\| u \|_q \le K$ for any $u \in U$. Choose any $(x, u) \in A$. Then
by definition $x(t) = x_0 + \int_0^t ( A(\tau) x(\tau) + B(\tau) u(\tau) ) \, d \tau$ for all $t \in [0, T]$, which
by H\"{o}lder's inequality implies that
$$
  |x(t)| \le |x_0| + \| B(\cdot) \|_{\infty} T^{1/q'} \| u \|_q
  + \| A(\cdot) \|_{\infty} \int_0^t |x(\tau)| \, d \tau.
$$
Consequently, by applying the Gr\"{o}nwall-Bellman inequality and the fact that $\| u \|_q \le K$ one obtains that
$\| x \|_{\infty} \le C$ for some $C > 0$ depending only on $K$, $A(\cdot)$, $B(\cdot)$, $T$, and $q$. Hence by
H\"{o}lder's inequality and the definition of the set $A$ (see~\eqref{AddConstr_LinearCase}) one obtains that
$$
  \| \dot{x} \|_p =  \big\| A(\cdot) x(\cdot) + B(\cdot) u(\cdot) \big\|_p
  \le T^{1/p} \| A(\cdot) \|_{\infty} C + \| B(\cdot) \|_{\infty} T^{\frac{q - p}{qp}} K
  \quad \forall (x, u) \in A,
$$
i.e. the set $A$ is bounded in $X$ and in $L^d_{\infty}(0, T) \times L^m_q(0, T)$. Therefore, both $\mathcal{I}$ and
$\Phi_{\lambda}$, for any $\lambda \ge 0$, are bounded below on $A$ due to the fact that the function $\theta$
satisfies the growth conditions of order $(q, 1)$ (see assumption~\ref{Assumpt_LTI_DerivGrowthCond} of
Theorem~\ref{Theorem_FixedEndPointProblem_Linear}). Now, arguing in the same way as in the proof of
Theorem~\ref{Theorem_FixedEndPointProblem_Linear}, but replacing $S_{\lambda_0}(c) \cap \Omega_{\delta}$ with $A$ and
utilising Theorem~\ref{THEOREM_COMPLETEEXACTNESS_GLOBAL} instead of Theorem~\ref{Theorem_CompleteExactness}, one obtains
the desired result.
\end{proof}

\subsection{Linear Evolution Equations}
\label{SubSec_EvolEq_TerminalConstr}

Let us demonstrate that Theorems~\ref{Theorem_FixedEndPointProblem_Linear} and
\ref{Theorem_FixedEndPointProblem_Linear_Global} can be easily extended to the case of optimal control problems
for linear evolution equations in Hilbert spaces. In this section we use standard definitions and results on control
problems for infinite dimensional systems that can be found, e.g. in monograph\cite{TucsnakWeiss}.

Let $\mathscr{H}$ and $\mathscr{U}$ be complex Hilbert spaces, $\mathbb{T}$ be a strongly continuous semigroup on
$\mathscr{H}$ with generator $\mathcal{A} \colon \mathcal{D}(\mathcal{A}) \to \mathscr{H}$, and let $\mathcal{B}$ be an
admissible control operator for $\mathbb{T}$ (see~\cite[Def.~4.2.1]{TucsnakWeiss}). Consider the following
fixed-endpoint optimal control problem:
\begin{equation} \label{EvolEqFixedEndPointProblem}
\begin{split}
  {}&\min_{(x, u)} \mathcal{I}(x, u) = \int_0^T \theta(x(t), u(t), t) \, dt \\
  {}&\text{subject to } \dot{x}(t) = \mathcal{A} x(t) + \mathcal{B} u(t), \quad t \in [0, T], \quad u \in U, \quad
  x(0) = x_0, \quad x(T) = x_T.
\end{split}
\end{equation}
Here $\theta \colon \mathscr{H} \times \mathscr{U} \times [0, T] \to \mathbb{R}$ is a given function, $T > 0$ and 
$x_0, x_T \in \mathscr{H}$ are fixed, and $U$ is a closed convex subset of the space $L^2((0, T); \mathscr{U})$
consisting of all those measurable functions $u \colon (0, T) \to \mathscr{U}$
for which $\| u \|_{L^2((0, T); \mathscr{U})} = \int_0^T \| u(t) \|_{\mathscr{U}}^2 dt < + \infty$.

Let us introduce a penalty function for problem \eqref{EvolEqFixedEndPointProblem}. As in the previous section, we only
penalise the terminal constraint $x(T) = x_T$. For any $t \ge 0$ let
$F_t u = \int_0^t \mathbb{T}_{t - \sigma} \mathcal{B} u(\sigma) \, d \sigma$ be the input map corresponding to
$(\mathcal{A}, \mathcal{B})$. By \cite[Proposition~4.2.2]{TucsnakWeiss}, $F_t$ is a bounded linear operator from
$L^2((0, T); \mathscr{U})$ to $\mathscr{H}$. Furthermore, by applying \cite[Proposition~4.2.5]{TucsnakWeiss} one obtains
that for any $u \in L^2((0, T); \mathscr{U})$ the initial value problem
\begin{equation} \label{LinearEvolEq}
  \dot{x}(t) = \mathcal{A} x(t) + \mathcal{B} u(t), \quad x(0) = x_0
\end{equation}
has a unique solution $x \in C([0, T]; \mathscr{H})$ given by
\begin{equation} \label{SolutionViaSemiGroup}
  x(t) = \mathbb{T}_t x_0 + F_t u \quad \forall t \in [0, T].
\end{equation}
Define $X = C([0, T]; \mathscr{H}) \times L^2((0, T); \mathscr{U})$, $M = \{ (x, u) \in X \mid x(T) = x_T \}$, and
$$
  A = \Big\{ (x, u) \in X \Bigm| x(0) = x_0, \: u \in U, \: \text{and $\eqref{SolutionViaSemiGroup}$ holds true} \Big\}.
$$
Then problem \eqref{EvolEqFixedEndPointProblem} can be rewritten as the problem of minimising $\mathcal{I}(x, u)$
subject to $(x, u) \in M \cap A$. Introduce the penalty term $\varphi(x, u) = \| x(T) - x_T \|_{\mathscr{H}}$. Then 
$M = \{ (x, u) \in X \mid \varphi(x, u) = 0 \}$, and one can consider the penalised problem of minimising the penalty
function $\Phi_{\lambda}(x, u) = \mathcal{I}(x, u) + \lambda \varphi(x, u)$ subject to $(x, u) \in A$, which is a
free-endpoint problem of the form:
\begin{equation} \label{EvolEqFreeEndPointProblem}
\begin{split}
  {}&\min_{(x, u)} \Phi_{\lambda}(x, u) = \mathcal{I}(x, u) + \lambda \varphi(x, u) 
  = \int_0^T \theta(x(t), u(t), t) \, dt + \lambda \| x(T) - x_T \|_{\mathscr{H}} \\
  {}&\text{subject to } \dot{x}(t) = \mathcal{A} x(t) + \mathcal{B} u(t), \quad t \in [0, T], \quad u \in U, \quad
  x(0) = x_0.
\end{split}
\end{equation}
Denote by $\mathcal{R}(x_0, T)$ the set that is reachable in time $T$, i.e. the set of all those $\xi \in \mathscr{H}$
for which there exists $u \in U$ such that $x(T) = \xi$, where $x(\cdot)$ is defined in \eqref{SolutionViaSemiGroup}.
Observe that by definition $\mathcal{R}(x_0, T) = F_T(U) + \mathbb{T}_T x_0$, which implies that the reachable set
$\mathcal{R}(x_0, T)$ is convex due to the convexity of the set $U$.

Our aim is to show that under a natural assumption on the reachable set $\mathcal{R}(x_0, T)$ the penalty function
$\Phi_{\lambda}$ is completely exact, i.e. for any sufficiently large $\lambda \ge 0$ free-endpoint problem
\eqref{EvolEqFreeEndPointProblem} is equivalent to fixed-endpoint problem \eqref{EvolEqFixedEndPointProblem}.

In this finite dimensional case we assumed that $x_T \in \relint \mathcal{R}(x_0, T)$. In the infinite dimensional
case we will use the same assumption, since to the best of author's knowledge it is the weakest assumption allowing one
to utilise Robinson's theorem. However, recall that the relative interior of a convex subset of a finite
dimensional space is always nonempty, but this statement is no longer true in infinite dimensional spaces
(see~\cite{BorweinLewis92,BorwinGoebel03}). Thus, in the finite dimensional case the condition 
$x_T \in \relint \mathcal{R}(x_0, T)$ simply restricts the location of $x_T$ in the reachable set, while in 
the infinite dimensional case it also imposes the assumption ($\relint \mathcal{R}(x_0, T) \ne \emptyset$) on the
reachable set itself. For the sake of completeness recall that the relative interior of a convex subset $C$ of a Banach
space $Y$, denoted $\relint C$, is the interior of $C$ relative to the \textit{closed} affine hull of $C$.

\begin{theorem} \label{Theorem_Exactness_EvolutionEquations}
Let the following assumptions be valid:
\begin{enumerate}
\item{$\theta$ is continuous, and for any $R > 0$ there exist $C_R > 0$ and an a.e. nonnegative function
$\omega_R \in L^1(0, T)$ such that $| \theta(x, u, t) | \le C_R \| u \|_{\mathscr{U}}^2 + \omega_R(t)$ for all 
$x \in \mathscr{H}$, $u \in \mathscr{U}$, and $t \in (0, T)$ such that $\| x \|_{\mathscr{H}} \le R$;
\label{CorrectlyDefined_EvolEq_Assumpt}
}

\item{either the set $U$ is bounded in $L^2((0, T), \mathscr{U})$ or there exist $C_1 > 0$ and $\omega \in L^1(0, T)$
such that $\theta(x, u, t) \ge C_1 \| u \|_{\mathscr{U}}^2 + \omega(t)$ for all $x \in \mathscr{H}$, 
$u \in \mathscr{U}$, and $t \in [0, T]$;
\label{GrowthCond_EvolEq_Assumpt}
}

\item{$\theta$ is differentiable in $x$ and $u$, the functions $\nabla_x \theta$ and $\nabla_u \theta$ are continuous,
and for any $R > 0$ there exist $C_R > 0$, and a.e. nonnegative functions $\omega_R \in L^1(0, T)$ and 
$\eta_R \in L^2(0, T)$ such that
\begin{equation} \label{DerivGrowthCond_EvolEq}
  \| \nabla_x \theta(x, u, t) \|_{\mathscr{H}} \le C_R \| u \|_{\mathscr{U}}^2 + \omega_R(t), \quad
  \| \nabla_u \theta(x, u, t) \|_{\mathscr{U}} \le C_R \| u \|_{\mathscr{U}} + \eta_R(t)
\end{equation}
for all $x \in \mathscr{H}$, $u \in \mathscr{U}$, and $t \in (0, T)$ such that $\| x \|_{\mathscr{H}} \le R$;
\label{DerivGrowthCond_EvolEq_Assumpt}
}

\item{there exists a globally optimal solution of problem \eqref{EvolEqFixedEndPointProblem}, 
$\relint \mathcal{R}(x_0, T) \ne \emptyset$ and $x_T \in \relint \mathcal{R}(x_0, T)$.
\label{EndPointInterior_Assumpt}
}
\end{enumerate}
Then for all $c \in \mathbb{R}$ there exists $\lambda^*(c) \ge 0$ such that for any $\lambda \ge \lambda^*(c)$ 
the penalty function $\Phi_{\lambda}$ for problem \eqref{EvolEqFixedEndPointProblem} is completely exact on the set
$S_{\lambda}(c)$.
\end{theorem}

\begin{proof}
Our aim is to apply Theorem~\ref{Theorem_CompleteExactness}. It is easily seen that
assumption~\ref{CorrectlyDefined_EvolEq_Assumpt} ensures that the functional $\mathcal{I}(x, u)$ is correctly defined
and finite for any $(x, u) \in X$. In turn, from assumption~\ref{GrowthCond_EvolEq_Assumpt} it follows that for
any $c \in \mathbb{R}$ and $\lambda \ge 0$ there exists $K > 0$ such that $\| u \|_{L^2((0, T); \mathscr{U})} \le K$ for
any $(x, u) \in S_{\lambda}(c)$, and the penalty function $\Phi_{\lambda}$ is bounded below on $A$ for all 
$\lambda \ge 0$ (if $U$ is bounded, then this fact follows from assumption~\ref{CorrectlyDefined_EvolEq_Assumpt}). 
Hence taking into account \eqref{SolutionViaSemiGroup}, and the facts that $\| F_t \| \le \| F_T \|$ for any $t \le T$
(see \cite[formula $(4.2.5)$]{TucsnakWeiss}), and $\| \mathbb{T}_t \| \le M_{\omega} e^{\omega t}$ for all $t \ge 0$ and
for some $\omega \in \mathbb{R}$ and $M_{\omega} \ge 1$ by \cite[Proposition~2.1.2]{TucsnakWeiss} one obtains that
$\| x \|_{C([0, T]; \mathscr{H})} \le M_{\omega} \max_{t \in [0, T]} e^{\omega t} \| x_0 \| + \| F_T \| K$,
i.e. the set $S_{\lambda}(c)$ is bounded in $X$ for any $\lambda \ge 0$ and $c \in \mathbb{R}$. 

Observe that the penalty term $\varphi$ is continuous on $X$, since by the reverse triangle inequality one has
$$
  |\varphi(x, u) - \varphi(y, v)| = \big| \| x(T) - x_T \|_{\mathscr{H}} - \| y(T) - x_T \|_{\mathscr{H}} \big|
  \le \| x(T) - y(T) \|_{\mathscr{H}} \le \| x - y \|_{C([0, T]; \mathscr{H})}
$$
for all $(x, u), (y, v) \in X$. Furthermore from \eqref{SolutionViaSemiGroup}, the closedness of the set $U$, and 
the fact that $F_t$ continuously maps $L^2((0, T); \mathscr{U})$ to $\mathscr{H}$ by
\cite[Proposition~4.2.2]{TucsnakWeiss} it follows that the set $A$ is closed. 

Let us check that assumption~\ref{DerivGrowthCond_EvolEq_Assumpt} ensures that the functional $\mathcal{I}$ is
Lipschitz continuous on any bounded subset of $X$ (in particular, on any bounded open set containing the set
$S_{\lambda}(c)$). Indeed, fix any $(x, u) \in X$, $(h, v) \in X$, and $\alpha \in (0, 1]$. By the mean value theorem
for a.e. $t \in (0, T)$ there exists $\alpha(t) \in (0, \alpha)$ such that
\begin{multline} \label{MeanValue_Func_EvolEq}
  \frac{1}{\alpha} \Big( \theta(x(t) + \alpha h(t), u(t) + \alpha v(t), t) - \theta(x(t), u(t), t) \Big) \\
  = \langle \nabla_x \theta(x(t) + \alpha(t) h(t), u(t) + \alpha(t) v(t), t), h(t) \rangle 
  + \langle \nabla_u \theta(x(t) + \alpha(t) h(t), u(t) + \alpha(t) v(t), t), v(t) \rangle.
\end{multline}
The right-hand side of this equality converges to
$\langle \nabla_x \theta(x(t), u(t), t), h(t) \rangle + \langle \nabla_u \theta(x(t), u(t), t), v(t) \rangle$
as $\alpha \to 0$ for a.e. $t \in (0, T)$ due to the continuity of the gradients $\nabla_x \theta$ and 
$\nabla_u \theta$. Furthermore, by \eqref{DerivGrowthCond_EvolEq} there exist $C_R > 0$, and a.e. nonnegative
functions $\omega_R \in L^1(0, T)$ and $\eta_R \in L^2(0, T)$ such that 
\begin{align*}
  \big| \langle \nabla_x \theta(x(t) + \alpha h(t), u(t) + \alpha v(t), t), h(t) \rangle \big| &\le
  \big( 4 C_R ( \| u(t) \|^2_{\mathscr{U}} + \| v(t) \|^2_{\mathscr{U}} ) + \omega_R(t) \big) 
  \| h \|_{C([0, T]; \mathscr{H})} \\
  \big| \langle \nabla_u \theta(x(t) + \alpha h(t), u(t) + \alpha v(t), t), v(t) \rangle \big| &\le
  \big( C_R (\| u(t) \|_{\mathscr{U}} + \| v(t) \|_{\mathscr{U}}) + \eta_R(t) \big) \| v(t) \|_{\mathscr{U}} 
\end{align*}
for a.e. $t \in (0, T)$ and all $\alpha \in [0, 1]$. Note that the right-hand sides of these inequalities belong to
$L^1(0, T)$ and do not depend on $\alpha$. Therefore, integrating \eqref{MeanValue_Func_EvolEq} from $0$ to $T$ and
passing to the limit with the use of Lebesgue's dominated convergence theorem one obtains that the functional
$\mathcal{I}$ is G\^{a}teaux differentiable at every point $(x, u) \in X$, and its G\^{a}teaux derivative has the form
$$
  \mathcal{I}'(x, u)[h, v] = \int_0^T \Big( \langle \nabla_x \theta(x(t), u(t), t), h(t) \rangle + 
  \langle \nabla_u \theta(x(t), u(t), t), v(t) \rangle \Big) \, dt.
$$
Hence and from \eqref{DerivGrowthCond_EvolEq} it follows that for any $R > 0$ and $(x, u) \in X$ such that
$\| x \|_{C([0, T]; \mathscr{H})} \le R$ there exist $C_R > 0$, and a.e. nonnegative functions $\omega_R \in L^1(0, T)$
and $\eta_R \in L^2(0, T)$ such that
$$
  \big\| \mathcal{I}'(x, u) \big\| \le C_R \| u \|^2_{L^2((0, T); \mathscr{U})} + \| \omega_R \|_1
  + C_R \| u \|_{L^2((0, T); \mathscr{U})} + \| \eta_R \|_2.
$$
Therefore, the G\^{a}teaux derivative of $\mathcal{I}$ is bounded on bounded subsets of the space $X$, which, as is
well-known and easy to check, implies that the functional $\mathcal{I}$ is Lipschitz continuous on bounded subsets 
of $X$. 

Fix any $\lambda \ge 0$ and $c > \inf_{(x, u) \in \Omega} \mathcal{I}(x, u)$. By Theorem~\ref{Theorem_CompleteExactness}
it remains to check that there exists $a > 0$ such that $\varphi^{\downarrow}_A(x, u) \le - a$ for any 
$(x, u) \in S_{\lambda}(c)$ such that $\varphi(x, u) > 0$. Choose any such $(x, u)$ and 
$(\widehat{x}, \widehat{u}) \in \Omega$. Note that $x(T) \ne x_T$ due to the inequality $\varphi(x, u) > 0$. Define
$\Delta x = ( \widehat{x} - x ) / \sigma$ and $\Delta u = ( \widehat{u} - u ) / \sigma$, where
$\sigma = \| \widehat{x} - x \|_{C([0, T]; \mathscr{H})} + \| \widehat{u} - u \|_{L^2((0, T); \mathscr{U})} > 0$.
Then $\| (\Delta x, \Delta u) \|_X = 
\| \Delta x \|_{C([0, T]; \mathscr{H})} + \| \Delta u \|_{L^2((0, T); \mathscr{U})} = 1$. Due to the linearity of the
system and the convexity of the set $U$, for any $\alpha \in [0, \sigma]$ one has 
$(x + \alpha \Delta x, u + \alpha \Delta x) \in A$, 
$(x + \alpha \Delta x)(T) = x(T) + \alpha \sigma^{-1} (x_T - x(T))$, since $\widehat{x}(T) = x_T$ by definition.
Hence
\begin{align*}
  \varphi^{\downarrow}_A(x, u) &\le \lim_{\alpha \to +0} 
  \frac{\varphi(x + \alpha \Delta x, u + \alpha \Delta u) - \varphi(x, u)}{\alpha \| (\Delta x, \Delta u) \|_X} \\
  &= \lim_{\alpha \to +0} 
  \frac{(1 - \alpha \sigma^{-1}) \| x(T) - x_T \|_{\mathscr{H}} - \| x(T) - x_T \|_{\mathscr{H}}}{\alpha}
  = - \frac{1}{\sigma} \| x(T) - x_T \|_{\mathscr{H}}.
\end{align*}
Therefore, it remains to check that there exists $C > 0$ such that for any 
$(x, u) \in S_{\lambda}(c) \setminus \Omega$ one can find $(\widehat{x}, \widehat{u}) \in \Omega$ satisfying the
inequality
\begin{equation} \label{PropertyS_EvolutionEquation}
  \| x - \widehat{x} \|_{C([0, T]; \mathscr{H})} 
  + \| u - \widehat{u} \|_{L^2((0, T); \mathscr{U})} \le C \| x(T) - x_T \|_{\mathscr{H}}.
\end{equation}
Then $\varphi^{\downarrow}_A(x, u) \le - 1 / C$ for any such $(x, u)$, and the proof is complete.

From \eqref{SolutionViaSemiGroup} and the inequality $\| F_t \| \le \| F_T \|$, $t \in [0, T]$ 
(see \cite[formula $(4.2.5)$]{TucsnakWeiss}), it follows that for any $(x, u) \in A$ and 
$(\widehat{x}, \widehat{u}) \in A$ one has
$\| x - \widehat{x} \|_{C([0, T]; \mathscr{H})} \le \| F_T \| \| u - \widehat{u} \|_{L^2((0, T); \mathscr{U})}$.
Consequently, it is sufficient to check that there exists $C > 0$ such that for any 
$(x, u) \in S_{\lambda}(c) \setminus \Omega$ one can find $(\widehat{x}, \widehat{u}) \in \Omega$ satisfying the
inequality
\begin{equation} \label{MetricReg_InputMap}
  \| u - \widehat{u} \|_{L^2((0, T); \mathscr{U})} \le C \| x(T) - x_T \|_{\mathscr{H}}.
\end{equation}
To this end, fix any $(x_*, u_*) \in \Omega$, and denote by 
$\mathcal{T} \colon \cl \linhull (U - u_*) \to \cl \linhull F_T(U - u_*)$ the mapping such that 
$\mathcal{T}(u) = F_T (u)$ for any $u \in \cl \linhull (U - u_*)$. Note that $\mathcal{T}$ is correctly defined, since
the operator $F_T$ maps $\cl \linhull (U - u_*)$ to $\cl \linhull F_T(U - u_*)$. Indeed, by definition 
$F_T(\linhull (U - u_*)) \subseteq \cl \linhull F_T(U - u_*)$. If $u_0 \in \cl \linhull U$, then there exists a
sequence $\{ u_n \} \subset \linhull (U - u_*)$ converging to $u_0$. Due to the continuity of $F_T$ the sequence 
$\{ F_T u_n \}$ converges to $F_T u_0$, which yields $F_T u_0 \in \cl \linhull F_T(U - u_*)$.

Observe that $\mathcal{T}$ is a bounded linear operator between Banach spaces, since the operator $F_T$ is bounded.
Furthermore, by \eqref{SolutionViaSemiGroup} for any $u \in U$ one has $F_T(u - u_*) = x(T) - x_T$, which implies that
$\mathcal{T}(U - u_*) = F_T(U - u_*) = \mathcal{R}(x_0, T) - x_T$. Therefore, by
assumption~\ref{EndPointInterior_Assumpt} one has $0 \in \interior \mathcal{T}(U - u_*)$, since the closed affine hull
of $\mathcal{R}(x_0, T)$ coincides with $\cl \linhull F_T(U - u_*) + x_T$ due to the fact that 
$0 \in F_T(U - u_*)$. Hence by Robinson's theorem (Theorem~\ref{Theorem_Robinson_Ursescu} with $C = U - u_*$, 
$x^* = 0$, and $y = 0$) there exists $\kappa > 0$ such that
\begin{equation} \label{RobinsonUrsescu_InputMap_direct}
  \dist\big( u - u_*, \mathcal{T}^{-1}(0) \cap (U - u_*) \big) \le
  \kappa \big( 1 + \| u - u_* \|_{L^2((0, T); \mathscr{H})} \big) 
  \big\| \mathcal{T}(u - u_*) \big\|_{\mathscr{H}}
  \quad \forall u \in U.
\end{equation}
Fix any $(x, u) \in S_{\lambda}(c) \setminus \Omega$ (i.e. $x(T) \ne x_T$). Then taking into account the fact that
$\mathcal{T}(u - u_*) = x(T) - x_T$ and utilising inequality \eqref{RobinsonUrsescu_InputMap_direct} one obtains that
there exists $v \in U - u_*$ such that $\mathcal{T}(v) = 0$ and
\begin{equation} \label{RobinsonUrsescu_InputMap}
  \| u - u_* - v \|_{L^2((0, T); \mathscr{H})} \le 
  2 \kappa \big( 1 + \| u - u_* \|_{L^2((0, T); \mathscr{H})} \big) \big\| x(T) - x_T \big\|_{\mathscr{H}}.
\end{equation}
Define $\widehat{u} = u_* + v \in U$, and let $\widehat{x}$ be the corresponding solution of \eqref{LinearEvolEq}. Then
$\widehat{x}(T) - x_T = \mathcal{T}(\widehat{u} - u_*) = \mathcal{T}(v) = 0$, i.e. 
$(\widehat{x}, \widehat{u}) \in \Omega$. Note that 
$C := \sup_{(x, u) \in S_{\lambda}(c)} 2 \kappa ( 1 + \| u - u_* \|_{L^2((0, T); \mathscr{H})} ) < + \infty$ due to
the boundedness of the set $S_{\lambda}(c)$. Consequently, by \eqref{RobinsonUrsescu_InputMap} one 
for any $(x, u) \in S_{\lambda}(c) \setminus \Omega$ there exists $(\widehat{x}, \widehat{u}) \in \Omega$ such that
$\| u - \widehat{u} \|_{L^2((0, T); \mathscr{H})} \le C \| x(T) - x_T \|_{\mathscr{H}}$, i.e. \eqref{MetricReg_InputMap}
holds true, and the proof is complete.
\end{proof}

\begin{remark}
Let us note that for the validity of the assumption $x_T \in \relint \mathcal{R}(x_0, T)$ in the case 
$\image(F_T) = \mathscr{H}$ it is sufficient to suppose that $x_T$ belongs to the interior of $\mathcal{R}(x_0, T)$,
while in the case $U = L^2((0, T); \mathscr{U})$ this assumption is satisfied iff the image of the input map $F_T$ is
closed.
\end{remark}

\begin{remark} \label{Remark_ComparisonZuazua}
Recall that system \eqref{LinearEvolEq} is called \textit{exactly controllable} using $L^2$-controls in
time $T$, if for any initial state $x_0 \in \mathscr{H}$ and for any final state $x_T \in \mathscr{H}$ there exists 
$u \in L^2((0, T), \mathscr{U})$ such that for the corresponding solution $x$ of \eqref{LinearEvolEq} one has 
$x(T) = x_T$. It is easily seen that this system is exactly controllable using $L^2$-controls in time $T$ iff the input
map $F_T$ is surjective, i.e. $\image(F_T) = \mathscr{H}$. Thus, in particular,
in Theorem~\ref{Theorem_Exactness_EvolutionEquations} it is sufficient to suppose that system \eqref{LinearEvolEq} is
exactly controllable and $x_T \in \interior \mathcal{R}(x_0, T)$. If, in addition, $\interior U \ne \emptyset$, then it
is sufficient to suppose that system \eqref{LinearEvolEq} is exactly controllable and there exists a feasible point
$(x_*, u_*)$ of problem \eqref{EvolEqFixedEndPointProblem} such that $u_* \in \interior U$.

The exactness of the penalty function $\Phi_{\lambda}$ for problem \eqref{EvolEqFixedEndPointProblem} with
$\theta(x, u, t) = \| u \|_{\mathscr{U}}^2 / 2$ and no constraints on the control inputs (i.e. $U = L^2((0, T);
\mathscr{U})$) was proved by Gugat and Zuazua \cite{Zuazua} under the assumption that system \eqref{LinearEvolEq} is
exactly controllable, and the control $u$ from the definition of exact controllability satisfies the inequality
\begin{equation} \label{ExactControl_Unnecessary}
  \| u \|_{L^2((0, T); \mathscr{U})} \le C \big( \| x_0 \| + \| x_T \| \big)
\end{equation}
for some $C > 0$ independent of $x_0$ and $x_T$. Note that our Theorem~\ref{Theorem_Exactness_EvolutionEquations}
significantly generalises and strengthens \cite[Theorem~1]{Zuazua}, since we consider a more general objective function
and convex constraints on the control inputs, impose a less restrictive assumption on the input map $F_T$ (instead of
exact controllability it is sufficient to suppose that $\image(F_T)$ is closed), and demonstrate that inequality
\eqref{ExactControl_Unnecessary} is, in fact, redundant.
\end{remark}

Let us also extend Theorem~\ref{Theorem_FixedEndPointProblem_Linear_Global} to the case of optimal control problems for
linear evolution equations.

\begin{theorem} \label{Theorem_Exactness_EvolutionEquations_Global}
Let all assumptions of Theorem~\ref{Theorem_Exactness_EvolutionEquations} be valid, and suppose that either the set $U$
of admissible control inputs is bounded in $L^2((0, T), \mathscr{U})$ or the function $(x, u) \mapsto \theta(x, u, t)$
is convex for all $t \in [0, T]$. Then the penalty function $\Phi_{\lambda}$ for problem
\eqref{EvolEqFixedEndPointProblem} is completely exact on $A$.
\end{theorem}

\begin{proof}
Suppose at first that the set $U$ is bounded. Recall that the input map $F_t$ continuously maps 
$L^2((0, T); \mathscr{U})$ to $\mathscr{H}$ by \cite[Proposition~4.2.2]{TucsnakWeiss} and $\| F_t \| \le \| F_T \|$ for
any $t \le T$ (see \cite[formula $(4.2.5)$]{TucsnakWeiss}). Note also that by \cite[Proposition~2.1.2]{TucsnakWeiss}
there exist $\omega \in \mathbb{R}$ and $M_{\omega} \ge 1$ such that $\| \mathbb{T}_t \| \le M_{\omega} e^{\omega t}$
for all $t \ge 0$.

Fix any $(x, u) \in A$. By our assumption there exists $K > 0$ such that $\| u \|_{L^2((0, T), \mathscr{U})} \le K$ for
any $u \in U$. Hence 
$\| x \|_{C([0, T]; \mathscr{H})} \le M_{\omega} \max_{t \in [0, T]} e^{\omega t} \| x_0 \| + \| F_T \| K$
due to \eqref{SolutionViaSemiGroup}, and the bounds on $\| \mathbb{T}_t \|$ and $\| F_t \|$. Thus, the set $A$ is
bounded in $X$. Now, arguing in the same way as in the proof of Theorem~\ref{Theorem_Exactness_EvolutionEquations}, but
replacing $S_{\lambda}(c)$ with $A$ and utilising Theorem~\ref{THEOREM_COMPLETEEXACTNESS_GLOBAL} instead of
Theorem~\ref{Theorem_CompleteExactness}, one obtains the required result. 

Suppose now that the function $(x, u) \mapsto \theta(x, u, t)$ is convex. Then the functional $\mathcal{I}(x, u)$
and the penalty function $\Phi_{\lambda}$ are convex. Hence with the use of the fact that the set $A$ is convex one
obtains that any point of local minimum of $\Phi_{\lambda}$ on $A$ is also a point of global minimum of $\Phi_{\lambda}$
on $A$. Furthermore, any inf-stationary point of $\Phi_{\lambda}$ on $A$ is also a point of global minimum of
$\Phi_{\lambda}$ on $A$. Indeed, let $(x^*, u^*)$ be and inf-stationary point of $\Phi_{\lambda}$ on $A$. Arguing by
reductio ad absurdum, suppose that $(x^*, u^*)$ is not a point of global minimum. Then there exists $(x_0, u_0) \in A$
such that $\Phi_{\lambda}(x_0, u_0) < \Phi_{\lambda}(x^*, u^*)$. By applying the convexity of $\Phi_{\lambda}$ one gets
that
$$
  \Phi_{\lambda}(x^* + \alpha( x_0 - x^* ), u^* + \alpha (u_0 - u^*)) \le
  \Phi_{\lambda}(x^*, u^*) + \alpha(\Phi_{\lambda}(x_0, u_0) - \Phi_{\lambda}(x^*, u^*))
  \quad \forall \alpha \in [0, 1].
$$
Consequently, one has
$$
  (\Phi_{\lambda})^{\downarrow}_A(x^*, u^*) \le \liminf_{\alpha \to + 0}
  \frac{\Phi_{\lambda}(x^* + \alpha( x_0 - x^* ), u^* + \alpha (u_0 - u^*)) - \Phi_{\lambda}(x^*, u^*)}
  {\alpha \| (x^*, u^*) - (x_0, u_0) \|_X} \le 
  \frac{\Phi_{\lambda}(x_0, u_0) - \Phi_{\lambda}(x^*, u^*)}{\| (x^*, u^*) - (x_0, u_0) \|_X} < 0
$$
which is impossible, since by the definition of inf-stationary point $(\Phi_{\lambda})^{\downarrow}_A(x^*, u^*) \ge 0$.

Similarly, any point of local minimum/inf-stationary point of $\mathcal{I}$ on $\Omega$ is a globally optimal solution
of problem \eqref{EvolEqFixedEndPointProblem} due to the convexity of $\mathcal{I}$ and $\Omega$. Therefore, in the
convex case the penalty function $\Phi_{\lambda}$ is completely exact on $A$ if and only if it is globally exact. In
turn, the global exactness of this function follows from Theorem~\ref{Theorem_Exactness_EvolutionEquations}.
\end{proof}

\subsection{Nonlinear Systems: Complete Exactness}

Now we turn to the analysis of nonlinear finite dimensional fixed-endpoint optimal control problems of the form:
\begin{equation} \label{FixedEndPointProblem}
\begin{split}
  {}&\min \: \mathcal{I}(x, u) = \int_0^T \theta(x(t), u(t), t) \, dt \\
  {}&\text{subject to } \dot{x}(t) = f(x(t), u(t), t), \quad t \in [0, T], \quad
  x(0) = x_0, \quad x(T) = x_T, \quad u \in U.
\end{split}
\end{equation}
Here $\theta \colon \mathbb{R}^d \times \mathbb{R}^m \times [0, T] \to \mathbb{R}$ and
$f \colon \mathbb{R}^d \times \mathbb{R}^m \times [0, T] \to \mathbb{R}^d$ are given functions, 
$x_0, x_T \in \mathbb{R}^d$, and $T > 0$ are fixed, $x \in W^d_{1, p}(0, T)$, $U \subseteq L_q^m(0, T)$ is
a nonempty closed set, and $1 \le p, q \le + \infty$.

As in the case of linear problems, we penalise only the terminal constraint $x(T) = x_T$. To this end, 
define $X = W_{1, p}^d(0, T) \times L_q^m(0, T)$, $M = \{ (x, u) \in X \mid x(T) = x_T \}$, and
\begin{equation} \label{Set_A_Nonlinear_FixedEndPoint}
  A = \Big\{ (x, u) \in X \Bigm| x(0) = x_0, \: u \in U, \:
  \dot{x}(t) = f(x(t), u(t), t) \text{ for a.e. } t \in [0, T] \Big\}.
\end{equation}
Then problem \eqref{FixedEndPointProblem} can be rewritten as the problem of minimising $\mathcal{I}(x, u)$
subject to $(x, u) \in M \cap A$. As in the previous sections, define $\varphi(x, u) = |x(T) - x_T|$. Then 
$M = \{ (x, u) \in X \mid \varphi(x, u) = 0 \}$, and one can consider the penalised problem
\begin{equation} \label{PenProblem_NonlinearFixedEndPoint}
\begin{split}
  {}&\min \: \Phi_{\lambda}(x, u) = \mathcal{I}(x, u) + \lambda \varphi(x, u) 
  = \int_0^T \theta(x(t), u(t), t) \, dt + \lambda |x(T) - x_T| \\
  {}&\text{subject to } \dot{x}(t) = f(x(t), u(t), t), \quad t \in [0, T], \quad
  x(0) = x_0, \quad u \in U,
\end{split}
\end{equation}
which is a nonlinear free-endpoint optimal control problem.

The nonlinearity of the systems makes an analysis of the exactness of the penalty function $\Phi_{\lambda}(x, u)$ a
very challenging problem. Unlike the linear case, it does not seem possible to obtain any easily verifiable conditions
for the complete exactness of this function. Therefore, the main goal of this section is to understand what 
properties the system $\dot{x} = f(x, u, t)$ must have for the penalty function $\Phi_{\lambda}(x, u)$ to be completely
exact.

In the linear case, the main assumption ensuring the complete exactness of the penalty function was 
$x_T \in \relint \mathcal{R}(x_0, T)$. Therefore, it is natural to expect that in the nonlinear case one must also
impose some assumptions on the reachable set of the system $\dot{x} = f(x, u, t)$. Moreover, in the linear case we
utilised Robinson's theorem, but there are no nonlocal analogues of this theorem in the nonlinear case.
Consequently, we must impose an assumption that allows one to avoid the use of this theorem.

Thus, to prove the complete exactness of the penalty function $\Phi_{\lambda}$ in the nonlinear case we need to impose
two assumptions on the controlled system $\dot{x} = f(x, u, t)$. The first one does not allow the reachable set
of this system to be, roughly speaking, too ``wild'' near the point $x_T$, while the second one ensures that this system
is, in a sense, sensitive enough with respect to the control inputs. It should be mentioned that the exactness of the
penalty function $\Phi_{\lambda}$ for problem \eqref{FixedEndPointProblem} can be proved under a much weaker assumption
that imposes some restrictions on the reachable set and sensitivity with respect to the control inputs simultaneously.
However, for the sake of simplicity we split this rather complicated assumption in two assumptions that are much easier
to understand and analyse.

Denote by $\mathcal{R}(x_0, T) = \{ \xi \in \mathbb{R}^d \mid \exists (x, u) \in A \colon \xi = x(T) \}$ the
set that is reachable in time $T$. We obviously suppose that $x_T \in \mathcal{R}(x_0, T)$. Also, we exclude the trivial
case when $x_T$ is an isolated point of $\mathcal{R}(x_0, T)$, since in this case the penalty function
$\Phi_{\lambda}$ is completely exact on $S_{\lambda}(c)$ for any $c \in \mathbb{R}$ iff $\Phi_{\lambda}$ is bounded
below on $A$ due to the fact that in this case $\Omega_{\delta} \setminus \Omega = \emptyset$ for any sufficiently
small $\delta > 0$ (see~Remark~\ref{Remark_OmegaDeltaEmpty}).

\begin{definition}
One says that the set $\mathcal{R}(x_0, T)$ has \textit{the negative tangent angle property} near $x_T$, if there
exist a neighbourhood $\mathcal{O}(x_T)$ of $x_T$ and $\beta > 0$ such that for any 
$\xi \in \mathcal{O}(x_T) \cap \mathcal{R}(x_0, T)$, $\xi \ne x_T$, there exists a sequence 
$\{ \xi_n \} \subset \mathcal{R}(x_0, T)$ converging to $\xi$ and such that
\begin{equation} \label{NegativeTangentAngleCond}
  \left\langle \frac{\xi - x_T}{|\xi - x_T|}, \frac{\xi_n - \xi}{|\xi_n - \xi|} \right\rangle \le - \beta
  \quad \forall n \in \mathbb{N}.
\end{equation}
\end{definition}

One can easily see that if $x_T$ belongs to the interior of $\mathcal{R}(x_0, T)$ or if there exists a neighbourhood
$\mathcal{O}(x_T)$ of $x_T$ such that the intersection $\mathcal{O}(x_T) \cap \mathcal{R}(x_0, T)$ is convex, then the
set $\mathcal{R}(x_0, T)$ has the negative tangent angle property near $x_T$ (take as $\{ \xi_n \}$ any sequence of
points from the segment $\co\{ x_T, \xi \}$ converging to $\xi$ and put $\beta = 1$). However, this property holds true
in a much more general case. In particular, $x_T$ can be the vertex of a cusp. 

The negative tangent angle property excludes the sets that, roughly speaking, are ``very porous'' near $x_T$ (i.e. sets
having an infinite number of ``holes'' in any neighbourhood of $x_T$) or are very wiggly near this point (like the graph
of $y = x \sin(1 / x)$ near $(0, 0)$). Furthermore, bearing in mind the equality
$$
  \big\{ \xi \in \mathbb{R}^d \mid \exists (x, u) \in \Omega_{\delta} \setminus \Omega \colon \xi = x(T) \big\} 
  = \{ \xi \in \mathcal{R}(x_0, T) \mid 0 < |\xi - x_T| < \delta  \},
$$
the definition of the rate of steepest descent, and the fact that $\varphi(x, u) = |x(T) - x_T|$ one can check that for
the validity of the inequality $\varphi^{\downarrow}_A(x, u) \le - a$ for all 
$(x, u) \in \Omega_{\delta} \setminus \Omega$ and some $a, \delta > 0$ it is \textit{necessary} that there exists
$\beta > 0$ such that for any $\xi \in \mathcal{R}(x_0, T)$ lying in a neighbourhood of $x_T$ inequality
\eqref{NegativeTangentAngleCond} holds true. 

Indeed, suppose that $\varphi^{\downarrow}_A(x, u) \le - a$ for all $(x, u) \in \Omega_{\delta} \setminus \Omega$ and
some $a, \delta > 0$. Let $\xi \in \mathcal{R}(x_0, T)$ satisfy the inequalities
$0 < |\xi - x_T| < \delta$, and $(x, u) \in \Omega_{\delta} \setminus \Omega$ be such that $x(T) = \xi$. By the
definition of $\varphi^{\downarrow}_A(x, u)$ there exists a sequence $\{ (x_n, u_n) \} \subset A$ converging to 
$(x, u)$ and such that
\begin{multline*}
  - \frac{2a}{3} \ge \frac{\varphi(x_n, u_n) - \varphi(x, u)}{\| (x_n - x, u_n - u) \|_X}
  = \frac{|x_n(T) - x_T| - |x(T) - x_T|}{\| (x_n - x, u_n - u) \|_X} \\
  = \frac{1}{\| (x_n - x, u_n - u) \|_X} 
  \left( \left\langle \frac{x(T) - x_T}{|x(T) - x_T|}, x_n(T) - x(T) \right \rangle + o(|x_n(T) - x(T)| \right)
\end{multline*}
for all $n \in \mathbb{N}$. Hence with the use of inequality \eqref{SobolevImbedding} one obtains that 
$0 < |x_n(T) - x_T| < \delta$ for any sufficient large $n$, and there exists $n_0 \in \mathbb{N}$ such that inequality
\eqref{NegativeTangentAngleCond} is satisfied with $\xi_n = x_{n + n_0}(T)$, $n \in \mathbb{N}$, and 
$\beta = a / 3 C_p$. Thus, the negative tangent angle property is closely related to the validity of assumption
\ref{NegativeDescentRateAssumpt} of Theorem~\ref{Theorem_CompleteExactness}.

\begin{definition} \label{Def_SensitivityProperty}
Let $K \subset A$ be a given set. One says that the property $(\mathcal{S})$ is satisfied on the set $K$, if there
exists $C > 0$ such that for any $(x, u) \in K$ one can find a neighbourhood $\mathcal{O}(x(T)) \subset \mathbb{R}^d$ of
$x(T)$ such that for all $\widehat{x}_T \in \mathcal{O}(x(T)) \cap \mathcal{R}(x_0, T)$ there exists a control input
$\widehat{u} \in U$ that steers the system from $x(0) = x_0$ to $\widehat{x}_T$ in time $T$, and 
\begin{equation} \label{SensitivityCondition}
  \| u - \widehat{u} \|_q + \| x - \widehat{x} \|_{1, p} \le C | x(T) - \widehat{x}(T) |,
\end{equation}
where $\widehat{x}$ is a trajectory corresponding to $\widehat{u}$, i.e. $(\widehat{x}, \widehat{u}) \in A$.
\end{definition}

Let $K \subset A$ be a given set. Recall that the set $A$ consists of all those pairs $(x, u) \in X$ for which 
$u \in U$, and $x$ is a solution of $\dot{x} = f(x, u, t)$ with $x(0) = x_0$
(see~\eqref{Set_A_Nonlinear_FixedEndPoint}). Roughly speaking, the property $(\mathcal{S})$ is satisfied on $K$ iff for
any $(x, u) \in K$ and any reachable end-point $\widehat{x}_T \in \mathcal{R}(x_0, T)$ lying sufficiently close to
$x(T)$ one can reach $\widehat{x}_T$ by slightly changing the control input $u$ in such a way that the corresponding
trajectory stays in a sufficiently small neighbourhood of $x(\cdot)$ (more precisely, the magnitude of change of $u$ and
$x$ must be proportional to $| x(T) - \widehat{x}_T |$). Note that the property $(\mathcal{S})$ implicitly appeared in
the proofs of Theorems~\ref{Theorem_FixedEndPointProblem_Linear} and
\ref{Theorem_Exactness_EvolutionEquations} (cf.~\eqref{ErrorBound_TerminalConstraint} and
\eqref{PropertyS_EvolutionEquation}) and was proved with the use of Robinson's theorem.

\begin{remark} \label{Remark_SensitivityProperty}
Let the function $(x, u) \mapsto f(x, u, t)$ be locally Lipschitz continuous uniformly for all $t \in (0, T)$. Suppose
also that the set $A$ is bounded in $L_{\infty}^d(0, T) \times L_{\infty}^m(0, T)$, i.e. the control inputs and
corresponding trajectories of the system are uniformly bounded. By definition
$|x_1(t) - x_2(t)| = \int_0^t | f(x_1(\tau), u_1(\tau), \tau) - f(x_2(\tau), u_2(\tau), \tau) | \, dt$ for any
$(x_1, u_1), (x_2, u_2) \in A$. Therefore, due to the boundedness of $A$ and the Lipschitz continuity of $f$ there
exists $L > 0$ such that
$$
  |x_1(t) - x_2(t)|
  \le L \int_0^t |x_1(\tau) - x_2(\tau)| d \tau + L \int_0^t |u_1(\tau) - u_2(\tau)| \, d \tau
  \le L \int_0^t |x_1(\tau) - x_2(\tau)| d \tau + L T^{1/q'} \| u_1 - u_2 \|_q
$$
for all $t \in [0, T]$. Hence with the use of the Gr\"{o}nwall-Bellman inequality one can easily check that there exists
$L_1 > 0$ such that $\| x_1 - x_2 \|_{\infty} \le L_1 \| u_1 - u_2 \|_q$ for any $(x_1, u_1), (x_2, u_2) \in A$. Then by
applying the inequality
$$
  |\dot{x}_1(t) - \dot{x}_2(t)| \le \big| f(x_1(t), u_1(t), t) - f(x_2(t), u_2(t), t) \big|
  \le L |x_1(t) - x_2(t)| + L |u_1(t) - u_2(t)|
$$
and H\"{o}lder's inequality (here we suppose that $q \ge p$) one obtains that there exists $L_2 > 0$ such that
$\| x_1 - x_2 \|_{1, p} \le L_2 \| u_1 - u_2 \|_q$ for all $(x_1, u_1), (x_2, u_2) \in A$. In other words, the map
$u \mapsto x_u$, where $x_u$ is a solution of $\dot{x} = f(x, u, t)$ with $x(0) = x_0$, is Lipschitz continuous on $U$.
Therefore, under the assumptions of this remark inequality \eqref{SensitivityCondition} in the definition of the
property $(\mathcal{S})$ can be replaced with the inequality $\| u - \widehat{u} \|_q \le C | x(T) - \widehat{x}(T) |$.
\end{remark}

A detailed analysis of the property $(\mathcal{S})$ lies beyond the scope of this paper. Here we only note that the
property $(\mathcal{S})$ is, in essence, a reformulation of the assumption that the mapping $u \mapsto x_u(T)$ is 
\textit{metrically regular} on the set $K$ (here $x_u$ is a solution of $\dot{x} = f(x, u, t)$ with $x(0) = x_0$). Thus,
it seems possible to apply general results on metric regularity \cite{Aze,Cominetti,Ioffe,Dmitruk} to verify whether the
property $(\mathcal{S})$ is satisfied in particular cases. Our aim is to show that this property along with the negative
tangent angle property ensures that the penalty function $\Phi_{\lambda}$ for fixed-endpoint problem
\eqref{FixedEndPointProblem} is completely exact. Denote by $\mathcal{I}^*$ the optimal value of this problem.

\begin{theorem} \label{Theorem_FixedEndPointProblem_NonLinear}
Let the following assumptions be valid:
\begin{enumerate}
\item{$\theta$ is continuous and differentiable in $x$ and $u$, and the functions $\nabla_x \theta$, $\nabla_u \theta$,
and $f$ are continuous;
}

\item{either $q = + \infty$ or $\theta$ and $\nabla_x \theta$ satisfy the growth condition of order $(q, 1)$, 
$\nabla_u \theta$ satisfies the growth condition of order $(q - 1, q')$;
}

\item{there exists a globally optimal solution of problem \eqref{FixedEndPointProblem};}

\item{the set $\mathcal{R}(x_0, T)$ has the negative tangent angle property near $x_T$;}

\item{there exist $\lambda_0 > 0$, $c > \mathcal{I}^*$, and $\delta > 0$ such that the set 
$S_{\lambda_0}(c) \cap \Omega_{\delta}$ is bounded in $W^d_{1, p}(0, T) \times L_q^m(0, T)$, the property
$(\mathcal{S})$ is satisfied on $S_{\lambda_0}(c) \cap (\Omega_{\delta} \setminus \Omega)$, and the function
$\Phi_{\lambda_0}(x, u)$ is bounded below on $A$.
}
\end{enumerate}
Then there exists $\lambda^* \ge 0$ such that for any $\lambda \ge \lambda^*$ the penalty function $\Phi_{\lambda}$ for
problem \eqref{FixedEndPointProblem} is completely exact on $S_{\lambda}(c)$. 
\end{theorem}

\begin{proof}
As was noted in the proof of Theorem~\ref{Theorem_FixedEndPointProblem_Linear}, the growth conditions on the function
$\theta$ and its derivatives ensure that the functional $\mathcal{I}$ is Lipschitz continuous on any bounded open set
containing the set $S_{\lambda_0}(c) \cap \Omega_{\delta}$. The continuity of the penalty term 
$\varphi(x, u) = |x(T) - x_T|$ can also be verified in the same way as in the proof of
Theorem~\ref{Theorem_FixedEndPointProblem_Linear}.

Let us check that the set $A$ is closed. Indeed, choose any sequence $\{ (x_n, u_n) \} \subset A$ converging to some
$(x_*, u_*) \in X$. Recall that the set $U$ is closed and by definition $\{ u_n \} \subset U$. Therefore $u_* \in U$.
By inequality \eqref{SobolevImbedding} the sequence $x_n$ converges to $x_*$ uniformly on $[0, T]$,
which, in particular, implies that $x_*(0) = x_0$. Note also that by definition $\{ \dot{x}_n \}$ converges to
$\dot{x}_*$ in $L_p^d(0, T)$, while $\{ u_n \}$ converges to $u_*$ in $L_q^m(0, T)$. As is well known 
(see, e.g. \cite[Theorem~2.20]{FonsecaLeoni}), one can extract subsequences $\{ \dot{x}_{n_k} \}$ and 
$\{ u_{n_k} \}$ that converge almost everywhere. From the fact that $(x_{n_k}, u_{n_k}) \in A$ it follows that
$\dot{x}_{n_k}(t) = f(x_{n_k}(t), u_{n_k}(t), t)$ for a.e. $t \in (0, T)$. Consequently, passing to the limit as 
$k \to \infty$ with the use of the continuity of $f$ one obtains that $\dot{x}_*(t) = f(x_*(t), u_*(t), t)$ for a.e. 
$t \in (0, T)$, i.e. $(x_*, u_*) \in A$, and the set $A$ is closed. Thus, by Theorem~\ref{Theorem_CompleteExactness}
it remains to check that there exists $0 < \eta \le \delta$ such that $\varphi^{\downarrow}_A(x, u) \le - a$ for any
$(x, u) \in S_{\lambda_0}(c) \cap (\Omega_{\eta} \setminus \Omega)$.

Let $0 < \eta \le \delta$ be arbitrary, and fix $(x, u) \in S_{\lambda_0}(c) \cap (\Omega_{\eta} \setminus \Omega)$. By
definition one has $0 < \varphi(x, u) = |x(T) - x_T| < \eta$. Decreasing $\eta$, if necessary, and utilising the
negative tangent angle property one obtains that there exist $\beta > 0$ (independent of $(x, u)$) and a sequence 
$\{ \xi_n \} \subset \mathcal{R}(x_0, T)$ converging to $x(T)$ such that 
\begin{equation} \label{NegativeTangentAngle}
  \left\langle \frac{x(T) - x_T}{|x(T) - x_T|}, \frac{\xi_n - x(T)}{|\xi_n - x(T)|} \right\rangle \le - \beta 
  \quad \forall n \in \mathbb{N}.
\end{equation}
By applying the property $(\mathcal{S})$ one obtains that there exists $C > 0$ (independent of $(x, u)$) such that for
any sufficiently large $n \in \mathbb{N}$ one can find $(x_n, u_n) \in A$ satisfying the inequality
\begin{equation} \label{SensitivityCond_Sequence}
  \sigma_n := \| u - u_n \|_q + \| x - x_n \|_{1, p} \le C |x(T) - x_n(T)|
\end{equation}
and such that $x_n(T) = \xi_n$. 

By the definition of rate of steepest descent one has
$$
  \varphi^{\downarrow}_A(x, u) 
  \le \liminf_{n \to \infty} \frac{\varphi(x_n, u_n) - \varphi(x, u)}{\sigma_n}.
$$
Taking into account the equality
$$
  \varphi(x_n, u_n) - \varphi(x, u) = 
  \left\langle \frac{x(T) - x_T}{|x(T) - x_T|}, \xi_n - x(T) \right\rangle + o(|\xi_n - x(T)|),
$$
where $o(|\xi_n - x(T)|) / |\xi_n - x(T)| \to 0$ as $n \to \infty$ and inequality \eqref{NegativeTangentAngle} one
obtains that
$$
  \varphi^{\downarrow}_A(x, u) \le \liminf_{n \to \infty} 
  \left( - \beta \frac{|\xi_n - x(T)|}{\sigma_n} + \frac{o(|\xi_n - x(T)|)}{\sigma_n} \right).
$$
By applying the inequality $|x(T) - x_n(T)| \le T^{1/p'} \| \dot{x} - \dot{x}_n \|_p \le T^{1/p'} \sigma_n$ one gets
that $o(|\xi_n - x(T)|) / \sigma_n \to 0$ as $n \to \infty$. Hence with the use of inequality
\eqref{SensitivityCond_Sequence} one obtains that $\varphi^{\downarrow}_A(x, u) \le - \beta / C$, and the proof is
complete.
\end{proof}

\begin{remark}
It is worth noting that in Example~\ref{Example_EndPoint_NotRelInt}, $(i)$ the functions $\theta$ and $f$ satisfy all
assumptions of Theorem~\ref{Theorem_FixedEndPointProblem_NonLinear}, $(ii)$ there exists a globally optimal solution,
$(iii)$ the set $\mathcal{R}(x_0, T)$ has the negative tangent angle property near $x_T$, and $(iv)$ the set $A$ is
bounded and the penalty function $\Phi_{\lambda}$ is bounded below on $A$ for any $\lambda \ge 0$. However, this penalty
function is not globally exact. Therefore, by Theorem~\ref{Theorem_FixedEndPointProblem_NonLinear} one can conclude
that in this example the property $(\mathcal{S})$ is not satisfied on 
$S_{\lambda_0}(c) \cap (\Omega_{\delta} \setminus \Omega)$ for any $\lambda_0 \ge 0$, $c > \mathcal{I}^*$, and 
$\delta > 0$, when $x_T = 0$. Arguing in a similar way to the proof of
Theorem~\ref{Theorem_Exactness_EvolutionEquations} and utilising Robinson's theorem one can check that 
the property $(\mathcal{S})$ is satisfied on $S_{\lambda_0}(c) \cap (\Omega_{\delta} \setminus \Omega)$ for some
$\lambda_0 \ge 0$, $c > \mathcal{I}^*$, and $\delta > 0$, provided $x_T \in \{ 0 \} \times (0, 1)$. Thus, although the
property $(\mathcal{S})$ might seem independent of the end-point $x_T$, the validity of this property on the set
$S_{\lambda_0}(c) \cap (\Omega_{\delta} \setminus \Omega)$ depends on the point $x_T$ and, in particular, its location
in the reachable set $\mathcal{R}(x_0, T)$.
\end{remark}

\subsection{Nonlinear Systems: Local Exactness}

Although Theorem~\ref{Theorem_FixedEndPointProblem_NonLinear} gives a general understanding of sufficient conditions
for the complete exactness of the penalty function $\Phi_{\lambda}$ for problem \eqref{FixedEndPointProblem}, its
assumptions cannot be readily verified for any particular problem. Therefore, it is desirable to have at least
verifiable sufficient conditions for the \textit{local} exactness of this penalty function. Our aim is to show a
connection between the local exactness of the penalty function $\Phi_{\lambda}$ for problem \eqref{FixedEndPointProblem}
and the complete controllability of the corresponding linearised system. This result serves as an illuminating example
of how one can apply Theorems~\ref{Theorem_LocalExactness} and \ref{Theorem_LocalErrorBound} to verify the local
exactness of a penalty function.

Recall that the linear system
\begin{equation} \label{ExactControllability_Def}
  \dot{x}(t) = A(t) x(t) + B(t) u(t)
\end{equation}
with $x \in \mathbb{R}^d$ and $u \in \mathbb{R}^m$ is called \textit{completely controllable} using $L^q$-controls in
time $T$, if for any initial state $x_0 \in \mathbb{R}^d$ and any finial state $x_T \in \mathbb{R}^d$ one can find 
$u \in L_q^m(0, T)$ such that there exists an absolutely continuous solution $x$ of \eqref{ExactControllability_Def}
with $x(0) = x_0$ defined on $[0, T]$ and satisfying the equality $x(T) = x_T$.

\begin{theorem} \label{Theorem_LocalExactness_TerminalConstraint}
Let $U = L_q^m(0, T)$, $q \ge p$, and $(x^*, u^*)$ be a locally optimal solution of problem
\eqref{FixedEndPointProblem}. Let also the following assumptions be valid:
\begin{enumerate}
\item{$\theta$ and $f$ are continuous, differentiable in $x$ in $u$, and the functions $\nabla_x \theta$, 
$\nabla_u \theta$, $\nabla_x f$, and $\nabla_u f$ are continuous;
\label{Assumpt_Smoothness_LocalEx_TerminConstr}}

\item{either $q = + \infty$ or $\theta$ and $\nabla_x \theta$ satisfy the growth condition of order $(q, 1)$, 
$\nabla_u \theta$ satisfies the growth condition of order $(q - 1, q')$, $f$ and $\nabla_x f$ satisfy the growth
condition of order $(q / p, p)$, and $\nabla_u f$ satisfies the growth condition of order $(q / s, s)$ with 
$s = qp / (q - p)$ in the case $q > p$, and $\nabla_u f$ does not depend on $u$ in the case $q = p$;
}

\item{the linearised system
$$
  \dot{h}(t) = A(t) h(t) + B(t) v(t)
$$
with $A(t) = \nabla_x f(x^*(t), u^*(t), t)$ and $B(t) = \nabla_u f(x^*(t), u^*(t), t)$ is completely controllable using
$L^q$-controls in time $T$. 
\label{Assumpt_LinearisedCompleteControllability}}
\end{enumerate}
Then the penalty function $\Phi_{\lambda}$ for problem \eqref{FixedEndPointProblem} is locally exact at $(x^*, u^*)$.
\end{theorem}

\begin{proof}
As was noted in the proof of Theorem~\ref{Theorem_FixedEndPointProblem_Linear}, the growth conditions on the function
$\theta$ and its derivatives ensure that the functional $\mathcal{I}$ is Lipschitz continuous on any bounded subset of
$X$ (in particular, in any bounded neighbourhood of $(x^*, u^*)$).

For any $(x, u) \in X$ define
$$
  F(x, u) = \begin{pmatrix} \dot{x}(\cdot) - f(x(\cdot), u(\cdot), \cdot) \\ x(T) \end{pmatrix},
  \quad K = \begin{pmatrix} 0 \\ x_T \end{pmatrix}.
$$
Our aim is to apply Theorem~\ref{Theorem_LocalErrorBound} with $C = \{ (x, u) \in X \mid x(0) = x_0 \}$ to the
operator $F$. Then one gets that there exists $a > 0$ such that
$$
  \dist(F(x, u), K) \ge a \dist( (x, u), F^{-1}(K) \cap C)
$$
for any $(x, u) \in C$ in a neighbourhood of $(x^*, u^*)$. Hence taking into account the facts that 
$\dist(F(x, u), K) = |x(T) - x_T| = \varphi(x, u)$ for any $(x, u) \in A$, and $F^{-1}(K) \cap C$ coincides with the
feasible set $\Omega$ of problem \eqref{FixedEndPointProblem} one obtains that
$\varphi(x) \ge a \dist((x, u), \Omega)$ for any $(x, u) \in A$ in a neighbourhood of $(x^*, u^*)$. Then by applying
Theorem~\ref{Theorem_LocalExactness} one obtains the desired result.

By Theorem~\ref{Theorem_DiffNemytskiiOperator} (see Appendix~B) the growth conditions on the
function $f$ and its derivatives guarantee that the nonlinear operator $F$ maps $X$ to 
$L^d_p(0, T) \times \mathbb{R}^d$, is strictly differentiable at $(x^*, u^*)$, and its Fr\'{e}chet derivative at this
point has the form
$$
  DF(x^*, u^*)[h, v] = \begin{pmatrix} \dot{h}(\cdot) - A(\cdot) h(\cdot) - B(\cdot) v(\cdot) \\ h(T) \end{pmatrix},
$$
where $A(t) = \nabla_x f(x^*(t), u^*(t), t)$ and $B(t) = \nabla_u f(x^*(t), u^*(t), t)$. Observe also that
$C - (x^*, u^*) = \{ (h, v) \in X \mid h(0) = 0 \}$ and $K - F(x^*, u^*) = (0, 0)^T$, since $x^*(0) = x_0$ and
$x^*(T) = x_T$ by definition. Consequently, the regularity condition \eqref{MetricRegCond} from
Theorem~\ref{Theorem_LocalErrorBound} takes the form: for any $\omega \in L^d_p(0, T)$ and $h_T \in \mathbb{R}^d$ there
exists $(h, v) \in X$ such that
\begin{equation} \label{MetricRegCond_TerminalConstraint}
  \dot{h}(t) = A(t) h(t) + B(t) v(t) + \omega(t) \quad \text{for a.e. } t \in (0, T), \quad
  h(0) = 0, \quad h(T) = h_T.
\end{equation}
Let us check that this condition holds true. Then by applying Theorem~\ref{Theorem_LocalErrorBound} we arrive at the
required result.

Fix any $\omega \in L^d_p(0, T)$ and $h_T \in \mathbb{R}^d$. Let $h_1$ be an absolutely continuous solution of the
equation $\dot{h}_1(t) = A(t) h_1(t) + \omega(t)$ with $h_1(0) = 0$ defined on $[0, T]$ (the existence of such solution
follows from \cite[Theorem~1.1.3]{Filippov}). From the fact that $\nabla_x f$ satisfies the growth condition of order 
$(q / p, p)$ in the case $q < + \infty$ it follows that $A(\cdot) \in L_p^{d \times d}(0, T)$ (in the case 
$q = + \infty$ one obviously has $A(\cdot) \in L_{\infty}^{d \times d}(0, T)$). Hence $h_1 \in W_{1, p}^d(0, T)$, since
$h_1$ is absolutely continuous and the right-hand side of the equality $\dot{h}_1(t) = A(t) h_1(t) + \omega(t)$ belongs 
to $L^d_p(0, T)$.

Let $v \in L_q^m(0, T)$ be such that an absolutely continuous solution $h_2$ of the system 
$\dot{h}_2(t) = A(t) h_2(t) + B(t) v(t)$ with $h_2(0) = 0$ satisfies the equality $h_2(T) = - h_1(T) + h_T$. Note that
such $v$ exists due to the complete controllability assumption. By applying the fact that $\nabla_u f$ satisfies the
growth condition of order $(q / s, s)$ one obtains that $B(\cdot) \in L_s^{d \times m}(0, T)$ in the case 
$p < q < + \infty$, which with the use of H\"{o}lder inequality implies that $B(\cdot) v(\cdot) \in  L_p^d(0, T)$ (in
the case $q = + \infty$ one obviously has $B(\cdot) v(\cdot) \in  L_{\infty}^d(0, T)$, while in the case 
$p = q < + \infty$ one has $B(\cdot) \in L_{\infty}^{d \times m}(0, T)$, since $\nabla_u f$ does not depend on $u$, and 
$B(\cdot) v(\cdot) \in  L_p^d(0, T)$). Therefore $h_2 \in W_{1, p}^d(0, T)$ by virtue of the fact that the right-hand
side of $\dot{h}_2(t) = A(t) h_2(t) + B(t) v(t)$ belongs to $L^d_p(0, T)$. It remains to note that the pair 
$(h_1 + h_2, v)$ belongs to $X$ and satisfies \eqref{MetricRegCond_TerminalConstraint}.
\end{proof}

\begin{remark}
From the proof of Theorem~\ref{Theorem_LocalExactness_TerminalConstraint} it follows that under the assumption of this
theorem the penalty function
$$
  \Psi_{\lambda}(x, u) = \mathcal{I}(x, u) + \lambda \bigg[ |x(T) - x_T| 
  + \Big( \int_0^T \big| \dot{x}(t) - f(x(t), u(t), t) \big|^p \, dt \Big)^{1 / p} \bigg]
$$
is locally exact at $(x^*, u^*)$, i.e. $(x^*, u^*)$ is a point of local minimum of this penalty function on the set 
$\{ (x, u) \in X \mid x(0) = x_0 \}$ for any sufficiently large $\lambda$.
\end{remark}

\begin{remark}
Let us note that Theorem~\ref{Theorem_LocalExactness_TerminalConstraint} can be extended to the case of problems with
convex constraints on control inputs, but in this case the complete controllability assumption must be replaced by a
much more restrictive assumption. Namely, let $U \subset L_q^m(0, T)$ be a closed convex set, $(x^*, u^*)$ be a locally
optimal solution of problem \eqref{FixedEndPointProblem}, and $\dot{h}(t) = A(t) h(t) + B(t) v(t)$ be the corresponding
linearised system. Define $C = \{ (x, u) \in X \mid x(0) = x_0, \: u \in U \}$ and $K = (0, x_T)^T$. 
One can easily see that in this case the regularity condition \eqref{MetricRegCond} takes the form: for any 
$\omega \in L^d_p(0, T)$ and $h_T \in \mathbb{R}^d$ there exists $(h, v) \in X$ such that $v \in \cone(U - u^*)$ and
$$
  \dot{h}(t) = A(t) h(t) + B(t) v(t) + \omega(t) \quad \text{for a.e. } t \in (0, T), \quad
  h(0) = 0, \quad h(T) = h_T,
$$
where $\cone(U - u^*) = \bigcup_{\alpha \ge 0} \alpha(U - u^*)$ is the cone generated by the set $U - u^*$. If 
$u^* \in \interior U$, then $\cone(U - u^*) = L^m_q(0, T)$, and this regularity condition is equivalent to the
complete controllability of the linearised system. However, if $u^* \notin \interior U$, then one must suppose that for
any initial state $x_0 \in \mathbb{R}^d$ and for any finial state $x_T \in \mathbb{R}^d$ one can find 
$u \in \cone(U - u^*)$ such that there exists an absolutely continuous solution $x$ of \eqref{ExactControllability_Def}
with $x(0) = x_0$ defined on $[0, T]$ and satisfying the equality $x(T) = x_T$, i.e. the linearised system must be
completely controllable using control inputs from $\cone(U - u^*)$. If this assumption is satisfied, then arguing in the
same way as in the proof of Theorem~\ref{Theorem_LocalExactness_TerminalConstraint} one can verify that the penalty
function $\Phi_{\lambda}$ for problem \eqref{FixedEndPointProblem} is locally exact at $(x^*, u^*)$. 
\end{remark}

It should be noted that the complete controllability of the linearised system is \textit{not} necessary for the local
exactness of the penalty function $\Phi_{\lambda}(x, u)$, as the following simple example shows.

\begin{example}
Let $d = m = 1$ and $p = q = 2$. Consider the following fixed-endpoint optimal control problem:
\begin{equation} \label{Ex_DegenerateLinearisation}
  \min \mathcal{I}(u) = - \int_0^T u(t)^2 dt \quad 
  \text{s.t.} \quad \dot{x}(t) = x(t) + u(t)^2, \quad t \in [0, T], \quad x(0) = x(T) = 0, \quad u \in L^2(0, T).
\end{equation}
Solving the differential equation one obtains that $x(t) = \int_0^t e^{t - \tau} u(\tau)^2 d \tau$ for all 
$t \in [0, T]$, which implies that the only feasible point of this problem is $(x^*, u^*)$ with $x^*(t) \equiv 0$ and
$u^*(t) = 0$ for a.e. $t \in [0, T]$. Thus, $(x^*, u^*)$ is a globally optimal solution of this problem. The linearised
system at this point has the form $\dot{h} = h$. Clearly, it is not completely controllable, which renders
Theorem~\ref{Theorem_LocalExactness_TerminalConstraint} inapplicable. Let us show that, nevertheless, the penalty
function $\Phi_{\lambda}$ for problem \eqref{Ex_DegenerateLinearisation} is globally exact.

Indeed, in this case the penalised problem has the form
$$
  \min \Phi_{\lambda}(x, u) = - \int_0^T u(t)^2 dt + \lambda |x(T)| \quad 
  \text{s.t.} \quad \dot{x}(t) = x(t) + u(t)^2, \quad t \in [0, T], \quad x(0) = 0, \quad u \in U.
$$
With the use of the fact that $x(t) = \int_0^t e^{t - \tau} u(\tau)^2 d \tau$ one gets that
$$
  \Phi_{\lambda}(x, u) = - \int_0^T u(t)^2 dt + \lambda \int_0^T e^{T-t} u(t)^2 \, dt
  \ge - \int_0^T u(t)^2 dt + \lambda \int_0^T u(t)^2 \, dt
$$
for any $u \in U$. Therefore, for all $\lambda \ge 1$ one has $\Phi_{\lambda}(x, u) \ge 0 = \Phi_{\lambda}(x^*, u^*)$
for any feasible point of the penalised problem, i.e. the penalty function $\Phi_{\lambda}$ for problem
\eqref{Ex_DegenerateLinearisation} is globally exact.
\end{example}

\begin{remark}
As Theorem~\ref{Theorem_LocalExactness_TerminalConstraint} demonstrates, the local exactness of the penalty function
$\Phi_{\lambda}$ for problem \eqref{FixedEndPointProblem} is implied by the complete controllability of the
corresponding linearised system. It should be noted that a similar result can be proved in the case of complete
exactness, but one must assume some sort of \textit{uniform} complete controllability of linearised systems.

A definition of uniform complete controllability can be be given in the following way. With the use of the open mapping
theorem (see~\cite[formula~$(0.2)$]{Ioffe}) one can check that if the linear system
\begin{equation} \label{LinSys_UniformExactControllability}
  \dot{x}(t) = A(t) x(t) + B(t) u(t),
\end{equation}
is completely controllable using $L^q$-controls, then there exists $C > 0$ such that for any $x_T$ one can find 
$u \in L^m_q(0, T)$ with $\| u \|_q \le C |x_T|$ such that for the corresponding solution $x(\cdot)$ of 
\eqref{LinSys_UniformExactControllability} with $x(0) = 0$ one has $x(T) = x_T$. In other words, one can steer the
state of system \eqref{LinSys_UniformExactControllability} from the origin to any point $x_T$ in time $T$ with the use
of a control input whose $L^q$-norm is proportional to $|x_T|$. Denote the greatest lower bound of all such $C$ by 
$C_T(A(\cdot), B(\cdot))$. In the case when system \eqref{LinSys_UniformExactControllability} is not completely
controllable we put $C_T(A(\cdot), B(\cdot)) = + \infty$. Then one can say that the nonlinear system
$\dot{x} = f(x, u, t)$ is \textit{uniformly completely controllable in linear approximation} on a set $K \subseteq A$,
if there exists $C > 0$ such that
$C_T( \nabla_x f(x(\cdot), u(\cdot), \cdot), \nabla_u f(x(\cdot), u(\cdot), \cdot) ) \le C$ for any $(x, u) \in K$.
With the use of general results on nonlocal metric regularity\cite{Dmitruk} one can check that under some
natural assumptions on the function $f$ uniform complete controllability in linear approximation on a set 
$K \subseteq A$ guarantees that the property $(\mathcal{S})$ is satisfied on this set, provided $U = L^m_q(0, T)$.
Hence, by applying Theorem~\ref{Theorem_FixedEndPointProblem_NonLinear} one can prove that uniform complete
controllability in linear approximation of the nonlinear system $\dot{x} = f(x, u, t)$ implies that the penalty function
$\Phi_{\lambda}$ for problem \eqref{FixedEndPointProblem} is completely exact. A detailed proof of this result lies
beyond the scope of this paper, and we leave it to the interested reader.
\end{remark}

\subsection{Variable-Endpoint Problems}

Let us briefly outline how the main results on the exact penalisation of terminal constraints from previous sections
(in particular, Theorems~\ref{Theorem_LocalExactness_TerminalConstraint} and \ref{Theorem_FixedEndPointProblem_Linear})
can be extended to the case of variable-endpoint problems of the form
\begin{equation} \label{VariableEndPointPenaltyProblem}
\begin{split}
  &\min \: \mathcal{I}(x, u) = \int_0^T \theta(x(t), u(t) t) \, dt + \zeta(x(T)) \quad 
  \text{subject to} \quad \dot{x}(t) = f(x(t), u(t), t), \quad t \in [0, T], \\
  &x(0) = x_0, \quad g_i(x(T)) \le 0 \quad \forall i \in I, \quad g_k(x(T)) = 0 \quad \forall k \in J,
  \quad u \in U.
\end{split}
\end{equation}
Here $\theta \colon \mathbb{R}^d \times \mathbb{R}^m \times [0, T] \to \mathbb{R}$, 
$\zeta \colon \mathbb{R}^d \to \mathbb{R}$, $f \colon \mathbb{R}^d \times \mathbb{R}^m \times [0, T] \to \mathbb{R}^d$,
and $g_i \colon \mathbb{R}^d \to \mathbb{R}$ are given functions, $i \in I \cup J$, $I = \{ 1, \ldots, l_1 \}$,
$J = \{ l_1 + 1, \ldots, l_2 \}$, $x_0 \in \mathbb{R}^d$, and $T > 0$ are fixed, $x \in W^d_{1, p}(0, T)$, 
$U \subseteq L_q^m(0, T)$ is a nonempty closed set, and $1 \le p, q \le + \infty$.

We penalise only the endpoint constraints. To this end, define $X = W_{1, p}^d(0, T) \times L_q^m(0, T)$, 
$M = \{ (x, u) \in X \mid g_i(x(T)) \le 0, \: i \in I, \: g_k(x(T)) = 0, \: k \in J \}$, and
$$
  A = \Big\{ (x, u) \in X \Bigm| x(0) = x_0, \: u \in U, \:
  \dot{x}(t) = f(x(t), u(t), t) \text{ for a.e. } t \in [0, T] \Big\}.
$$
Then problem \eqref{VariableEndPointPenaltyProblem} can be rewritten as the problem of minimising $\mathcal{I}(x, u)$
subject to $(x, u) \in M \cap A$. Define 
$$
  \varphi(x, u) = \sum_{i \in I} \max\{ g_i(x(T)), 0 \} + \sum_{k \in J} |g_k(x(T))|.
$$
Then $M = \{ (x, u) \in X \mid \varphi(x, u) = 0 \}$, and one can consider the penalised problem
\begin{equation*}
\begin{split}
  {}&\min \: \Phi_{\lambda}(x, u)
  = \int_0^T \theta(x(t), u(t), t) \, dt + \zeta(x(T))
  + \lambda \Big( \sum_{i \in I} \max\{ g_i(x(T)), 0 \} + \sum_{k \in J} |g_k(x(T))| \Big) \\
  {}&\text{subject to } \dot{x}(t) = f(x(t), u(t), t), \quad t \in [0, T], \quad
  x(0) = x_0, \quad u \in U,
\end{split}
\end{equation*}
which is a nonlinear free-endpoint optimal control problem.

For any $x \in \mathbb{R}^d$ denote $I(x) = \{ i \in I \mid g_i(x) = 0 \}$. Let the functions $g_i$, $i \in I \cup J$,
be differentiable. Recall that one says that the Mangasarian-Fromovitz constraint qualifications (MFCQ) holds at a point
$x_T \in \mathbb{R}^d$, if the gradients $\nabla g_k(x_T)$, $k \in J$, are linearly independent, and there exists 
$h \in \mathbb{R}^d$ such that $\langle \nabla g_k(x_T), h \rangle = 0$ for any $k \in J$, and
$\langle \nabla g_i(x_T), h \rangle < 0$ for any $i \in I(x_T)$. Let us show that the complete controllability of the
linearised system along with MFCQ guarantee the local exactness of the penalty function $\Phi_{\lambda}$ for problem
\eqref{VariableEndPointPenaltyProblem}

\begin{theorem}
Let $U = L_q^m(0, T)$, $q \ge p$, and $(x^*, u^*)$ be a locally optimal solution of problem
\eqref{VariableEndPointPenaltyProblem}. Suppose also that
assumptions~\ref{Assumpt_Smoothness_LocalEx_TerminConstr}--\ref{Assumpt_LinearisedCompleteControllability} of
Theorem~\ref{Theorem_LocalExactness_TerminalConstraint} are satisfied, $\zeta$ is locally Lipschitz continuous, 
the functions $g_i$, $i \in I \cup J$ are continuously differentiable in a neighbourhood of $x^*(T)$, and MFCQ holds
true at the point $x^*(T)$. Then the penalty function $\Phi_{\lambda}$ for problem
\eqref{VariableEndPointPenaltyProblem} is locally exact at $(x^*, u^*)$.
\end{theorem}

\begin{proof}
For any $(x, u) \in X$ define
$$
  F(x, u) = \begin{pmatrix} \dot{x}(\cdot) - f(x(\cdot), u(\cdot), \cdot) \\ g(x(T)) \end{pmatrix},
  \quad K = \begin{pmatrix} 0 \\ \mathbb{R}_{-}^{l_1} \times \{ \mathbf{0}_{l_2 - l_1} \} \end{pmatrix},
$$
where $g(\cdot) = (g_1(\cdot), \ldots, g_{l_2}(\cdot))^T$, $\mathbb{R}_{-} = (- \infty, 0]$ and
$\mathbf{0}_{l_2 - l_1}$ is the zero vector of dimension $l_2 - l_1$. Let us apply Theorem~\ref{Theorem_LocalErrorBound}
with $C = \{ (x, u) \in X \mid x(0) = x_0 \}$ to the operator $F$. Then arguing in the same way as in the proof of
Theorem~\ref{Theorem_LocalExactness_TerminalConstraint} one arrives at the required result.

With the use of Theorem~\ref{Theorem_DiffNemytskiiOperator} (see Appendix~B) one obtains that
the nonlinear operator $F$ maps $X$ to  $L^d_p(0, T) \times \mathbb{R}^{l_2}$, is strictly differentiable at 
$(x^*, u^*)$, and its Fr\'{e}chet derivative at this point has the form
$$
  DF(x^*, u^*)[h, v] = 
  \begin{pmatrix} \dot{h}(\cdot) - A(\cdot) h(\cdot) - B(\cdot) v(\cdot) \\ \nabla g(x^*(T)) h(T) \end{pmatrix}, \quad
  A(t) = \nabla_x f(x^*(t), u^*(t), t), \quad B(t) = \nabla_u f(x^*(t), u^*(t), t).
$$
Hence bearing in mind the fact that $C - (x^*, u^*) = \{ (h, v) \in X \mid h(0) = 0 \}$, since $x^*(0) = x_0$, one gets
that the regularity condition \eqref{MetricRegCond} from Theorem~\ref{Theorem_LocalErrorBound} takes the form 
$0 \in \core K(x^*, u^*)$ with
\begin{equation} \label{VariableEndPoint_RegularityCone}
  K(x^*, u^*) = \left\{ \begin{pmatrix} \dot{h}(\cdot) - A(\cdot) h(\cdot) - B(\cdot) v(\cdot) \\ 
  \nabla g(x^*(T)) h(T) \end{pmatrix} -
  \begin{pmatrix} 0 \\ K_0 \end{pmatrix} \Biggm| (h, v) \in X, \: h(0) = 0 \right\},
\end{equation}
where $K_0 = (\mathbb{R}_{-}^{l_1} - (g_1(x^*(T)), \ldots, g_{l_1}(x^*(T)))^T \times \{ \mathbf{0}_{l_2 - l_1} \}$. Let
us check that this condition is satisfied.

Indeed, denote $g_J(\cdot) = (g_{l_1 + 1}(\cdot), \ldots g_{l_2}(\cdot))^T$. By MFCQ the matrix $\nabla g_J(x^*(T))$ has
full row rank. Therefore, by the open mapping theorem (see~\cite[formula~$(0.2)$]{Ioffe}) there exists $\eta > 0$ such
that for any $y_2 \in \mathbb{R}^{l_2 - l_1}$ one can find $h_2 \in \mathbb{R}^d$ with $|h_2| \le \eta |y_2|$
satisfying the equality $\nabla g_J(x^*(T)) h_2 = y_2$.

Fix any $\omega \in L^d_p(0, T)$, $r_2 > 0$, and $y_2 \in B(\mathbf{0}_{l_2 - l_1}, r_2)$. Then there exists 
$h_2 \in \mathbb{R}^d$ with $|h_2| \le \eta r_2$ such that $\nabla g_J(x^*(T)) h_2 = y_2$. By MFCQ there exists 
$h_1 \in \mathbb{R}^d$ such that $\nabla g_J(x^*(T)) h_1 = 0$ and 
$\langle \nabla g_i(x^*(T)), h_1 \rangle < 0$ for any $i \in I(x^*(T))$. Taking into account the fact that 
$g_i(x^*(T)) < 0$ for any $i \notin I(x^*(T))$ one obtains that there exists $\alpha > 0$ such that
$\langle \nabla g_i(x^*(T)), \alpha h_1 \rangle + g_i(x^*(T)) < 0$ for all $i \in I$. Furthermore, decreasing
$r_2$, if necessary, one can suppose that 
\begin{equation} \label{VariableEndpoint_DecayDirection}
  \max_{|h| \le \eta r_2} \langle \nabla g_i(x^*(T)), \alpha h_1 + h \rangle + g_i(x^*(T)) < 0 \quad \forall i \in I.
\end{equation}
Observe also that $\nabla g_J(x^*(T))(\alpha h_1 + h_2) = y_2$.

Denote $h_T = \alpha h_1 + h_2$. Arguing in the same way as in the proof of
Theorem~\ref{Theorem_LocalExactness_TerminalConstraint} and utilising the complete controllability assumption one can
verify that there exists $(h, v) \in X$ such that
$$
  \dot{h}(t) = A(t) h(t) + B(t) v(t) + \omega(t) \quad \text{for a.e. } t \in (0, T), \quad
  h(0) = 0, \quad h(T) = h_T.
$$
Consequently, one has 
$L^d_p(0, T) \times (r_1, + \infty)^{l_1} \times B(\mathbf{0}_{l_2 - l_1}, r_2) \subset K(x^*, u^*)$, where
$$
  r_1 = \max_{i \in I} \Big(
  \max_{|h| \le \eta r_2} \langle \nabla g_i(x^*(T)), \alpha h_1 + h \rangle + g_i(x^*(T)) \Big) < 0
$$ 
(see~\eqref{VariableEndPoint_RegularityCone} and \eqref{VariableEndpoint_DecayDirection}). 
Thus, $0 \in \core K(x^*, u^*)$, and the proof is complete.
\end{proof}

Let us now show how one can extend Theorem~\ref{Theorem_FixedEndPointProblem_Linear} to the case of variable-endpoint
problems. Theorem~\ref{Theorem_Exactness_EvolutionEquations} can be extended to the case of variable-endpoint problems
for linear evolution equations in a similar way. Let, as in the proof of the previous theorem, 
$g_J(\cdot) = (g_{l_1 + 1}(\cdot), \ldots g_{l_2}(\cdot))^T$, and denote by
$\mathcal{R}_I(x_0, T) = \{ \xi \in \mathcal{R}(x_0, T) \mid g_i(\xi) \le 0, \: i \in I \}$ the set of all those
reachable points that satisfy the terminal inequality constraints.

\begin{theorem}
Let $q \ge p$, the functions $g_i$, $i \in I$ and the set $U$ be convex, the functions $g_k$, $k \in J$ be affine, and
the following assumptions be valid:
\begin{enumerate}
\item{$f(x, u, t) = A(t) x + B(t) u$ for some $A(\cdot) \in L_{\infty}^{d \times d}(0, T)$ and 
$B(\cdot) \in L_{\infty}^{d \times m}(0, T)$;
}

\item{the function $\zeta$ is locally Lipschitz continuous, the function $\theta = \theta(x, u, t)$ is continuous,
differentiable in $x$ and $u$, and the functions $\nabla_x \theta$ and $\nabla_u \theta$ are continuous;
}

\item{either $q = + \infty$ or the functions $\theta$ and $\nabla_x \theta$ satisfy the growth condition of order 
$(q, 1)$, while the function $\nabla_u \theta$ satisfies the growth condition of order $(q - 1, q')$;
}

\item{there exists a globally optimal solution of problem \eqref{VariableEndPointPenaltyProblem}, and the following
Slater condition holds true: $0 \in \relint g_J(\mathcal{R}_I(x_0, T))$ and there exists a feasible point 
$(\widehat{x}, \widehat{u}) \in \Omega$ such that $g_i(\widehat{x}(T)) < 0$ for all $i \in I$;
}

\item{there exist $\lambda_0 > 0$, $c > \mathcal{I}^*$ and $\delta > 0$ such that the set 
$S_{\lambda_0}(c) \cap \Omega_{\delta}$ is bounded in $W^d_{1, p}(0, T) \times L_q^m(0, T)$, and the function
$\Phi_{\lambda_0}(x, u)$ is bounded below on $A$.
}
\end{enumerate}
Then there exists $\lambda^* \ge 0$ such that for any $\lambda \ge \lambda^*$ the penalty function $\Phi_{\lambda}$ for
problem \eqref{VariableEndPointPenaltyProblem} is completely exact on $S_{\lambda}(c)$.
\end{theorem}

\begin{proof}
Arguing in almost the same way as in the proof of Theorem~\ref{Theorem_FixedEndPointProblem_Linear} and utilising
Theorem~\ref{Theorem_CompleteExactness} one obtains that it is sufficient to check that there exists $a > 0$ such that
$\varphi^{\downarrow}_A(x, u) \le - a$ for all $(x, u) \in S_{\lambda_0}(c) \cap (\Omega_{\delta} \setminus \Omega)$.
Fix any such $(x, u)$, and define $I_+(x, u) = \{ i \in I \mid g_i(x(T)) > 0 \}$. Let us consider two cases.

\textbf{Case~I.} Suppose that $I_+(x, u) \ne \emptyset$. Define 
$(\Delta x, \Delta u) = (\widehat{x} - x, \widehat{u} - u)$, where $(\widehat{x}, \widehat{u})$ is from Slater's
condition. Observe that $(x + \alpha \Delta x, u + \alpha \Delta u) \in A$ for any $\alpha \in [0, 1]$ due to the
convexity of the set $U$ and the linearity of the system. Furthermore, by virtue of the convexity of the functions
$g_i$, $i \in I$, for any $\alpha \in [0, 1]$ one has
\begin{equation} \label{ConvexityTerminalConstr}
  g_i(x(T) + \alpha \Delta x(T)) \le \alpha g_i(\widehat{x}(T)) + (1 - \alpha) g_i(x(T))
  \le \alpha \eta + (1 - \alpha) g_i(x(T)), \quad
  \eta = \max_{i \in I} g_i(\widehat{x}(T)) < 0.
\end{equation}
Consequently, for any $i \notin I_+(x, u)$ one has $g_i(x(T) + \alpha \Delta x(T)) < 0$ for all $\alpha \in [0, 1]$ by
\eqref{ConvexityTerminalConstr}, while for any $i \in I_+(x, u)$ one has $g_i(x(T) + \alpha \Delta x(T)) \ge 0$ for any
sufficiently small $\alpha$ due to the fact that a convex function defined on a finite dimensional space is continuous
in the interior of its effective domain (see, e.g. \cite[Theorem~3.5.3]{IoffeTihomirov}). Moreover,
$g_J(x(T) + \alpha \Delta x(T)) = (1 - \alpha) g_J(x(T))$, since the functions $g_k$, $k \in J$, are affine and 
$g_J(\widehat{x}(T)) = 0$ (recall that $(\widehat{x}, \widehat{u}) \in \Omega$). Hence with the use of
\eqref{ConvexityTerminalConstr} one obtains that
\begin{multline*}
  \varphi^{\downarrow}_A(x, u) \le \liminf_{\alpha \to +0} 
  \frac{\varphi(x + \alpha \Delta x, u + \alpha \Delta u)}{\alpha \| (\Delta x, \Delta u) \|_X} \\
  = \frac{1}{\| (\Delta u, \Delta x) \|_X} \liminf_{\alpha \to + 0}
  \frac{1}{\alpha} \bigg( \sum_{i \in I_+(x(T))} \big( g_i(x(T) + \alpha \Delta x(T)) - g_i(x(T)) \big) 
  + \sum_{k \in J} \big( (1 - \alpha)|g_k(x(T))| - |g_k(x(T))| \big) \bigg) \\
  \le \frac{1}{\| (\Delta u, \Delta x) \|_X} \Big( \sum_{i \in I_+(x(T))} \big( \eta - g_i(x(T)) \big)
  - \sum_{k \in J} |g_k(x(T))| \Big)
  \le \frac{\eta}{\| (\Delta u, \Delta x) \|_X}.
\end{multline*}
From the fact that the set $S_{\lambda_0}(c) \cap \Omega_{\delta}$ is bounded it follows that there exists $C > 0$
(independent of $(x, u)$) such that $\| (\Delta u, \Delta x \|_X = \| (\widehat{x} - x, \widehat{u} - u) \|_X \le C$.
Thus, $\varphi^{\downarrow}_A(x, u) \le \eta / C < 0$, and the proof of the first case is complete.

\textbf{Case~II.} Let now $I_+(x, u) = \emptyset$. Note that $g_J(x(T)) \ne 0$, since $(x, u) \notin \Omega$. Choose
any  $(\widetilde{x}, \widetilde{u}) \in \Omega$ and define 
$(\Delta x, \Delta u) = (\widetilde{x} - x, \widetilde{u} - u)$. Then, as in the first case, for any 
$\alpha \in [0, 1]$ one has $(x + \alpha \Delta x, u + \alpha \Delta u) \in A$, and
$$
  g_i(x(T) + \alpha \Delta x(T)) \le \alpha g_i(\widetilde{x}(T)) + (1 - \alpha) g_i(x(T)) \le 0
  \quad \forall i \in I
$$
due to the convexity of the functions $g_i$ and the fact that $I_+(x, u) = \emptyset$. In addition,
$g_J(x(T) + \alpha \Delta x(T)) = (1 - \alpha) g_J(x(T))$ for all $\alpha \in [0, 1]$, since the functions $g_k$, 
$k \in J$, are affine and $(\widetilde{x}, \widetilde{u}) \in \Omega$. Therefore
\begin{align*}
  \varphi^{\downarrow}_A(x, u) \le \liminf_{\alpha \to +0}
  \frac{\varphi(x + \alpha \Delta x, u + \alpha \Delta u)}{\alpha \| (\Delta x, \Delta u) \|_X}
  &= \frac{1}{\| (\Delta u, \Delta x) \|_X} \liminf_{\alpha \to + 0}
  \sum_{k \in J} \Big( (1 - \alpha)|g_k(x(T))| - |g_k(x(T))| \Big) \\
  &= - \frac{1}{\| (\Delta u, \Delta x) \|_X} \sum_{k \in J} |g_k(x(T))| \le 
  - \frac{|g_J(x(T))|}{\| (\Delta u, \Delta x) \|_X} < 0.
\end{align*}
Thus, it remains to show that there exists $C > 0$ such that for any 
$(x, u) \in S_{\lambda_0}(c) \cap (\Omega_{\delta} \setminus \Omega)$ such that $I_+(x, u) = \emptyset$ one can find
$(\widetilde{x}, \widetilde{u}) \in \Omega$ satisfying the inequality 
\begin{equation} \label{ErrorBoundAffineTerminalConstrWeak}
  \| (\Delta x, \Delta u) \|_X = \| x - \widetilde{x} \|_{1, p} + \| u - \widetilde{u} \|_q
  \le C |g_J(x(T))|.
\end{equation}
Then $\varphi^{\downarrow}_A(x, u) \le - 1 / C$, and the proof is complete.

As was shown in the proof of Theorem~\ref{Theorem_FixedEndPointProblem_Linear}, there exists $L > 0$ (depending only on
$A(\cdot)$, $B(\cdot)$, $T$, $p$, and $q$) such that $\| x_1 - x_2 \|_{1, p} \le L \| u_1 - u_2 \|_q$ for any 
$(x_1, u_1), (x_2, u_2) \in A$. Therefore, instead of \eqref{ErrorBoundAffineTerminalConstrWeak} it sufficient to prove
the validity of the inequality 
\begin{equation} \label{ErrorBoundAffineTerminalConstr}
  \| u - \widetilde{u} \|_q \le C |g_J(x(T))|.
\end{equation}
Moreover, from \eqref{SobolevImbedding} and the inequality $\| x_1 - x_2 \|_{1, p} \le L \| u_1 - u_2 \|_q$ it
follows that the map $u \mapsto x_u(T)$ is Lipschitz continuous, where $x_u$ is a solution of the system 
$\dot{x}_u(\cdot) = A(\cdot) x_u(\cdot) + B(\cdot) u(\cdot)$ such that $x_u(0) = x_0$. Hence with the use of the fact
that the functions $g_i$ are convex and continuous one obtains that the set 
$U_I = \{ u \in U \mid g_i(x_u(T)) \le 0, \: i \in I \}$ is closed and convex.

Let us prove inequality \eqref{ErrorBoundAffineTerminalConstr} with the use of Robinson's theorem. Note that the
function $g_J(\cdot) - g_J(0)$ is linear, since the functions $g_k$, $k \in J$, are affine. Define the linear operator
$\mathcal{T} \colon L^m_q(0, T) \to \mathbb{R}^{l_2 - l_1}$, $\mathcal{T} v = g_J(h(T)) - g_J(0)$, where $h$ is a
solution of the differential equation
$$
  \dot{h}(t) = A(t) h(t) + B(t) v(t), \quad h(0) = 0, \quad t \in [0, T].
$$
As was shown in the proof of Theorem~\ref{Theorem_FixedEndPointProblem_Linear}, the mapping $v \mapsto h(T)$ is
continuous, which implies that the linear operator $\mathcal{T}$ is bounded.

Fix any $(x_*, u_*) \in \Omega$. By definition for all $(x, u) \in A$ one has 
$\dot{x}(t) - \dot{x}_*(T) = A(t) (x(t) - x_*(t)) + B(t) (u(t) - u_*(t))$ for a.e. $t \in [0, T]$, $x(0) - x_*(0) = 0$,
$g_J(x_*(T)) = 0$, and 
$$
  \mathcal{T}(u - u_*) = g_J(x(T) - x_*(T)) - g_J(0)
  = g_J(x(T)) - g_J(0) - (g_J(x_*(T)) - g_J(0)) = g_J(x(T)). 
$$
Therefore $g_J(\mathcal{R}_I(x_0, T)) = \mathcal{T}(U_I - u_*)$. Define $X_0 = \cl\linhull(U_I - u_*)$ and
$Y_0 = \linhull \mathcal{T}(U_I - u_*)$. Note that $Y_0$ is a closed subspace of $\mathbb{R}^{l_2 - l_1}$ and
$\mathcal{T}(X_0) = Y_0$. Finally, introduce the operator $\mathcal{T}_0 \colon X_0 \to Y_0$,
$\mathcal{T}_0(v) = \mathcal{T}(v)$ for all $v \in X_0$. Clearly, $\mathcal{T}_0$ is a bounded linear operator
between Banach spaces. Moreover, by Slater's condition 
$0 \in \relint g_J(\mathcal{R}_I(x_0, T)) = \interior \mathcal{T}(U_I - u_*)$.
Consequently, by Robinson's theorem (Theorem~\ref{Theorem_Robinson_Ursescu} with $C = U_I - u_*$, $x^* = 0$, and 
$y = y^* = 0$) there exists $\kappa > 0$ such that
$$
  \dist\big( u - u_*, \mathcal{T}_0^{-1}(0) \cap (U_I - u_*) \big) \le 
  \kappa \big( 1 + \| u - u_* \|_q \big) \Big| \mathcal{T}_0(u - u_*) \Big|
  \quad \forall u \in U_I.
$$
Recall that $\mathcal{T}_0(u - u_*) = g_J(x(T))$, since $(x_*, u_*) \in \Omega$. Thus, for any 
$(x, u) \in S_{\lambda_0}(c) \cap (\Omega_{\delta} \setminus \Omega)$ such that $I_+(x, u) = \emptyset$ one can find
$v \in U_I - u_*$ such that $\mathcal{T}_0(v) = 0$ and
$$
  \big\| u - u_* - v \big\|_q \le 2 \kappa \big( 1 + \| u - u_* \|_q \big) \big| g_J(x(T)) \big|.
$$
Define $\widetilde{u} = u_* + v$, and let $\widetilde{x}$ be the corresponding solution of the original system, i.e.
$(\widetilde{x}, \widetilde{u}) \in A$. Then $g_i(\widetilde{x}(T)) \le 0$ for all $i \in I$, since 
$\widetilde{u} \in U_I$, and $g_J(\widetilde{x}(T)) = \mathcal{T} (\widetilde{u} - u_*) = \mathcal{T}(v) = 0$,
i.e. $g_J(\widetilde{x}(T)) = 0$ and $(\widetilde{x}, \widetilde{u}) \in \Omega$. Moreover, one has
$\| u - \widetilde{u} \|_q \le 2 \kappa ( 1 + \| u - u_* \|_q ) | g_J(x(T))|$. By our assumption the set
$S_{\lambda_0}(c) \cap (\Omega_{\delta} \setminus \Omega)$ is bounded. Therefore there exists
$C > 0$ such that $2 \kappa ( 1 + \| u - u_* \|_q ) \le C$ for any 
$(x, u) \in S_{\lambda_0}(c) \cap (\Omega_{\delta} \setminus \Omega)$ such that $I_+(x, u) = \emptyset$. Thus, for all
such $(x, u)$ there exists $(\widetilde{x}, \widetilde{u}) \in \Omega$ satisfying
\eqref{ErrorBoundAffineTerminalConstr}, and the proof is complete.
\end{proof}

\begin{remark}
Note that in the case when there are no equality constraints Slater's condition takes an especially simple form.
Namely, it is sufficient suppose that there exists a feasible point $(\widehat{x}, \widehat{u}) \in \Omega$ such that
$g_i(\widehat{x}(T)) < 0$ for all $i \in I$.
\end{remark}

\begin{remark}
Let us briefly discuss how one can extend Theorem~\ref{Theorem_FixedEndPointProblem_NonLinear} to the case of nonlinear
variable-endpoint problems. In the case when there are no equality constraints and the inequality constraints are
differentiable, one has to replace the negative tangent angle property with the assumption that there exist 
$\delta > 0$ and $\beta > 0$ such that for any $\xi \in \mathcal{R}(x_0, T)$ satisfying the inequalities 
$0 < \sum_{i \in I} \max\{ g_i(\xi), 0 \} < \delta$ one can find a sequence $\{ \xi_n \} \subset \mathcal{R}(x_0, T)$
converging to $\xi$ such that
$$
  \left\langle \nabla g_i(\xi), \frac{\xi_n - \xi}{|\xi_n - \xi|} \right\rangle \le - \beta
  \quad \forall n \in \mathbb{N} \quad \forall i \in I \colon g_i(\xi) \ge 0.
$$
Then arguing in essentially the same way as in the proof of Theorem~\ref{Theorem_FixedEndPointProblem_NonLinear} one
can show that the penalty function $\Phi_{\lambda}$ for problem \eqref{VariableEndPointPenaltyProblem} is completely
exact on $S_{\lambda}(c)$ for any sufficiently large $\lambda$. In the general case, a similar but more cumbersome
assumption must be imposed on both equality and inequality constraints.
\end{remark}

\section{Exact Penalisation of Pointwise State Constraints}
\label{Sect_ExactPen_StateConstraint}

Let us now turn to the analysis of the exactness of penalty functions for optimal control problems with
pointwise state constraints. In this case the situation is even more complicated than in the case of problems with
terminal constraints. It seems that verifiable sufficient conditions for the complete exactness of a penalty function
for problems with state constraints can be obtained either under very stringent assumptions on the controllability of
the system or in the case of linear systems and convex state constraints. Furthermore, a penalty term for state
constraints can be designed with the use of the $L^p$-norm with any $1 \le p \le + \infty$. The smooth norms with 
$1 < p < + \infty$ and the $L^1$-norm are more appealing for practical applications, while, often, one can guarantee
exact penalisation of state constraints only in the case $p = + \infty$.

\subsection{A Counterexample}

We start our analysis of state constrained problems with a simple counterexample that illuminates the difficulties of
designing exact penalty functions for state constraints. It also demonstrates that in the case when the functional
$\mathcal{I}(x, u)$ explicitly depends on control it is apparently impossible to define an exact penalty function
for problems with state \textit{equality} constraints within the framework adopted in our study.

\begin{example} \label{CounterExample_StateEqConstr}
Let $d = 2$, $m = 1$, and $p = q = 2$. Define 
$U = \{ u \in L^2(0, T) \mid u(t) \in [-1, 1] \text{ for a.e. }  t \in (0, T) \}$, and consider the following
fixed-endpoint optimal control problem with state equality constraint:
\begin{align*}
  &\min \: \mathcal{I}(u) = - \int_0^T u(t)^2 dt \\
  &\text{s.t.} \quad \begin{cases} \dot{x}^1 = 1 \\ \dot{x}^2 = u \end{cases} \quad t \in [0, T], \quad 
  x(0) = \begin{pmatrix} 0 \\ 0 \end{pmatrix}, \quad 
  x(T) = \begin{pmatrix} T \\ 0 \end{pmatrix}, \quad 
  u \in U, \quad g(x(t)) \equiv 0,
\end{align*}
where $g(x^1, x^2) = x^2$. The only feasible point of this problem is $(x^*, u^*)$ with $x^*(t) \equiv (t, 0)^T$ and 
$u^*(t) = 0$ for a.e. $t \in [0, T]$. Thus, $(x^*, u^*)$ is a globally optimal solution of this problem.

We would like to penalise the state equality constraint $g(x(t)) = x^2(t) = 0$. One can define the penalty term in one
of the following ways:
$$
  \varphi(x) = \Big( \int_0^T |g(x(t))|^r \, dt \Big)^{1/r}, \quad 1 \le r < + \infty, \quad
  \varphi(x) = \max_{t \in [0, T]} |g(x(t))|, \quad
  \varphi(x) = \int_0^T |g(x(t))|^{\alpha} \, dt, \quad 0 < \alpha < 1.
$$
Clearly, all these functions are continuous with respect to the uniform metric, which by inequality
\eqref{SobolevImbedding} implies that they are continuous on $W^d_{1,p}(0, T)$. Therefore, instead of choosing a
particular function $\varphi$, we simply suppose that $\varphi \colon W^d_{1,p}(0, T) \to [0, + \infty)$ is an arbitrary
function, continuous with respect to the uniform metric, and such that $\varphi(x) = 0$ if and only if 
$g(x(t)) \equiv 0$. One can consider the penalised problem
\begin{equation} \label{Ex_PenalizedStatEqConstr}
\begin{split}
  &\min \: \Phi_{\lambda}(x, u) = - \int_0^T u(t)^2 dt + \lambda \varphi(x) \\
  &\text{s.t.} \quad \begin{cases} \dot{x}^1 = 1 \\ \dot{x}^2 = u \end{cases} \quad t \in [0, T], \quad 
  x(0) = \begin{pmatrix} 0 \\ 0 \end{pmatrix}, \quad 
  x(T) = \begin{pmatrix} T \\ 0 \end{pmatrix}, \quad 
  u \in U.
\end{split}
\end{equation}
Observe that the goal function $\mathcal{I}$ is Lipschitz continuous on any bounded subset of $L^2(0, T)$ by
\cite[Proposition~4]{DolgopolikFominyh}, and the set 
$$
  A = \big\{ (x, u) \in X \mid x(0) = (0, 0)^T, \: x(T) = (T, 0)^T, \: u \in U, \: 
  \dot{x}^1 = 1, \: \dot{x}^2 = u \text{ for a.e. } t \in [0, T] \big\}
$$
is obviously closed in $X = W^2_{1, 2}(0, T) \times L^2(0, T)$. Consequently, penalised problem
\eqref{Ex_PenalizedStatEqConstr} fits the framework of Section~\ref{Sect_ExactPenaltyFunctions}. However, the penalty
function $\Phi_{\lambda}$ is not exact regardless of the choice of the penalty term $\varphi$.

Indeed, arguing by reductio ad absurdum, suppose that $\Phi_{\lambda}$ is globally exact. Then there exists 
$\lambda \ge 0$ such that $\Phi_{\lambda}(x, u) \ge \Phi_{\lambda}(x^*, u^*)$ for all $(x, u) \in A$. For any 
$n \in \mathbb{N}$ define
$$
  u_n(t) = \begin{cases}
    1, & \text{if } t \in \left[\frac{T(2k - 2)}{2n}, \frac{T(2k - 1)}{2n} \right), \: k \in \{ 1, 2, \ldots, n \}, \\
    -1, & \text{if } t \in \left[ \frac{T(2k - 1)}{2n}, \frac{Tk}{n} \right), \: k \in \{ 1, 2, \ldots, n \},
  \end{cases}
$$
i.e. $u_n$ takes alternating values $\pm 1$ on the segments of length $T / 2n$. For the corresponding trajectory $x_n$
one has $x(0) = (0, 0)^T$, $x(T) = (T, 0)^T$ (i.e. $(x_n, u_n) \in A$), and 
$\| x_n^2 \|_{\infty} = T/2n$. Therefore, $\varphi(x_n) \to 0$ as $n \to \infty$ due to the continuity of the function
$\varphi$ with respect to the uniform metric. On the other hand, $\mathcal{I}(u_n) = -T$ for all $n \in \mathbb{N}$,
which implies that $\Phi_{\lambda}(x_n, u_n) \to - T$ as $n \to \infty$. Consequently, 
$\Phi_{\lambda}(x_n, u_n) < 0 = \Phi_{\lambda}(x^*, u^*)$ for any sufficiently large $n \in \mathbb{N}$, which
contradicts our assumption. Thus, the penalty function $\Phi_{\lambda}$ is not globally exact for any penalty term
$\varphi$ that is continuous with respect to the uniform metric.
\end{example}

The previous example might lead one to think that linear penalty functions for state constrained optimal control
problems cannot be exact. Our aim is to show that in some cases exact penalisation of state constraints (especially,
state inequality constraints) is nevertheless possible, but one must utilise the highly nonsmooth $L^{\infty}$-norm to
achieve exactness. Furthermore, we demonstrate that exact $L^p$-penalisation with finite $p$ is possible in the case
when either the problem is convex and Lagrange multipliers corresponding to state constraints belong to $L^{p'}(0, T)$
or the functional $\mathcal{I}(x, u)$ does not depend on the control inputs explicitly.

\subsection{Linear Evolution Equations}

We start with the convex case, i.e. with the case when the controlled system is linear and state inequality constraints
are convex. The convexity of constraints, along with widely known Slater's conditions from convex optimisation, allows
one to prove the complete exactness of $L^{\infty}$-penalty function under relatively mild assumptions. The main results
on exact penalty functions in this case can be obtained for both linear time varying systems and linear evolution
equations in Hilbert spaces. For the sake of shortness, we consider only evolution equations.

Let, as in Section~\ref{SubSec_EvolEq_TerminalConstr}, $\mathscr{H}$ and $\mathscr{U}$ be complex Hilbert spaces,
$\mathbb{T}$ be a strongly continuous semigroup on $\mathscr{H}$ with generator 
$\mathcal{A} \colon \mathcal{D}(\mathcal{A}) \to \mathscr{H}$, and let $\mathcal{B}$ be an admissible control operator
for $\mathbb{T}$. For any $t \ge 0$ denote by 
$F_t u = \int_0^t \mathbb{T}_{t - \sigma} \mathcal{B} u(\sigma) \, d \sigma$ the input map corresponding 
to $(\mathcal{A}, \mathcal{B})$. Then, as was pointed out in Section~\ref{SubSec_EvolEq_TerminalConstr}, for any 
$u \in L^2((0, T); \mathscr{U})$ the initial value problem $\dot{x}(t) = \mathcal{A} x(t) + \mathcal{B} u(t)$, 
$x(0) = x_0$ with $x_0 \in \mathscr{H}$ has a unique solution $x \in C([0, T]; \mathscr{H})$ given by
\begin{equation} \label{SolutionViaSemiGroup_SC}
  x(t) = \mathbb{T}_t x_0 + F_t u \quad \forall t \in [0, T].
\end{equation}
Consider the following fixed-endpoint optimal control problem with state constraints:
\begin{equation} \label{EvolEqStateConstrainedProblem}
\begin{split}
  {}&\min_{(x, u)} \, \mathcal{I}(x, u) = \int_0^T \theta(x(t), u(t), t) \, dt \quad
  \text{subject to} \quad \dot{x}(t) = \mathcal{A} x(t) + \mathcal{B} u(t), \quad t \in [0, T], \\
  {}&x(0) = x_0, \quad x(T) = x_T, \quad u \in U, \quad
  g_j(x(t), t) \le 0 \quad  \forall t \in [0, T], \quad j \in J.
\end{split}
\end{equation}
Here $\theta \colon \mathscr{H} \times \mathscr{U} \times [0, T] \to \mathbb{R}$ and 
$g_j \colon \mathscr{H} \times [0, T] \to \mathbb{R}$, $j \in J = \{ 1, \ldots, l \}$, are given
functions, $T > 0$ and $x_0, x_T \in \mathscr{H}$ are fixed, and $U \subseteq L^2((0, T); \mathscr{U})$ is a
closed convex set.

Let us introduce a penalty function for problem \eqref{EvolEqStateConstrainedProblem}. Our aim is to penalise the state
inequality constraints $g_j(x(t), t) \le 0$. To this end, define 
$X = C([0, T]; \mathscr{H}) \times L^2((0, T); \mathscr{U})$, 
$M = \{ (x, u) \in X \mid g_j(x(t), t) \le 0 \quad \forall t \in [0, T], \, j \in J \}$, and
$$
  A = \Big\{ (x, u) \in X \Bigm| x(0) = x_0, \: x(T) = x_T, \: u \in U, \: 
  \text{and $\eqref{SolutionViaSemiGroup_SC}$ holds true} \Big\}.
$$
Then problem \eqref{EvolEqStateConstrainedProblem} can be rewritten as the problem of minimizing $\mathcal{I}(x, u)$
subject to $(x, u) \in M \cap A$. Introduce the penalty term 
$\varphi(x, u) = \sup_{t \in [0, T]} \{ g_1(x(t), t), \ldots, g_l(x(t), t), 0 \}$.
Then $M = \{ (x, u) \in X \mid \varphi(x, u) = 0 \}$, and one can consider the penalised problem of minimising 
$\Phi_{\lambda}$ over the set $A$, which is a fixed-endpoint problem without state constraints of the form:
\begin{equation*}
\begin{split}
  {}&\min_{(x, u)} \, \int_0^T \theta(x(t), u(t), t) \, dt 
  + \lambda \sup_{t \in [0, T]} \big\{ g_1(x(t), t), \ldots, g_l(x(t), t), 0 \big\} \\
  {}&\text{subject to } \dot{x}(t) = \mathcal{A} x(t) + \mathcal{B} u(t), \quad t \in [0, T], 
  \quad u \in U, \quad x(0) = x_0, \quad x(T) = x_T.
\end{split}  
\end{equation*}
Note that after discretisation in $t$ this problem becomes a standard minimax problem with convex constraints, which can
be solved via a wide variety of existing numerical methods of minimax optimisation or nonsmooth convex optimisation in
the case when the function $(x, u) \mapsto \theta(x, u, t)$ is convex. Our aim is to show that this fixed-endpoint
problem is equivalent to problem \eqref{EvolEqStateConstrainedProblem}, provided Slater's condition holds true, i.e.
provided there exists a control input $\widehat{u} \in U$ such that for the corresponding solution $\widehat{x}$
(see~\eqref{SolutionViaSemiGroup_SC}) one has $\widehat{x}(T) = x_T$ and $g_j(\widehat{x}(t), t) < 0$ for all 
$t \in [0, T]$ and $j \in J$.

\begin{theorem} \label{Theorem_LinEvolEq_StateConstr}
Let the following assumptions be valid:
\begin{enumerate}
\item{$\theta$ is continuous, and for any $R > 0$ there exist $C_R > 0$ and an a.e. nonnegative function
$\omega_R \in L^1(0, T)$ such that $| \theta(x, u, t) | \le C_R \| u \|_{\mathscr{U}}^2 + \omega_R(t)$ for all 
$x \in \mathscr{H}$, $u \in \mathscr{U}$, and $t \in (0, T)$ such that $\| x \|_{\mathscr{H}} \le R$;
\label{Assumpt_LinEvolEq_SC_ThetaGrowth}}

\item{either the set $U$ is bounded in $L^2((0, T), \mathscr{U})$ or there exist $C_1 > 0$ and $\omega \in L^1(0, T)$
such that $\theta(x, u, t) \ge C_1 \| u \|_{\mathscr{U}}^2 + \omega(t)$ for all $x \in \mathscr{H}$, 
$u \in \mathscr{U}$, and for a.e. $t \in [0, T]$;
}

\item{$\theta$ is differentiable in $x$ and $u$, the functions $\nabla_x \theta$ and $\nabla_u \theta$ are continuous,
and for any $R > 0$ there exist $C_R > 0$, and a.e. nonnegative functions $\omega_R \in L^1(0, T)$ and 
$\eta_R \in L^2(0, T)$ such that
$$
  \| \nabla_x \theta(x, u, t) \|_{\mathscr{H}} \le C_R \| u \|_{\mathscr{U}}^2 + \omega_R(t), \quad
  \| \nabla_u \theta(x, u, t) \|_{\mathscr{U}} \le C_R \| u \|_{\mathscr{U}} + \eta_R(t)
$$
for all $x \in \mathscr{H}$, $u \in \mathscr{U}$ and $t \in (0, T)$ such that $\| x \|_{\mathscr{H}} \le R$;
}

\item{there exists a globally optimal solution of problem \eqref{EvolEqStateConstrainedProblem};
\label{Assumpt_LinEvolEq_SC_GlobSol}}

\item{the functions $g_j(x, t)$, $j \in J$, are convex in $x$, continuous jointly in $x$ and $t$, and Slater's condition
holds true.
\label{Assumpt_LinEvolEq_StateConstr}}
\end{enumerate}
Then for all $c \in \mathbb{R}$ there exists $\lambda^*(c) \ge 0$ such that for any $\lambda \ge \lambda^*(c)$ the
penalty function $\Phi_{\lambda}$ for problem \eqref{EvolEqStateConstrainedProblem} is completely exact on the set
$S_{\lambda}(c)$.
\end{theorem}

\begin{proof}
Almost literally repeating the first part of the proof of Theorem~\ref{Theorem_Exactness_EvolutionEquations} one
obtains that the assumptions on the function $\theta$ and its derivatives ensure that the functional
$\mathcal{I}(x, u)$ is Lipschitz continuous on any bounded subset of $X$, the set $S_{\lambda}(c)$ is bounded in $X$
for all $c \in \mathbb{R}$ and $\lambda \ge 0$, and the penalty function $\Phi_{\lambda}$ is bounded below on $A$. In
addition, the set $A$ is closed by virtue of the closedness of the set $U$ and the fact that the input map $F_t$ is
a bounded linear operator from $L^2((0, T); \mathscr{U})$ to $\mathscr{H}$ (see~\eqref{SolutionViaSemiGroup_SC}). 
Finally, the mappings $x \mapsto g_j(x(\cdot), \cdot)$, $j \in J$, and the penalty term $\varphi$ are continuous
by Proposition~\ref{Prop_ContNonlinearMap_in_C} and Corollary~\ref{Corollary_StateConstrPenTerm_Contin} (see
Appendix~B). 

Fix any $\lambda \ge 0$ and $c \in \mathbb{R}$. By applying Theorem~\ref{Theorem_CompleteExactness} one gets that it is
remains to verify that there exists $a > 0$ such that $\varphi^{\downarrow}_A(x, u) \le - a$ for any 
$(x, u) \in S_{\lambda}(c) \cap (\Omega_{\delta} \setminus \Omega)$ 
(i.e. $(x, u) \in S_{\lambda}(c)$ and $0 < \varphi(x, u) < \delta$).

Fix any $\delta > 0$ and $(x, u) \in S_{\lambda}(c) \cap (\Omega_{\delta} \setminus \Omega)$, and let a pair
$(\widehat{x}, \widehat{u})$ be from Slater's condition. Denote 
$\sigma = \| (\widehat{x}, \widehat{u}) - (x, u) \|_X = 
\| \widehat{x} - x \|_{C([0, T]; \mathscr{H})} + \| \widehat{u} - u \|_{L^2((0, T); \mathscr{U})}$.
Note that there exists $R > 0$ (independent of $(x, u)$) such that $\sigma \le R$, since the set $S_{\lambda}(c)$ is
bounded. Furthermore, $\sigma > 0$, since $(\widehat{x}, \widehat{u}) \in \Omega$ by definition.

Define $\Delta x = (\widehat{x} - x) / \sigma$ and $\Delta u = (\widehat{u} - u) / \sigma$. Observe that
$\| (\Delta x, \Delta u) \|_X = 1$, and $(x + \alpha \Delta x, u + \alpha \Delta u) \in A$ for any 
$\alpha \in [0, \sigma]$ due to the convexity of the set $U$ and the linearity of the system 
$\dot{x} = \mathcal{A} x + \mathcal{B} u$. With the use of the convexity of the functions $g_j(x, t)$ in $x$ one obtains
that
$$
  g_j(x(t) + \alpha \Delta x(t), t) \le \frac{\alpha}{\sigma} g_j(\widehat{x}(t), t) 
  + \left( 1 - \frac{\alpha}{\sigma} \right) g_j(x(t), t) 
  \le \frac{\alpha \eta}{\sigma} + \left( 1 - \frac{\alpha}{\sigma} \right) \varphi(x, u)
  \le \frac{\alpha \eta}{\sigma} + \varphi(x, u).
$$
for any $\alpha \in [0, \sigma]$, $t \in [0, T]$, and $j \in J$, where $\eta = \max_{t \in [0, T]}\{ g_j(\widehat{x}(t),
t) \mid j \in J \}$ (note that $\eta < 0$ due to Slater's condition). Therefore, one has
$$
  \max_{t \in [0, T]} \big\{ g_1(x(t) + \alpha \Delta x(t), t), \ldots, g_l(x(t) + \alpha \Delta x(t), t) \big\}
  \le \frac{\alpha \eta}{\sigma} + \varphi(x, u)
  \quad \forall \alpha \in [0, \sigma].
$$
Recall that $\varphi(x, u) > 0$, since $(x, u) \notin \Omega$. Consequently, bearing in mind the fact that 
the right-hand side of the inequality above is positive for any $\alpha < \varphi(x, u) \sigma / |\eta|$ one obtains
that
$$
  \varphi(x + \alpha \Delta x, u + \alpha \Delta u) = 
  \max_{t \in [0, T], j \in J} \big\{ g_j(x(t) + \alpha \Delta x(t), t), 0 \big\}
  \le \frac{\alpha \eta}{\sigma} + \varphi(x, u)
  \quad \forall \alpha \in \left[ 0, \min\left\{ \sigma, \frac{\varphi(x, u) \sigma}{|\eta|} \right\} \right).
$$
Dividing this inequality by $\alpha$ and passing to the limit superior as $\alpha \to + 0$ one finally gets that
$$
  \varphi^{\downarrow}_A(x, u) 
  \le \limsup_{\alpha \to +0} \frac{\varphi(x + \alpha \Delta x, u + \alpha \Delta u)}{\alpha}
  \le \frac{\eta}{\sigma} \le \frac{\eta}{R} < 0
$$
where both $\eta$ and $R$ are independent of $(x, u) \in S_{\lambda}(c) \cap (\Omega_{\delta} \setminus \Omega)$. Thus, 
$\varphi^{\downarrow}_A(x, u) \le - \eta / R$ for any 
$(x, u) \in S_{\lambda}(c) \cap (\Omega_{\delta} \setminus \Omega)$, and the proof is complete.
\end{proof}

\begin{corollary} \label{Corollary_LinEvolEq_StateConstr}
Let all assumptions of Theorem~\ref{Theorem_LinEvolEq_StateConstr} be valid. Suppose also that either the set $U$ is
bounded in $L^2((0, T); \mathscr{U})$ or the function $(x, u) \mapsto \theta(x, u, t)$ is convex for all $t \in [0, T]$.
Then the penalty function $\Phi_{\lambda}$ for problem \eqref{EvolEqStateConstrainedProblem} is completely exact on $A$.
\end{corollary}

\begin{proof}
If the set $U$ is bounded, then by the first part of the proof of
Theorem~\ref{Theorem_Exactness_EvolutionEquations_Global} the set $A$ is bounded in $X$. Therefore, arguing in
the same way as in the proof of Theorem~\ref{Theorem_LinEvolEq_StateConstr}, but replasing the set $S_{\lambda}(c)$ with
$A$ and utilising Theorem~\ref{THEOREM_COMPLETEEXACTNESS_GLOBAL} instead of Theorem~\ref{Theorem_CompleteExactness} we
arrive at the required result.

If the function $(x, u) \mapsto \theta(x, u, t)$ is convex, then, as was shown in the proof of 
Theorem~\ref{Theorem_Exactness_EvolutionEquations_Global}, the penalty function $\Phi_{\lambda}$ for problem
\eqref{EvolEqStateConstrainedProblem} is completely exact on $A$ if and only if it is globally exact. It remains to
note that its global exactness follows from Theorem~\ref{Theorem_LinEvolEq_StateConstr}.
\end{proof}

\begin{remark}
Theorem~\ref{Theorem_LinEvolEq_StateConstr} and corollary~\ref{Corollary_LinEvolEq_StateConstr} can be easily extended
to the case of problems with inequality constraints of the form $g_j(x, u) \le 0$, where 
$g_j \colon C([0, T]; \mathscr{H}) \times L^2((0, T); \mathscr{U}) \to \mathbb{R}$ are continuous convex functions. In
particular, one can consider the integral constraint  $\| u \|_{L^2((0, T); \mathscr{U})} \le C$ for some $C > 0$. In
this case one can define $\varphi(x, u) = \max\{ g_1(x, u), \ldots g_l(x, u), 0 \}$, while Slater's condition takes the
form: there exists a feasible point $(\widehat{x}, \widehat{u})$ such that $g_j(\widehat{x}, \widehat{u}) < 0$ for 
all $j$.
\end{remark}

\begin{remark}
It should be noted that Theorem~\ref{Theorem_LinEvolEq_StateConstr} and Corollary~\ref{Corollary_LinEvolEq_StateConstr}
can be applied to problems with distributed $L^{\infty}$ state constraints. For instance, suppose that
$\mathscr{H} = W^{1, 2}(0, 1)$, and let the constraints have the form $b_1 \le x(t, r) \le b_2$ for all $t \in [0, T]$
and a.e. $r \in (0, 1)$. Then one can define $g_1(x(t)) = \esssup_{r \in (0, 1)} x(t, r) - b_2$ and 
$g_2(x(t)) = \esssup_{r \in (0, 1)} (- x(t, r)) + b_1$ and consider the state constraints $g_1(x(\cdot)) \le 0$ and
$g_2(x(\cdot)) \le 0$. One can easily check that both functions $g_1$ and $g_2$ are convex and continuous. Slater's
condition in this case takes the form: there exists $(\widehat{x}, \widehat{u}) \in \Omega$ such that 
$b_1 + \varepsilon \le \widehat{x}(t, r) \le b_2 - \varepsilon$ for some $\varepsilon > 0$, for all $t \in [0, T]$, and 
a.e. $r \in (0, 1)$.
\end{remark}

Observe that Slater's condition imposes some restriction on the initial and final states. Namely, Slater's conditions
implies that $g_j(x_0, 0) < 0$ and $g_j(x_T, T) < 0$ for all $j \in J$ (in the general case only the inequalities
$g_j(x_0, 0) \le 0$ and $g_j(x_T, T) \le 0$ hold true). Let us give an example demonstrating that in the case when
the strict inequalities are not satisfied, the penalty function $\Phi_{\lambda}$ for problem
\eqref{EvolEqStateConstrainedProblem} need not be exact. For the sake of simplicity, we consider a free-endpoint
finite dimensional problem. As one can readily verify, Theorem~\ref{Theorem_LinEvolEq_StateConstr} remains valid in the
case of free-endpoint problems.

\begin{example} \label{CounterExample_StateInEqConstr}
Let $d = 2$, $m = 1$, $p = q = 2$. Define 
$U = \{ u \in L^2(0, T) \mid u(t) \ge 0 \text{ for a.e. }  t \in (0, T), \: \| u \|_2 \le 1 \}$, and consider the
following free-endpoint optimal control problem with the state inequality constraint:
$$
  \min \: \mathcal{I}(u) = - \int_0^T u(t)^2 dt \quad
  \text{s.t.} \quad \begin{cases} \dot{x}^1 = 1 \\ \dot{x}^2 = u \end{cases} \quad t \in [0, T], \quad 
  x(0) = \begin{pmatrix} 0 \\ 0 \end{pmatrix}, \quad u \in U, \quad g(x(t)) \le 0,
$$
where $g(x^1, x^2) = x^2$. The only feasible point of this problem is $(x^*, u^*)$ with 
$x^*(t) \equiv (t, 0)^T$ and $u^*(t) = 0$ for a.e. $t \in [0, T]$. Thus, $(x^*, u^*)$ is a globally optimal solution of
this problem. Note also that the function $\theta(x, u, t) = - (u)^2$ satisfies the assumptions of
Theorem~\ref{Theorem_LinEvolEq_StateConstr}. Furthermore, in this case the set $A$ is obviously bounded in $X$, and 
the penalty function $\Phi_{\lambda}(x, u) = \mathcal{I}(u) + \lambda \varphi(x)$ with
$\varphi(x) = \max_{t \in [0, T]} \{ g(x(t)), 0 \}$ is bounded below on $A$.

Observe that $g(x(0)) = 0$, which implies that Slater's condition does not hold true. Let us check that the penalty
function $\Phi_{\lambda}$ is not exact. Arguing by reductio ad absurdum, suppose that $\Phi_{\lambda}$ is globally
exact. Then there exists $\lambda^* \ge 0$ such that for any $\lambda \ge \lambda^*$ and $(x, u) \in A$ one has
$\Phi_{\lambda}(x, u) \ge \Phi_{\lambda}(x^*, u^*)$. For any $n \in \mathbb{N}$ define $u_n(t) = n$, if 
$t \in [0, 1 / n^2]$, and $u_n(t) = 0$, if $t > 1 / n^2$. Then $\| u_n \|_2 = 1$, and $(x_n, u_n) \in A$, where 
$x_n(t) = (t, \min\{ nt, 1 / n \})^T$ is the corresponding trajectory of the system. Observe that 
$\mathcal{I}(u_n) = -1$ and $\varphi(x_n) = 1 / n$ for any $n \in \mathbb{N}$. Consequently, 
$\Phi_{\lambda}(x_n, u_n) < 0 = \Phi_{\lambda}(x^*, u^*)$ for any sufficiently large $n \in \mathbb{N}$, which
contradicts our assumption. Thus, the penalty function $\Phi_{\lambda}$ is not globally exact. Moreover, one can easily
see that the penalty function $\Phi_{\lambda}(x) = \mathcal{I}(u) + \lambda \varphi(x)$ is not globally exact for any
penalty term $\varphi$ that is continuous with respect to the uniform metric. 
\end{example}

In the case when not only the state constraints but also the cost functional $\mathcal{I}$ are convex, one can utilise
the convexity of the problem to prove that the exact $L^p$-penalisation of state constraints with any 
$1 \le p < + \infty$ is possible, provided Lagrange multipliers corresponding to the state constraints are sufficiently
regular. Indeed, let $(x^*, u^*)$ be a globally optimal solution of problem \eqref{EvolEqStateConstrainedProblem}, and
let $E(x^*, u^*) \subset \mathbb{R} \times (C[0, T])^l$ be a set of all those vectors $(y_0, y_1, \ldots, y_l)$ for
which one can find $(x, u) \in A$ such that $\mathcal{I}(x, u) - \mathcal{I}(x^*, u^*) < y_0$ and 
$g_j(x(t), t) \le y_j(t)$ for all $t \in [0, T]$ and $j \in J$. The set $E(x^*, u^*)$ has nonempty interior due to 
the fact that $(0, + \infty) \times (C_+[0, T])^l \subset E(x^*, u^*)$ (put $(x, u) = (x^*, u^*)$), where $C_+[0, T]$ is
the cone of nonnegative functions. Observe also that $0 \notin E(x^*, u^*)$, since otherwise one can find a feasible
point $(x, u)$ of problem \eqref{EvolEqStateConstrainedProblem} such that $\mathcal{I}(x, u) < \mathcal{I}(x^*, u^*)$,
which contradicts the definition of $(x^*, u^*)$. Furthermore, with the use of the convexity of $\mathcal{I}$ and $g_j$
one can easily check that the set $E(x^*, u^*)$ is convex. Therefore, by applying the separation theorem (see, e.g.
\cite[Theorem~V.2.8]{DunfordSchwartz}) one obtains that there exist $\mu_0 \in \mathbb{R}$ and continuous linear
functionals $\psi_j$ on $C[0, T]$, $j \in J$, not all zero, such that
$\mu_0 y_0 + \sum_{j = 1}^l \psi_j(y_j) \ge 0$ for all $(y_0, y_1, \ldots, y_l) \in E(x^*, u^*)$. Taking into account
the fact that  $(0, + \infty) \times (C_+[0, T])^l \subset E(x^*, u^*)$ one obtains that $\mu_0 \ge 0$ and
$\psi_j(y) \ge 0$ for any $y \in C_+[0, T]$ and $j \in J$. Consequently, utilising the Riesz-Markov-Kakutani
representation theorem (see~\cite[Theorem~IV.6.3]{DunfordSchwartz}) and bearing in mind the definition of $E(x^*, u^*)$
one gets that there exist regular Borel measures $\mu_j$ on $[0, T]$, $j \in J$, such that
\begin{equation} \label{LagrangeMultiplier_StateConstr}
  \mu_0 \mathcal{I}(x, u) + \sum_{j = 1}^l \int_{[0, T]} g_j(x(t), t) \, d \mu_j(t) \ge \mu_0 \mathcal{I}(x^*, u^*)
  \quad \forall (x, u) \in A.
\end{equation}
If Slater's condition holds true, then obviously $\mu_0 > 0$, and we suppose that $\mu_0 = 1$. Any collection 
$(\mu_1, \ldots, \mu_l)$ of regular Borel measures on $[0, T]$ satisfying \eqref{LagrangeMultiplier_StateConstr} with
$\mu_0 = 1$ is called \textit{Lagrange multipliers} corresponding to the state constraints of problem
\eqref{EvolEqStateConstrainedProblem}. Let us note that one has to suppose that Lagrange multipliers are Borel measures,
since if one replaces $C[0, T]$ in the definition of $E(x^*, u^*)$ with $L^p[0, T]$, $1 \le p < + \infty$, then the set
$E(x^*, u^*)$, in the general case, has empty interior, which makes the separation theorem inapplicable.

\begin{theorem}
Let assumptions \ref{Assumpt_LinEvolEq_SC_ThetaGrowth}, \ref{Assumpt_LinEvolEq_SC_GlobSol}, and
\ref{Assumpt_LinEvolEq_StateConstr} of Theorem~\ref{Theorem_LinEvolEq_StateConstr} be valid, and let the function 
$(x, u) \mapsto \theta(x, u, t)$ be convex for any $t \in [0, T]$. Suppose, in addition, that for some 
$1 \le p < + \infty$ there exist Lagrange multipliers $(\mu_1, \ldots, \mu_l)$ such that the Borel measures $\mu_j$
are absolutely continuous with respect to the Lebesgue measure, and their Radon-Nikodym derivatives belong to 
$L^{p'}[0, T]$. Then there exists $\lambda^* \ge 0$ such that for any $\lambda \ge \lambda^*$ the penalty function
$$
  \Phi_{\lambda}(x, u) = \mathcal{I}(x, u) 
  + \lambda \sum_{j = 1}^l \bigg( \int_0^T \max\{ g_j(x(t), t), 0 \}^p \, dt \bigg)^{1/p}
$$
for problem \eqref{EvolEqStateConstrainedProblem} is completely exact on $A$.
\end{theorem}

\begin{proof}
Let $(x^*, u^*)$ be a globally optimal solution of problem \eqref{EvolEqStateConstrainedProblem}, and $h_j$ be 
the Radon-Nikodym derivative of $\mu_j$ with respect to the Lebesgue measure, $j \in J$. Denote 
$\lambda_0 = \max_{j \in J} \| h_j \|_{p'}$. Then by applying \eqref{LagrangeMultiplier_StateConstr} and
H\"{o}lder's inequality one obtains that
\begin{align*}
  \Phi_{\lambda}(x^*, u^*) = \mathcal{I}(x^*, u^*) 
  &\le \mathcal{I}(x, u) + \sum_{j = 1}^l \int_0^T g_j(x(t), t) h_j(t) \, dt 
  \le \mathcal{I}(x, u) + \sum_{j = 1}^l \int_0^T \max\{ g_j(x(t), t), 0 \} |h_j(t)| \, dt \\
  &\le \mathcal{I}(x, u) + \lambda_0 \sum_{j = 1}^l \bigg( \int_0^T \max\{ g_j(x(t), t), 0 \}^p \, dt \bigg)^{1/p} 
  \le \Phi_{\lambda}(x, u)
\end{align*}
for any $(x, u) \in A$ and $\lambda \ge \lambda_0$. Hence, as is easy to see, the penalty function $\Phi_{\lambda}$ is
globally exact. Now, bearing in mind the convexity of $\Phi_{\lambda}$ and arguing in the same way as in the proof of
Theorem~\ref{Theorem_Exactness_EvolutionEquations_Global} one arrives at the required result.
\end{proof}

\subsection{Nonlinear Systems: Local Exactness}

Let us now turn to general nonlinear optimal control problems with state constraints of the form:
\begin{equation} \label{StateConstrainedProblem}
\begin{split}
  &\min \: \mathcal{I}(x, u) = \int_0^T \theta(x(t), u(t), t) \, dt \quad
  \text{subject to } \dot{x}(t) = f(x(t), u(t), t), \quad t \in [0, T], \\
  &x(0) = x_0, \quad x(T) = x_T, \quad u \in U, \quad
  g_j(x(t), t) \le 0 \quad \forall t \in [0, T], \: j \in J.
\end{split}
\end{equation}
Here $\theta \colon \mathbb{R}^d \times \mathbb{R}^m \times [0, T] \to \mathbb{R}$,
$f \colon \mathbb{R}^d \times \mathbb{R}^m \times [0, T] \to \mathbb{R}^d$, 
$g_j \colon \mathbb{R}^d \times [0, T] \to \mathbb{R}$, $j \in J = \{ 1, \ldots, l \}$, are given functions, 
$x_0, x_T \in \mathbb{R}^d$, and $T > 0$ are fixed, $x \in W^d_{1, p}(0, T)$, and $U \subseteq L_q^m(0, T)$ is a closed
set of admissible control inputs.

Let $X = W^d_{1, p}(0, T) \times L_q^m(0, T)$. Define 
$M = \{ (x, u) \in X \mid g_j(x(t), t) \le 0 \text{ for all } t \in [0, T], \: j \in J \}$ and
$$
  A = \Big\{ (x, u) \in X \mid u \in U, \: x(0) = x_0, \: x(T) = x_T, \: 
  \dot{x}(t) = f(x(t), u(t), t) \text{ for a.e. } t \in (0, T) \Big\}.
$$
Then problem \eqref{StateConstrainedProblem} can be rewritten as the problem of minimising $\mathcal{I}(x, u)$ over the
set $M \cap A$. Define 
$\varphi(x, u) = \sup_{t \in [0, T]} \big\{ g_1(x(t), t), \ldots, g_l(x(t), t), 0 \big\}$.
Then $M = \{ (x, u) \in X \mid \varphi(x, u) = 0 \}$, and one can consider the penalised problem of minimising the
penalty function $\Phi_{\lambda}$ over the set $A$. Our first goal is to obtain simple sufficient conditions for 
the local exactness of the penalty function $\Phi_{\lambda}$.

\begin{theorem} \label{Theorem_StateConstr_LocalExact}
Let $U = L_q^m(0, T)$, $q \ge p$, and $(x^*, u^*)$ be a locally optimal solution of problem
\eqref{StateConstrainedProblem}. Let also the following assumptions be valid:
\begin{enumerate}
\item{$\theta$ and $f$ are continuous, differentiable in $x$ in $u$, and the functions $\nabla_x \theta$, 
$\nabla_u \theta$, $\nabla_x f$, and $\nabla_u f$ are continuous;
}

\item{either $q = + \infty$ or $\theta$ and $\nabla_x \theta$ satisfy the growth condition of order $(q, 1)$, 
$\nabla_u \theta$ satisfies the growth condition of order $(q - 1, q')$, $f$ and $\nabla_x f$ satisfy the growth
condition of order $(q / p, p)$, and $\nabla_u f$ satisfies the growth condition of order $(q / s, s)$ with 
$s = qp / (q - p)$ in the case $q > p$, and $\nabla_u f$ does not depend on $u$ in the case $q = p$;
}

\item{$g_j$, $j \in J$, are continuous, differentiable in $x$, and the functions $\nabla_x g_j$, $j \in J$, are
continuous.
}
\end{enumerate}
Suppose finally that the linearised system
\begin{equation} \label{LinearizedSystem}
  \dot{h}(t) = A(t) h(t) + B(t) v(t), 
  \quad A(t) = \nabla_x f(x^*(t), u^*(t), t), \quad B(t) = \nabla_u f(x^*(t), u^*(t), t),
\end{equation}
is completely controllable using $L^q$-controls in time $T$, $A(\cdot) \in L_{\infty}^{d \times d}(0, T)$, and there
exists $v \in L^q(0, T)$ such that the corresponding solution $h$ of \eqref{LinearizedSystem} with $h(0) = 0$ satisfies
the condition $h(T) = 0$, and for any $j \in J$ one has
\begin{equation} \label{MFCQ_StateConstr}
  \langle \nabla_x g_j(x^*(t), t), h(t) \rangle < 0
  \quad \forall t \in [0, T] \colon g_j(x^*(t), t) = 0.
\end{equation}
Then the penalty function $\Phi_{\lambda}$ for problem \eqref{StateConstrainedProblem} is locally exact at 
$(x^*, u^*)$.
\end{theorem}

\begin{proof}
By \cite[Propositions~3 and 4]{DolgopolikFominyh} the growth conditions on the function $\theta$ and its derivatives
ensure that the functional $\mathcal{I}$ is Lipschitz continuous in any bounded neighbourhood of $(x^*, u^*)$. Introduce
a nonlinear operator $F \colon X \to L_p^d(0, T) \times \mathbb{R}^d \times (C[0, T])^l$ and a closed convex set 
$K \subset L_p^d(0, T) \times \mathbb{R}^d \times (C[0, T])^l$ as follows:
$$
  F(x, u) = \begin{pmatrix} \dot{x}(\cdot) - f(x(\cdot), u(\cdot), \cdot) \\ x(T) \\ g(x(\cdot), \cdot) \end{pmatrix},
  \quad K = \begin{pmatrix} 0 \\ x_T \\ (C_-[0, T])^l \end{pmatrix}
$$
Here $(C[0, T])^l$ is the Cartesian product of $l$ copies of the space $C[0, T]$ of real-valued continuous functions
defined on $[0, T]$ endowed with the uniform norm, $g(\cdot) = (g_1(\cdot), \ldots g_l(\cdot))^T$, and 
$C_-[0, T] \subset C[0, T]$ is the cone of nonpositive functions. Our aim is to apply
Theorem~\ref{Theorem_LocalErrorBound} with $C = \{ (x, u) \in X \mid x(0) = x_0 \}$ to the
operator $F$. Then one obtains that there exists $a > 0$ such that
$\dist(F(x, u), K) \ge a \dist( (x, u), F^{-1}(K) \cap C)$
for any $(x, u) \in C$ in a neighbourhood of $(x^*, u^*)$. Consequently, taking into account the facts that the set
$F^{-1}(K) \cap C$ coincides with the feasible region of problem \eqref{StateConstrainedProblem}, and
$$
  \dist(F(x, u), K) = \sum_{j = 1}^l \max_{t \in [0, T]}\big\{ g_j(x(t), t), 0 \big\} \le l \varphi(x, u)
  \quad \forall (x, u) \in A
$$
one obtains that $\varphi(x, u) \ge (a / l) \dist((x, u), \Omega)$ for any $(x, u) \in A$ in a neighbourhood of 
$(x^*, u^*)$. Hence by applying Theorem~\ref{Theorem_LocalExactness} we arrive at the required result.

By Theorems~\ref{Theorem_DiffNemytskiiOperator} and \ref{Theorem_DiffStateConstr} (see
Appendix~B) the growth conditions on the function $f$ and its derivative guarantee that the
mapping $F$ is strictly differentiable at $(x^*, u^*)$, and its Fr\'{e}chet derivative at this point has the form
$$
  DF(x^*, u^*)[h, v] = 
  \begin{pmatrix} \dot{h}(\cdot) - A(\cdot) h(\cdot) - B(\cdot) v(\cdot) \\ h(T) \\ 
  \nabla_x g(x^*(\cdot), \cdot)h(\cdot) \end{pmatrix},
$$
where $A(\cdot)$ and $B(\cdot)$ are defined in \eqref{LinearizedSystem}. Observe also that
$C - (x^*, u^*) = \{ (h, v) \in X \mid h(0) = 0 \}$, since $x^*(0) = x_0$. Consequently, the regularity condition
\eqref{MetricRegCond} from Theorem~\ref{Theorem_LocalErrorBound} takes the form $0 \in \core K(x^*, u^*)$ with
\begin{equation} \label{RegularityCone_StateConstr}
  K(x^*, u^*) = \left\{ \begin{pmatrix} \dot{h}(\cdot) - A(\cdot) h(\cdot) - B(\cdot) v(\cdot) \\ h(T) \\
  \nabla_x g(x^*(\cdot), \cdot) h(\cdot) \end{pmatrix} -
  \begin{pmatrix} 0 \\ 0 \\ (C_-[0, T])^l - g(x^*(\cdot), \cdot) \end{pmatrix} \Biggm| (h, v) \in X, \: h(0) = 0
  \right\},
\end{equation}
Let us check that this condition is satisfied. Indeed, define $X_0 = \{ (h, v) \in X \mid h(0) = 0 \}$, and introduce
the linear operator $E \colon X_0 \to L_p^d(0, T)$, $E(h, v) = \dot{h}(\cdot) - A(\cdot) h(\cdot) - B(\cdot) v(\cdot)$.
This operator is surjective and bounded, since the linear differential equation $E(h, 0) = w$ has a unique solution
for any $w \in L^d_p(0, T)$ by \cite[Theorem~1.1.3]{Filippov}, and by H\"{o}lder's inequality one has
$$
  \| E(h, v) \|_p \le \| \dot{h} \|_p + \| A(\cdot) \|_\infty \| h \|_{p} + \| B(\cdot) \|_s \| v \|_q
  \le C \| (h, v) \|_X,
$$
where $C = \max\{ 1 + \| A(\cdot) \|_{\infty}, \| B(\cdot) \|_s \}$, and $s = + \infty$ in the case $q = p$
(note that $\| B(\cdot) \|_s$ is finite due to the growth condition on $\nabla_u f$; see the proof of
Theorem~\ref{Theorem_DiffNemytskiiOperator}). Consequently, by the open mapping theorem there exists $\eta_1 > 0$ such
that
$$
  \dist((h, v), E^{-1}(w)) \le \eta_1 \| w - E(h, v) \|_p
  \quad \forall (h, v) \in X_0, \: w \in L_p^d(0, T)
$$
(see~\cite[formula~$(0.2)$]{Ioffe}). Taking $(h, v) = (0, 0)$ in the previous inequality one gets that for any 
$w \in L_p^d(0, T)$ there exists $v_1 \in L_q^m(0, T)$ such that the solution $h_1$ of the pertubed linearised equation
\begin{equation} \label{PerturbedLinearizedEquation}
  \dot{h}_1(t) = A(t) h_1(t) + B(t) v_1(t) + w(t), \quad h(0) = 0, \quad t \in [0, T]
\end{equation}
satisfies the inequality $\| (h_1, v_1) \|_X \le (\eta_1 + 1) \| w \|_p$.

Introduce the operator $\mathcal{T} \colon L_q^m(0, T) \to \mathbb{R}^d$, $\mathcal{T} v = h(T)$, where $h$ is a
solution of \eqref{LinearizedSystem} with the initial condition $h(0) = 0$. Arguing in a similar way to the proof of
Theorem~\ref{Theorem_FixedEndPointProblem_Linear} (recall that $A(\cdot) \in L_{\infty}^{d \times d}(0, T)$) one can
check that the operator $\mathcal{T}$ is bounded, while the complete controllability assumption implies that it is
surjective. Hence by the open mapping theorem there exists $\eta_2 > 0$ such that
$$
  \dist(v, \mathcal{T}^{-1}(h_T)) \le \eta_2 | h_T - \mathcal{T}(v) |
  \quad \forall v \in L_q^m(0, T), \: h_T \in \mathbb{R}^d.
$$
Taking $v = 0$ one obtains that for any $h_T \in \mathbb{R}^d$ there exists $v_2 \in L_q^m(0, T)$ such that
$\| v_2 \|_q \le (\eta_2 + 1) |h_2(T)|$, where $h_2$ is a solution of \eqref{LinearizedSystem} with $v = v_2$ satisfying
the conditions $h_2(0) = 0$ and $h_2(T) = h_T$. Furthermore, by applying the Gr\"{o}nwall-Bellman and H\"{o}lder's
inequalities, and the fact that
$$
  |h_2(t)| \le \| B(\cdot) \|_s \| v_2 \|_q + \| A(\cdot) \|_{\infty} \int_0^t |h_2(\tau)| \, d \tau 
  \quad \forall t \in [0, T]
$$
one can verify that $\| h_2 \|_{1, p} \le L \| v_2 \|_q$ for some $L > 0$ (see Remark~\ref{Remark_SensitivityProperty}).
Therefore there exists $\eta_3 > 0$ such that for any $h_T \in \mathbb{R}^d$ one can find $v_2 \in L_q^m(0, T)$
satisfying the inequality $\| (h_2, v_2) \|_X \le \eta_3 |h(T)|$, where $h_2$ is a solution of \eqref{LinearizedSystem}
with $v = v_2$ such that $h_2(0) = 0$ and $h_2(T) = h_T$.

Choose $r_1, r_2 > 0$, $w \in L_p^d(0, T)$ with $\| w \|_p \le r_1$, and $h_T \in \mathbb{R}^d$ with $|h_T| \le r_2$.
As we proved earlier, there exists $(h_1, v_1) \in X$ satisfying \eqref{PerturbedLinearizedEquation} and such that
$\| (h_1, v_1) \|_X \le (\eta_1 + 1) \| w \|_p \le (\eta_1 + 1) r_1$. By inequality~\eqref{SobolevImbedding} one has 
$\| h_1 \|_{\infty} \le C_p \| h_1 \|_{1, p} \le C_p (\eta_1 + 1) r_1$ for some $C_p > 0$ independent of $h_1$.
Furthermore, there exists $(h_2, v_2) \in X_0$ satisfying \eqref{LinearizedSystem}, and such that $h(0) = 0$,
$h_2(T) = h_T - h_1(T)$, and $\| (h_2, v_2) \|_X \le \eta_3 |h_T - h_1(T)|$. Hence, in particular, one gets that
$$
  \| h_2 \|_{\infty} \le C_p \eta_3 |h_T - h_1(T)| \le C_p \eta_3 |h_T| + C_p \eta_3 \| h_1 \|_{\infty}
  \le C_p \eta_3 r_2 + C_p^2 \eta_3 (\eta_1 + 1) r_1.
$$
Finally, by our assumption there exists $(h_3, v_3) \in X_0$ satisfying \eqref{LinearizedSystem},
\eqref{MFCQ_StateConstr} and such that  $h_3(T) = 0$. For any $j \in J$ denote 
$T_j = \{ t \in [0, T] \mid g_j(x^*(t), t) = 0 \}$. Clearly, the sets $T_j$ are compact, which implies that for
any $j \in J$ there exists $\beta_j > 0$ such that $\langle \nabla_x g_j(x^*(t), t), h_3(t) \rangle \le - \beta_j$ for
all $t \in T_j$ due to \eqref{MFCQ_StateConstr} and the continuity of the functions $\nabla_x g_j$, $x^*$, and $h_3$.
With the use of the compactness of the sets $T_j$ one obtains that for any $j \in J$ there exists a set 
$\mathcal{O}_j \subset [0, T]$ such that $\mathcal{O}_j$ is open in $[0, T]$, $T_j \subset \mathcal{O}_j$, and 
$\langle \nabla_x g_j(x^*(t), t), h_3(t) \rangle \le - \beta_j / 2$ for all $t \in \mathcal{O}_j$. On the other hand,
for any $j \in J$ there exists $\gamma_j > 0$ such that $g_j(x^*(t), t) \le - \gamma_j$ for any 
$t \in [0, T] \setminus \mathcal{O}_j$, since by definition $g_j(x^*(t), t) < 0$ for all $t \notin T_j$ and the set 
$[0, T] \setminus \mathcal{O}_j$ is compact.

Note that for any $\alpha > 0$ the pair $(\alpha h_3, \alpha v_3)$ belongs to $X_0$ and satisfies
\eqref{LinearizedSystem} and the equality $\alpha h_3(T) = 0$. Choosing a sufficiently small $\alpha > 0$ one can
suppose that $\langle \nabla_x g_j(x^*(t), t), \alpha h_3(t) \rangle < \gamma_j$ for all 
$t \in [0, T] \setminus \mathcal{O}_j$ and $j \in J$, while for any $t \in \mathcal{O}_j$ one has
$\langle \nabla_x g_j(x^*(t), t), \alpha h_3(t) \rangle \le - \alpha \beta_j / 2$. Thus, replacing $(h_3, v_3)$ with
$(\alpha h_3, \alpha v_3)$, where $\alpha > 0$ is small enough, one can suppose that
$$
  \langle \nabla_x g_j(x^*(t), t), h_3(t) \rangle + g_j(x^*(t), t) < 0
  \quad  \forall t \in [0, T] \quad \forall j \in J.
$$
With the use of the continuity of $g_j$, $\nabla_x g_j$, $x^*$, and $h_3$ one obtains that there exists $r_3 > 0$ such
that
$$
  \langle \nabla_x g_j(x^*(t), t), h_3(t) \rangle + g_j(x^*(t), t) \le - r_3
  \quad \forall t \in [0, T] \quad \forall j \in J.
$$
Choosing $r_1 > 0$ and $r_2 > 0$ sufficiently small one gets that for any $j \in J$
\begin{equation} \label{NegConstrShift}
  \Big\langle \nabla_x g_j(x^*(t), t), h_1(t) + h_2(t) + h_3(t) \Big\rangle 
  + g_j(x^*(t), t) \le - \frac{r_3}{2}
  \quad \forall t \in [0, T],
\end{equation}
since $\| h_1 \|_{\infty}$ and $\| h_2 \|_{\infty}$ can be made arbitrarily small by a proper choice of $r_1$ and $r_2$.

Define $h = h_1 + h_2 + h_3$ and $v = v_1 + v_2 + v_3$. Then $(h, v) \in X$, $h(0) = 0$, $h(T) = h_T$, $(h, v)$
satisfies \eqref{PerturbedLinearizedEquation}, and \eqref{NegConstrShift} holds true. Therefore,
$(w, h_T, y)^T \in K(x^*, u^*)$ for any $y = (y_1, \ldots, y_l)^T \in (C[0, T])^l$ such that 
$\| y_j \|_{\infty} \le r_3 / 2$ for all $j \in J$ (see \eqref{RegularityCone_StateConstr}). In other words,
$B(0, r_1) \times B(0, r_2) \times B(0, r_3 / 2) \subset K(x^*, u^*)$, i.e. $0 \in \interior K(x^*, u^*)$, and the proof
is complete.
\end{proof}

\begin{remark}
{(i)~From \eqref{MFCQ_StateConstr} it follows that $g_j(x_0, 0) < 0$ and $g_j(x_T, T) < 0$ for all $j \in J$, since
$h(0) = h(T) = 0$ in \eqref{MFCQ_StateConstr}. Furthermore, the assumption that there exists a control input $v$ such
that the corresponding solution $h$ of the linearised system satisfies \eqref{MFCQ_StateConstr} is, roughly speaking,
equivalent to the assumption that there exists $(h, v) \in X$ such that for any sufficiently small $\alpha \ge 0$ the
point $(x_{\alpha}, u_{\alpha}) = (x + \alpha h + r_1(\alpha), u + \alpha v + r_2(\alpha))$ is feasible for problem
\eqref{StateConstrainedProblem} for some $(r_1(\alpha), r_2(\alpha)) \in X$ such that 
$\| (r_1(\alpha), r_2(\alpha)) \|_X / \alpha \to 0$ as $\alpha \to +0$, and $g_j(x_{\alpha}(t), t) < 0$ for all
$t \in [0, T]$, $j \in J$ and for any sufficiently small $\alpha$. Thus, assumption \eqref{MFCQ_StateConstr} is, in
essence, a local version of Slater's condition in the nonconvex case.
}

\noindent{(ii)~It should be noted that in the case when there is no terminal constraint the complete controllability
assumption and the assumptions that the equality $h(T) = 0$ holds true for $h$ satisfying \eqref{MFCQ_StateConstr}
can be dropped from Theorem~\ref{Theorem_StateConstr_LocalExact}.
}

\noindent{(iii)~One might want to use the cone 
$L^r(0, T)_- = \{ x \in L^r(0, T) \mid x(t) \le 0 \text{ for a.e. } t \in (0, T) \}$ instead of $C_-[0, T]$ in the proof
of Theorem~\ref{Theorem_StateConstr_LocalExact} in order verify the local exactness of the penalty function for problem
\eqref{StateConstrainedProblem} with the penalty term
$\varphi(x, u) = \sum_{i = 1}^l ( \int_0^T \max\{ g_j(x(t), t), 0 \}^r \, dt )^{1/r}$, $1 \le r < + \infty$. However,
note that the cone $L^r(0, T)_-$ has empty algebraic interior, and for that reason an attempt to apply
Theorem~\ref{Theorem_LocalErrorBound} leads to incompatible assumptions on the state constraints and the linearised
system. Indeed, in this case the regularity condition \eqref{MetricRegCond} from Theorem~\ref{Theorem_LocalErrorBound}
takes the form
$$
  0 \in \core \left\{ \begin{pmatrix} \dot{h}(\cdot) - A(\cdot) h(\cdot) - B(\cdot) v(\cdot) \\ h(T) \\
  \nabla_x g(x^*(\cdot), \cdot) h(\cdot) \end{pmatrix} -
  \begin{pmatrix} 0 \\ 0 \\ (L^r(0, T)_-)^l - g(x^*(\cdot), \cdot) \end{pmatrix} \Biggm| (h, v) \in X, \: h(0) = 0
  \right\}.
$$
Hence, in particular, $0 \in \core K_0(x^*)$, where 
$K_0(x^*)$ is the union of the cones 
$\{ \nabla_x g(x^*(\cdot), \cdot) h(\cdot) + g(x^*(\cdot), \cdot) \} - (L^r(0, T)_-)^l$ with $h$ being a solution of
\eqref{LinearizedSystem} for some $v \in L^m_q(0, T)$ such that $h(0) = h(T) = 0$. However, for the function 
$y(t) = - t^{1/2r}$ one obviously has $y \in L^r(0, T)$ and $\alpha y \notin K_0(x^*)$ for any $\alpha > 0$ (for the
sake of simplicity we assume that $l = 1$), since the function  
$\nabla_x g(x^*(\cdot), \cdot) h(\cdot) + g(x^*(\cdot), \cdot)$ is continuous and $h(0) = 0$. Thus, 
$0 \notin \core K_0(x^*)$, and Theorem~\ref{Theorem_LocalErrorBound} cannot be applied.
}
\end{remark}

\subsection{Nonlinear Systems: Complete Exactness}

Now we turn to the derivation of sufficient conditions for the complete exactness of the penalty function 
$\Phi_{\lambda}$ for problem \eqref{StateConstrainedProblem}. As in the case of terminal constraints, the derivation of
easily verifiable conditions for exact penalisation of pointwise state constraints does not seem possible in the
nonlinear case. Therefore, our main goal, once again, is not to obtain easily verifiable conditions, but to
understand what kind of general properties the nonlinear system and state constraints must possess to ensure exact
penalisation. To this end, we directly apply Theorem~\ref{Theorem_CompleteExactness} in order to obtain general
sufficient conditions for complete exactness. Then we consider a particular case in which one can obtain more readily
verifiable sufficient conditions for the complete exactness of the penalty function.

Recall that \textit{the contingent cone} to a subset $C$ of a normed space $Y$ at a point $x \in C$, denoted by
$K_C(x)$, consists of all those vectors $v \in Y$ for which there exist sequences $\{ v_n \} \subset Y$ and 
$\{ \alpha_n \} \subset (0, + \infty)$ such that $v_n \to v$ and $\alpha_n \to 0$ as $n \to \infty$, and 
$x + \alpha_n v_n \in C$ for all $n \in \mathbb{N}$. It should be noted that in the case $U = L^m_q(0, T)$ by
the Lyusternik-Graves theorem (see, e.g. \cite{Ioffe}) for any $(x, u) \in A$ one has
\begin{align*}
  K_A(x, u) = \Big\{ &(h, v) \in X \Bigm| \\
  &\dot{h}(t) = \nabla_x f(x(t), u(t), t) h(t) + \nabla_u f(x(t), u(t), t) v(t) \text{ for a.e. } t \in (0, T),
  h(0) = h(T) = 0 \Big\}
\end{align*}
provided the linearised system is completely controllable, and the assumptions of
Theorem~\ref{Theorem_DiffNemytskiiOperator} hold true. In the case when there is no terminal constraint, the complete
controllability assumption and the condition $h(T) = 0$ are redundant.

For any $(x, u) \in X$ denote $\phi(x, t) = \max_{j \in J} \max\{ g_j(x, t), 0 \}$. 
Then $\varphi(x, u) = \max_{t \in [0, T]} \phi(x(t), t)$. Define 
$T(x) = \{ t \in [0, T] \mid \phi(x(t), t) = \varphi(x, u) \}$ and 
$J(x, t) = \{ j \in J \mid g_j(x(t), t) = \varphi(x, u) \}$. Clearly, $J(x, t) \ne \emptyset$ iff $t \in T(x)$.

Let $\mathcal{I}^*$ be the optimal value of problem \eqref{StateConstrainedProblem}. Note that the set 
$\Omega_{\delta} = \{ (x, u) \in A \mid \varphi(x, u) < \delta \}$ consists of all those trajectories $x(\cdot)$ of the
system $\dot{x} = f(x, u, t)$, $u \in U$, $x(0) = x_0$, $x(T) = x_T$, which satisfy the perturbed state constraints
$g_j(x(t), t) < \delta$ for all $t \in [0, T]$ and $j \in J$.

\begin{theorem} \label{Theorem_StateConstrainedProblem_Nonlinear}
Let the following assumptions be valid:
\begin{enumerate}
\item{$\theta$ is continuous and differentiable in $x$ and $u$, the functions $g_j$, $j \in J$, are continuous,
differentiable in $x$, and the functions $\nabla_x \theta$, $\nabla_u \theta$, $\nabla_x g_j$, and $f$ are continuous;
}

\item{either $q = + \infty$ or $\theta$ and $\nabla_x \theta$ satisfy the growth condition of order $(q, 1)$, while
$\nabla_u \theta$ satisfies the growth condition of order $(q - 1, q')$;
}

\item{there exists a globally optimal solution of problem \eqref{StateConstrainedProblem};}

\item{there exist $\lambda_0 > 0$, $c > \mathcal{I}^*$, and $\delta > 0$ such that the set 
$S_{\lambda_0}(c) \cap \Omega_{\delta}$ is bounded in $X$, and the function $\Phi_{\lambda_0}$ is bounded below on $A$;
\label{Assumpt_StateConstr_SublevelBounded}
}

\item{there exists $a > 0$ such that for any $(x, u) \in S_{\lambda_0}(c) \cap (\Omega_{\delta} \setminus \Omega)$
one can find $(h, v) \in K_A(x, u)$ such that for any $t \in T(x)$ one has
\begin{equation} \label{MFCQ_StateConstr_Nonlocal}
  \langle \nabla_x g_j(x(t), t), h(t) \rangle \le - a \| (h, v) \|_X \quad \forall j \in J(x, t).
\end{equation}
\label{Assumpt_StateConstr_Decay}
}
\end{enumerate}
Then there exists $\lambda^* \ge 0$ such that for any $\lambda \ge \lambda^*$ the penalty function $\Phi_{\lambda}$
for problem \eqref{StateConstrainedProblem} is completely exact on $S_{\lambda}(c)$. 
\end{theorem}

\begin{proof}
By \cite[Propositions~3 and 4]{DolgopolikFominyh} the functional $\mathcal{I}$ is Lipschitz continuous on any bounded
open set containing the set $S_{\lambda_0}(c) \cap \Omega_{\delta}$ due to the growth conditions on the function
$\theta$ and its derivatives. Arguing in the same way as in the proof of
Theorem~\ref{Theorem_FixedEndPointProblem_NonLinear} one can easily verify that the continuity of the function $f$ along
with the closedness of the set $U$ ensure that the set $A$ is closed. The continuity of the penalty term $\varphi$ on
$X$ follows from Corollary~\ref{Corollary_StateConstrPenTerm_Contin} (see Appendix~B).

Thus, by Theorem~\ref{Theorem_CompleteExactness} it is sufficient to check that there exists $a > 0$ such that
$\varphi^{\downarrow}_A(x, u) \le - a$ for all 
$(x, u) \in S_{\lambda_0}(c) \cap (\Omega_{\delta} \setminus \Omega)$. Our aim is to show that
assumption~\ref{Assumpt_StateConstr_Decay} ensures the validity of this inequality. 

Fix any $(x, u) \in S_{\lambda_0}(c) \cap (\Omega_{\delta} \setminus \Omega)$, and let $(h, v) \in K_A(x, u)$ be from
assumption~\ref{Assumpt_StateConstr_Decay}. Then by the definition of contingent cone there exist sequences 
$(h_n, v_n) \subset X$ and $\{ \alpha_n \} \subset (0, + \infty)$ such that $(h_n, v_n) \to (h, v)$ and 
$\alpha_n \to 0$ as $n \to \infty$, and for all $n \in \mathbb{N}$ one has $(x + \alpha_n h_n, u + \alpha_n v_n) \in A$.

Denote $\phi_j(x) = \max_{t \in [0, T]} g_j(x(t), t)$. Observe that $\varphi(x, u) = \max_{j \in J} \phi_j(x)$, since
$\varphi(x, u) > 0$ (recall that $(x, u) \notin \Omega$). It is well-known (see, e.g. 
\cite[Sects.~4.4 and 4.5]{IoffeTihomirov}) that the following equality holds true:
$$
  \lim_{n \to \infty} \frac{\phi_j(x + \alpha_n h_n) - \phi_j(x)}{\alpha_n} 
  = \max_{t \in T_j(x)} \langle \nabla_x g_j(x(t), t), h(t) \rangle, \quad
  T_j(x) = \{ t \in [0, T] \mid g_j(x(t), t) = \phi_j(x) \}.
$$
Hence by the Danskin-Demyanov theorem (see, e.g. \cite[Theorem~4.4.3]{IoffeTihomirov}) one has
$$
  \lim_{n \to \infty} \frac{\varphi(x + \alpha_n h_n, u + \alpha_n v_n) - \varphi(x, u)}{\alpha_n}
  = \max_{j \in J(x)} \max_{t \in T_j(x)} \langle \nabla_x g_j(x(t), t), h(t) \rangle, \quad
  J(x) = \{ j \in J \mid \phi_j(x) = \varphi(x, u) \}.
$$
Observe that
$T(x) = \cup_{j \in J(x)} T_j(x)$, and $j \in J(x, t)$ for some $t \in T(x)$ iff $j \in J(x)$ and $t \in T_j(x)$.
Consequently, by applying assumption~\ref{Assumpt_StateConstr_Decay} one obtains that
$$
  \varphi^{\downarrow}_A(x, u) \le 
  \lim_{n \to \infty} \frac{\varphi(x + \alpha_n h_n, u + \alpha_n v_n) - \varphi(x, u)}{\alpha_n \| (h_n, v_n) \|_X}
  = \frac{1}{\| (h, v) \|_X} \max_{j \in J(x)} \max_{t \in T_j(x)} \langle \nabla_x g_j(x(t), t), h(t) \rangle
  \le -a,
$$
and the proof is complete.
\end{proof}

Clearly, the main assumption ensuring that the penalty function $\Phi_{\lambda}$ for problem
\eqref{StateConstrainedProblem} is completely exact is assumption~\ref{Assumpt_StateConstr_Decay}. This assumption can
be easily explained in the case $l = 1$, i.e. when there is only one state constraint. Roughly speaking, in this case
assumption~\ref{Assumpt_StateConstr_Decay} means that if a trajectory $x(\cdot)$ of the system $\dot{x} = f(x, u, t)$
with $x(0) = x_0$ and $x(T) = x_T$ slighly violates the constraints (i.e. $g_1(x(t), t) < \delta$ for all $t$), then by
changing the control input $u$ in such a way that the endpoint conditions $x(0) = x_0$ and $x(t) = x_T$ remain to hold
true one must be able to slighly shift the trajectory $x(t)$ in a direction close to $- \nabla_x g_1(x(t), t)$ at those
points $t$ for which the constraint violation measure $\phi(x(t), t) = \max\{ 0, g_1(x(t), t) \}$ is the largest.
However, note that this shift must be uniform for all 
$(x, u) \in S_{\lambda_0}(c) \cap (\Omega_{\delta} \setminus \Omega)$ in the sense that inequality 
\eqref{MFCQ_StateConstr_Nonlocal} must hold true for all those $(x, u)$. The validity of this inequality in a
neighbourhood of a given point can be verified with the use of the same technique as in the proof of
Theorem~\ref{Theorem_StateConstr_LocalExact}. Namely, one can check that in the case $U = L^m_q(0, T)$ inequality
\eqref{MFCQ_StateConstr_Nonlocal} holds true in a neighbourhood of a given point $(\widehat{x}, \widehat{u})$, provided
there exists a solution $(h, v)$ of the corresponding linearised system such that $h(0) = h(T) = 0$ and 
$\langle \nabla_x g_j(\widehat{x}(t), t), h(t) \rangle < 0$ for all $t \in T(x)$ and $j \in J(x, t)$. Consequently, the
main difficulty in verifying assumption~\ref{Assumpt_StateConstr_Decay} stems from the fact that the validity of
inequality \eqref{MFCQ_StateConstr_Nonlocal} must be checked not locally, but on the set
$S_{\lambda_0}(c) \cap (\Omega_{\delta} \setminus \Omega)$. Let us briefly discuss a particular case in which one can
easily verify that assumption~\ref{Assumpt_StateConstr_Decay} holds true.

\begin{example} \label{Example_LTV_SlaterImpliesExactPen}
Suppose that the system is linear, i.e. $f(x, u, t) = A(t) x + B(t) u$, the set $U$ of admissible control inputs is
convex, the functions $g_j(x, t)$ are convex in $x$, and Slater's condition holds true, i.e. there exists 
$(\widehat{x}, \widehat{u}) \in A$ such that $g_j(\widehat{x}(t), t) < 0$ for all $ t \in [0, T]$ and $j \in J$. Choose
any $(x, u) \in S_{\lambda_0}(c) \cap (\Omega_{\delta} \setminus \Omega)$. For any $n \in \mathbb{N}$ define 
$\alpha_n = 1 / n$, $(h, v) = (\widehat{x} - x, \widehat{u} - u)$, and 
$(x_n, u_n) = \alpha_n (\widehat{x}, \widehat{u}) + (1 - \alpha_n) (x, u) = (x, u) + \alpha_n (h, v)$. 
Then $(x_n, u_n) \in A$ for all $n \in \mathbb{N}$ and $(h, v) \in K_A(x, u)$ due to the convexity of the set $U$ and
the linearity of the system. Fix any $j \in J$ and $t \in [0, T]$ such that $g_j(x(t), t) \ge 0$. Due to the
convexity of $g_j(x, t)$ in $x$ one has
\begin{align*}
  \langle \nabla_x g_j(x(t), t), \alpha_n h(t) \rangle 
  &\le g_j(x(t) + \alpha_n h(t), t) - g_j(x(t), t) \\
  &\le \alpha_n g_j(\widehat{x}(t), t) + (1 - \alpha_n) g_j(x(t), t) - g_j(x(t), t)
  \le \alpha_n g_j(\widehat{x}(t), t) \le \alpha_n \eta,
\end{align*}
where $\eta = \max_{j \in J} \max_{t \in [0, T]} g_j(\widehat{x}(t), t)$. Note that $\eta < 0$ by Slater's condition.
Consequently, for any $t \in T(x)$ and $j \in J(x, t)$ one has
$$
  \langle \nabla_x g_j(x(t), t), h(t) \rangle \le \frac{\eta}{\| (h, v) \|_X} \| (h, v) \|_X.
$$
By assumption~\ref{Assumpt_StateConstr_SublevelBounded} of Theorem~\ref{Theorem_StateConstrainedProblem_Nonlinear} the
set $S_{\lambda_0}(c) \cap (\Omega_{\delta} \setminus \Omega)$ is bounded. Therefore, there exists $C > 0$ such that
$\| (h, v) \|_X = \| (\widehat{x} - x, \widehat{u} - u) \|_X \le C$ for all 
$(x, u) \in S_{\lambda_0}(c) \cap (\Omega_{\delta} \setminus \Omega)$. Hence assumption~\ref{Assumpt_StateConstr_Decay}
of Theorem~\ref{Theorem_StateConstrainedProblem_Nonlinear} is satisfied with $a = |\eta| / C$.
\end{example}

\begin{remark}
Apparently, assumption~\ref{Assumpt_StateConstr_Decay} of Theorem~\ref{Theorem_StateConstrainedProblem_Nonlinear} holds
true in a much more general case than the case of optimal control problems for linear systems with convex state
constraints. In particular, it seems that in the case when $j = 1$ (i.e. there is only one state constraint), 
$g_1(x_0, 0) < 0$, $g_1(x_T, T) < 0$, and $\inf |\nabla_x g_1(x, t)| > 0$, where the infimum is taken over all those 
$t \in [0, T]$ and $x \in \mathbb{R}^d$ for which $0 < g_1(x, t) < \delta$, assumption~\ref{Assumpt_StateConstr_Decay}
of Theorem~\ref{Theorem_StateConstrainedProblem_Nonlinear} is satisfied under very mild assumptions on the system. On
the other hand, if either initial or terminal states lie on the boundary of the feasible region (i.e. either 
$g_1(x_0, 0) = 0$ or $g_1(x_T, T) = 0$), then assumptions~\ref{Assumpt_StateConstr_Decay} cannot be satisfied. A
detailed analysis of these conditions lies outside the scope of this article, and we leave it as a challenging open
problem for future research.
\end{remark}

\begin{remark}
One can easily extend the proof of Theorem~\ref{Theorem_StateConstrainedProblem_Nonlinear} to the case when 
the penalty term $\varphi$ is defined as $\varphi(x, u) = \| \phi(x(\cdot), \cdot) \|_r$ for some $r \in (1, + \infty)$,
where $\phi(x, t) = \max_{j \in J} \max\{ g_j(x, t), 0 \}$ (i.e. the state constraints are penalised via the
$L^r$-norm). In this case assumption~\ref{Assumpt_StateConstr_Decay} takes the following form: there exists $a > 0$
such that for any $(x, u) \in S_{\lambda_0}(c) \cap (\Omega_{\delta} \setminus \Omega)$ one can find 
$(h, v) \in K_A(x, u)$ satisfying the inequality 
\begin{equation} \label{StateConstr_Decay_Impossible}
  \frac{1}{\varphi(x, u)^{r - 1}} \int_0^T \phi(x(t), t)^{r - 1} 
  \max_{j \in J_0(x(t), t)} \langle \nabla_x g_j(x(t), t), h(t) \rangle \, dt \le - a \| (h, v) \|_X,
\end{equation}
where $J_0(x, t) = \{ j \in J \cup \{ 0 \} \mid g_j(x, t) = \phi(x, t) \}$ and $g_0(x, t) \equiv 0$. However, the author
failed to find any optimal control problems for which this assumptions can be verified. 
\end{remark}

\subsection{Nonlinear Systems: A Different View}

As Examples~\ref{CounterExample_StateEqConstr} and \ref{CounterExample_StateInEqConstr} demonstrate, penalty functions
for problems with state constraints may fail to be exact due to the fact that the penalty term $\varphi$, unlike the
cost functional $\mathcal{I}(x, u)$, does not depend on the control inputs $u$ explicitly. In the case when
$\mathcal{I}$ does not explicitly depend on $u$, one can utilise a somewhat different approach and obtain stronger
results on the exactness of penalty functions for state constrained problems. Furthermore, this approach serves as a
proper motivation to consider a general theory of exact penalty functions in the \textit{metric} space setting (as it is
done in Section~\ref{Sect_ExactPenaltyFunctions}), but not in the normed space setting.

Consider the following variable-endpoint optimal control problem with state inequality constraints:
\begin{equation} \label{VariableEndpointProblem}
\begin{split}
  &\min \: \mathcal{I}(x) = \int_0^T \theta(x(t), t) \, dt + \zeta(x(T)) \quad 
  \text{subject to} \quad \dot{x}(t) = f(x(t), u(t), t), \quad t \in [0, T], \\
  &x(0) = x_0, \quad x(T) \in S_T, \quad u \in U, \quad
  g_j(x(t), t) \le 0 \quad \forall t \in [0, T], \: j \in J.
\end{split}
\end{equation}
Here $\theta \colon \mathbb{R}^d \times [0, T] \to \mathbb{R}$, $\zeta \colon \mathbb{R}^d \to \mathbb{R}$, 
$f \colon \mathbb{R}^d \times \mathbb{R}^m \times [0, T] \to \mathbb{R}^d$, and $g_j \colon \mathbb{R}^d \times [0, T]
\to \mathbb{R}$, $j \in J = \{ 1, \ldots, l \}$, are given functions, 
$x_0 \in \mathbb{R}^d$ and $T > 0$ are fixed, while $S_T \subseteq \mathbb{R}^d$ and
$U \subseteq L_q^m(0, T)$ are closed sets. It should be noted that with the use of the standard time scaling
transformation time-optimal control problems can be recast as problems of the form \eqref{VariableEndpointProblem}.

We will treat problem \eqref{VariableEndpointProblem} as a variational problem, not as an optimal control one. To this
end, fix some $p \in (1, + \infty]$, and define
$$
  X = \Big\{ x \in (C[0, T])^d \Bigm| \exists u \in U \colon 
  x(t) = x_0 + \int_0^t f(x(\tau), u(\tau), \tau) \, d \tau \quad \forall t \in [0, T] \Big\},
$$
i.e. $X$ is the set of trajectories of the controlled system under consideration. We equip $X$ with the metric 
$d_X(x, y) = \| x - y \|_p + |x(T) - y(T)|$. Define $A = \{ x \in X \mid x(T) \in S_T \}$ and 
$M = \{ x \in X \mid g_j(x(t), t) \le 0 \: \forall t \in [0, T], \: j \in J \}$. Then problem
\eqref{StateConstrainedProblem} can be rewritten as the problem of minimising $\mathcal{I}(x)$, $x \in X$, over the set 
$M \cap A$. Observe that the set $A$ is closed in $X$ due to the facts that the set $S_T$ is closed, and if a sequence
$\{ x_n \}$ converges to some $x$ in the metric space $X$, then $\{ x_n(T) \}$ converges to $x(T)$. Let us also point
out simple sufficient conditions for the metric space $X$ to be complete.

\begin{proposition} \label{Prop_CompleteMetricSpace}
Let the function $f$ be continuous and 
$U = \{ u \in L^m_{\infty}(0, T) \mid u(t) \in Q \text{ for a.e. } t \in (0, T) \}$ for some compact convex set
$Q \subset \mathbb{R}^m$. Suppose also that the set $f(x, Q, t)$ is convex for all $x \in \mathbb{R}^d$ and 
$t \in [0, T]$, and for any $u \in U$ a solution of $\dot{x} = f(x, u, t)$ with $x(0) = x_0$ is defined on $[0, T]$.
Then $X$ is a complete metric space and a compact subset of $(C[0, T])^d$.
\end{proposition}

\begin{proof}
Under our assumptions the space $X$ consists of all solutions of the differential inclusion $\dot{x} \in F(x, t)$,
$x(0) = x_0$, with $F(x, t) = f(x, Q, t)$ by Filippov's theorem (see, e.g. \cite[Theorem~8.2.10]{AubinFrankowska}).
Furthermore, by \cite[Theorem~2.7.6]{Filippov} the set $X$ is compact in $(C[0, T])^d$.

Let $\{ x_n \} \subset X$ be a Cauchy sequence in $X$. Since $X$ is compact in $(C[0, T])^d$, there exists a
subsequence $\{ x_{n_k} \}$ uniformly converging to some some $x^* \in X$, which obviously implies that $\{ x_{n_k} \}$
converges to $x^*$ in $X$. Hence with the use of the fact that $x_n$ is a Cauchy sequence in $X$ one can easily check
that $x_n$ converges to $x^*$ in $X$. Thus, $X$ is a complete metric space.
\end{proof}

Formally introduce the penalty term 
$$
  \varphi(x) = \| \phi(x(\cdot), \cdot) \|_p \quad \forall x \in X, \quad
  \phi(x, t) = \max\{ g_1(x, t), \ldots, g_l(x, t), 0 \} \quad \forall x \in \mathbb{R}^d, \: t \in [0, T]
$$
(note that here $p$ is the same as in the definition of metric in $X$). Then $M = \{ x \in X \mid \varphi(x) = 0 \}$,
and one can consider the penalised problem of minimising the penalty function 
$\Phi_{\lambda}(x) = \mathcal{I}(x) + \lambda \varphi(x)$ over the set $A$, which in the case $p < + \infty$
can be formally written as follows:
\begin{equation} \label{PenalizedProblem_StateConstraints}
\begin{split}
  {}&\min \: \Phi_{\lambda}(x) = \int_0^T \theta(x(t), t) \, dt +
  \lambda \bigg( \int_0^T \max\{ g_1(x(t), t), \ldots, g_l(x(t), t), 0 \}^p \, dt \bigg)^{1/p} + \zeta(x(T)) \\
  {}&\text{subject to } x(t) = x_0 + \int_0^t f(x(\tau), u(\tau), \tau) \, d \tau, \quad t \in [0, T], \quad
  x(T) \in S_T, \quad u \in U, \quad x \in (C[0, T])^d.
\end{split}
\end{equation}
Note, however, that due to our choice of the space $X$ and the metric in this space the notions of locally optimal
solutions/inf-stationary points of this problem (and problem \eqref{VariableEndpointProblem}) are understood in a rather
specific sense. In particular, $(x^*, u^*)$ is a locally optimal solution of this problem iff for any feasible point
$(x, u)$ satisfying the inequality $\| x - x^* \|_p + |x(T) - x^*(T)| < r$ for some $r > 0$ one has 
$\Phi_{\lambda}(x) \ge \Phi_{\lambda}(x^*)$. It should be mentioned that any locally optimal solution/inf-stationary
point of problem \eqref{VariableEndpointProblem} (or \eqref{PenalizedProblem_StateConstraints}) in $X$ is also a locally
optimal solution/inf-stationary point of problem \eqref{VariableEndpointProblem} (or
\eqref{PenalizedProblem_StateConstraints}) in the space $W^d_{1, p}(0, T) \times L^m_q(0, T)$, but the converse
statement is not true. In a sense, one can say that our choice of the underlying space $X$ in this section reduces 
the number of locally optimal solutions/inf-stationary points (and, as a result, leads to the weaker notion of the
complete exactness of $\Phi_{\lambda}$ than in the previous section). 

Let us derive sufficient conditions for the complete exactness of the penalty function $\Phi_{\lambda}$ for problem
\eqref{VariableEndpointProblem}. To conveniently formulate these conditions, define $g_0(x, t) \equiv 0$. 
Then $\phi(x, t) \equiv \max\{ g_j(x, t) \mid j \in J \cup \{ 0 \} \}$. For any $x \in \mathbb{R}^d$ and $t \in [0, T]$
let $J(x, t) = \{ j \in J \cup \{ 0 \} \mid \phi(x, t) = g_j(x, t) \}$. Finally, suppose that the functions $g_j$ are
differentiable in $x$, and define the subdifferential $\partial_x \phi(x, t)$ of the function $x \mapsto \phi(x, t)$ as
follows:
\begin{equation} \label{SubdiffOfMaxStateConstr}
  \partial_x \phi(x, t) = \co\big\{ \nabla_x g_j(x, t) \bigm| j \in J(x, t) \big\}.
\end{equation}
Let us point out that $\partial_x \phi(x, t)$ is a convex compact set, and $\partial_x \phi(x, t) = \{ 0 \}$, if 
$g_j(x, t) < 0$ for all $j \in J$.

Denote by $\mathcal{I}^*$ the optimal value of problem \eqref{VariableEndpointProblem}, and recall that
$\Omega_{\delta} = \{ x \in A \mid \varphi(x) < \delta \}$. Observe that in the case $p = + \infty$ the set 
$\Omega_{\delta}$ consists of all those trajectories $x(\cdot)$ of the system that satisfy the perturbed constraints
$g_j(x(t), t) < \delta$ for all $t \in [0, T]$ and $j \in J$. In the case $p < + \infty$ the set $\Omega_{\delta}$
consists of all those trajectories $x(\cdot)$ for which there exists $w \in L^p(0, T)$ with $\| w \|_p < \delta$ such
that $g_j(x(t), t) \le w(t)$ for all $t \in [0, T]$ and $j \in J$, which implies that at every point $t \in [0, T]$ the
violation of the state constraints can be arbitrarily large, i.e. $\phi(x(t), t)$ can be arbitrarily large as long as 
$\| \phi(x(\cdot), \cdot) \|_p < \delta$. 

To avoid the usage of some complicated and restrictive assumptions on the problem data, we prove the following theorem
in the simplest case when the set $X$ is compact in $(C[0, T])^d$. This assumption holds true, in particular, if
the assumptions of Proposition~\ref{Prop_CompleteMetricSpace} are satisfied.

\begin{theorem} \label{Theorem_StateConstr_Nonlinear_Alternative}
Let $p \in (1, + \infty]$ and the following assumptions be valid:
\begin{enumerate}
\item{$\zeta$ is locally Lipschitz continuous, $\theta$ and $g_j$, $j \in J$ are continuous,
differentiable in $x$, and the functions $\nabla_x \theta$ and $\nabla_x g_j$, $j \in J$, are continuous;
}

\item{the set $X$ is compact in $(C[0, T])^d$, and there exists a feasible point of problem
\eqref{VariableEndpointProblem};
}

\item{there exist $a > 0$ and $\eta > 0$ such that for any $x \in A \setminus \Omega$ one can find a sequence of
trajectories $\{ x_n \} \subseteq A$ converging to $x$ in the space $X$ such that 
$|x_n(T) - x(T)| \le \eta \| x_n - x \|_p$ for all $n \in \mathbb{N}$, the sequence
$\{ (x_n - x) / \| x_n - x \|_p \}$ converges to some $h \in L^d_p(0, T)$, and
\begin{equation} \label{StateConstr_Decay_Finite_P}
  \int_0^T \phi(x(t), t)^{p - 1} \max_{v \in \partial_x \phi(x(t), t)} \langle v, h(t) \rangle \, dt
  \le - a \varphi(x)^{p - 1}
\end{equation}
in the case $1 < p < + \infty$, while
\begin{equation} \label{StateConstr_Decay_Weakened_P_Infty}
  \langle \nabla_x g_j(x(t), t), h(t) \rangle \le - a \quad 
  \forall t \in [0, T], \: j \in J \colon \varphi(x) = g_j(x(t), t)
\end{equation}
in the case $p = + \infty$.
\label{Assumpt_StateConstr_Decay_Finite_p}
}
\end{enumerate}
Then the penalty function $\Phi_{\lambda}$ for problem \eqref{VariableEndpointProblem} is completely exact on $A$. 
\end{theorem}

\begin{proof}
The functional $\mathcal{I}$ is obviously continuous with respect to the uniform metric, which with the use of the fact
that $X$ is compact in $(C[0, T])^d$ implies that the penalty function $\Phi_{\lambda}$ is bounded below on $X$ for any
$\lambda \ge 0$. Moreover, the set $X$ is bounded in $(C[0, T])^d$. Hence by applying the mean value theorem and the
fact that the function $\zeta$ is locally Lipschitz continuous one obtains that there exist $L_{\zeta} > 0$ such that
\begin{align*}
  |\mathcal{I}(x) - \mathcal{I}(y)| &\le \Big| \int_0^T \theta(x(t), t) \, dt - \int_0^T \theta(y(t), t) \, dt \Big| 
  + |\zeta(x(T)) - \zeta(y(T))| \\
  &\le \int_0^T \sup_{\alpha \in [0, 1]} \big|\nabla_x \theta(x(t) + \alpha (y(t) - x(t)), t)\big| |x(t) - y(t)| \, dt
  + L_{\zeta} |x(T) - y(T)| \\
  &\le \max\{ T^{1 / p'} K, L_{\zeta} \} d_X(x, y)
\end{align*}
for all $x, y \in X$, where $K = \max\{ |\nabla_x \theta(z, t)| \colon |z| \le R, \: t \in [0, T] \}$ and
$R > 0$ is such that $\| x \|_{\infty} \le R$ for all $x \in X$. Thus, the functional $\mathcal{I}$ is Lipschitz
continuous on $X$.

Let a sequence $\{ (x_n, u_n) \}$ of feasible points of problem \eqref{VariableEndpointProblem} be such that
$\mathcal{I}(x_n)$ converges to the optimal value $\mathcal{I}^*$ of this problem (recall that we assume that at least
one feasible point exists). Since the set $X$ is compact in $(C[0, T])^d$, one can extract a subsequence $\{ x_{n_k} \}$
uniformly converging to some $x^* \in X$. From the uniform convergence, the continuity of the functions $g_j$, and the
closedness of the set $S_T$ it follows $x^*(T) \in S_T$ and $g_j(x^*(\cdot), \cdot) \le 0$ for all $j \in J$.
Furthermore, by the definition of $X$ there exists $u^* \in U$ such that 
$x^*(t) = x_0 + \int_0^t f(x^*(\tau), u^*(\tau), \tau) \, d \tau$ for all $t \in [0, T]$, which implies that $(x^*,
u^*)$ is a feasible point of problem \eqref{VariableEndpointProblem}. Taking into account the fact that the functional
$\mathcal{I}$ is continuous with respect to the uniform metric one obtains that $\mathcal{I}(x^*) = \mathcal{I}^*$.
Thus, $(x^*, u^*)$ is a globally optimal solution of problem \eqref{VariableEndpointProblem}, i.e. this problem has a
globally optimal solution.

Let us check that the penalty term $\varphi$ is continuous on $X$. Indeed, arguing by reductio ad absurdum, suppose
that $\varphi$ is not continuous at some point $x \in X$. Then there exists $\varepsilon > 0$ and a sequence 
$\{ x_n \} \subset X$ converging to $x$ in the space $X$ such that $|\varphi(x_n) - \varphi(x)| \ge \varepsilon$ for all
$n \in \mathbb{N}$. By applying the compactness of $X$ one obtains that there exists a subsequence $\{ x_{n_k} \}$
uniformly converging to some $\overline{x} \in X$. Clearly, $\{ x_{n_k} \}$ also converges to $\overline{x}$ in 
the space $X$, which implies that $\overline{x} = x$. Utilising the uniform convergence of $\{ x_{n_k} \}$ to $x$ and
the continuity of the functions $g_j$ one can easily prove that $\varphi(x_{n_k}) \to \varphi(x)$ as $k \to \infty$
(see~Corollary~\ref{Corollary_StateConstrPenTerm_Contin}), which contradicts our assumption. Therefore, the penalty term
$\varphi$ is continuous on $X$.

Thus, by Theorem~\ref{THEOREM_COMPLETEEXACTNESS_GLOBAL} it remains to check that there exists $a > 0$ such that
$\varphi^{\downarrow}_A(x) \le - a$ for all $x \in A \setminus \Omega$. Our aim is to show that this inequality is
implied by assumption~\ref{Assumpt_StateConstr_Decay_Finite_p}.

\textbf{The case $p < + \infty$.} Fix any $x \in A \setminus \Omega$, and let $\{ x_n \} \subset A$ and $h$ be from
assumption~\ref{Assumpt_StateConstr_Decay_Finite_p}. Define $\alpha_n = \| x_n - x \|_p$ and 
$h_n = (x_n - x) / \| x_n - x \|_p$. Then $x_n = x + \alpha_n h_n$. Let us verify that
\begin{equation} \label{PenTerm_SC_Hadamard_DD}
  \lim_{n \to \infty} \frac{\varphi(x + \alpha_n h_n) - \varphi(x)}{\alpha_n}
  = \frac{1}{\varphi(x)^{p - 1}} 
  \int_0^T \phi(x(t), t)^{p - 1} \max_{v \in \partial_x \phi(x(t), t)} \langle v, h(t) \rangle \, dt.
\end{equation}
Then by applying \eqref{StateConstr_Decay_Finite_P} and the inequality $|x_n(T) - x(T)| \le \eta \| x_n - x \|_p$
one gets that
$$
  \varphi^{\downarrow}_A(x) \le \liminf_{n \to \infty} \frac{\varphi(x_n) - \varphi(x)}{d_X(x_n, x)}
  = \liminf_{n \to \infty} \frac{\alpha_n}{d_X(x_n, x)} \frac{\varphi(x + \alpha_n h_n) - \varphi(x)}{\alpha_n}
  \le - \frac{a}{1 + \eta},
$$
and the proof is complete.

Instead of proving \eqref{PenTerm_SC_Hadamard_DD}, let us check that 
\begin{equation} \label{PenTerm_SC_in_pth_Hadamard_DD}
  \lim_{n \to \infty} \frac{\varphi(x + \alpha_n h_n)^p - \varphi(x)^p}{\alpha_n}
  = \int_0^T \phi(x(t), t)^{p - 1} \max_{v \in \partial_x \phi(x(t), t)} \langle v, h(t) \rangle \, dt.
\end{equation}
Then taking into account the facts that $\varphi(x) > 0$ (recall that $x \notin \Omega$), and the function 
$\omega(s) = s^{1/p}$ is differentiable at any point $s > 0$ one obtains that \eqref{PenTerm_SC_Hadamard_DD} holds true.

To prove \eqref{PenTerm_SC_in_pth_Hadamard_DD}, note at first that the multifunction 
$t \mapsto \partial_x \phi(x(t), t)$ is upper semicontinuous and thus measurable by
\cite[Proposition~8.2.1]{AubinFrankowska}, which by \cite[Theorem~8.2.11]{AubinFrankowska} implies that the function 
$t \mapsto \max_{v \in \partial_x \phi(x(t), t)} \langle v, h(t) \rangle$ is measurable. Arguing by reductio ad
absurdum, suppose that \eqref{PenTerm_SC_in_pth_Hadamard_DD} does not hold true. Then there exist $\varepsilon > 0$ and
a subsequence $\{ n_k \}$, $k \in \mathbb{N}$, such that
\begin{equation} \label{PenTerm_SC_not_HDD}
  \left| \frac{\varphi(x + \alpha_{n_k} h_{n_k})^p - \varphi(x)^p}{\alpha_{n_k}} -
  \int_0^T \phi(x(t), t)^{p - 1} \max_{v \in \partial_x \phi(x(t), t)} \langle v, h(t) \rangle \, dt \right| 
  \ge \varepsilon
\end{equation}
for all $k \in \mathbb{N}$. Since $h_n$ converges to $h$ in $L^d_p(0, T)$, one can find a subsequence of 
the sequence $\{ h_{n_k} \}$, which we denote once again by $\{ h_{n_k} \}$, that converges to $h$ almost everywhere.
Hence by the Danskin-Demyanov theorem (see, e.g. \cite[Theorem~4.4.3]{IoffeTihomirov}) for a.e. $t \in (0, T)$ one has
$\lim_{k \to \infty} \omega_k(t) = 0$, where
$$
  \omega_k(t) = \frac{\phi(x(t) + \alpha_{n_k} h_{n_k}(t), t)^p - \phi(x(t), t)^p}{\alpha_{n_k}} -
  \phi(x(t), t)^{p - 1} \max_{v \in \partial_x \phi(x(t), t)} \langle v, h(t) \rangle.
$$
With the use of a nonsmooth version of the mean value theorem (see, e.g. \cite[Proposition~2]{Dolgopolik_MCD}) one
obtains that for any $k \in \mathbb{N}$ and a.e. $t \in (0, T)$ there exist $\beta_k(t) \in (0, 1)$ and
$v_k(t) \in \partial_x \phi(x(t) + \beta_k(t) (x_{n_k}(t) - x(t)), t)$ such that
$$
  \omega_k(t) = \Big( \phi(x(t) + \beta_k(t) (x_{n_k}(t) - x(t)), t) \Big)^{p - 1} \langle v_k(t), h_{n_k}(t) \rangle
  - \phi(x(t), t)^{p - 1} \max_{v \in \partial_x \phi(x(t), t)} \langle v, h(t) \rangle.
$$
Consequently, bearing in mind the facts that the set $X$ is compact in $(C[0, T])^d$ and the functions $g_j$ and
$\nabla_x g_j$ are continuous (see~\eqref{SubdiffOfMaxStateConstr}) one obtains that there exists $C > 0$ such
that $|\omega_k(t)| \le C |h_{n_k}(t)| + C |h(t)|$ for all $k \in \mathbb{N}$ and a.e. $t \in (0, T)$. 

The sequence $\{ h_{n_k} \}$ converges to $h$ in $L^d_p(0, T)$, which by H\"{o}lder's inequality implies that it
converges to $h$ in $L^d_1(0, T)$. By the ``only if'' part of Vitali's theorem characterising convergence in
$L^p$-spaces (see, e.g. \cite[Theorem~III.6.15]{DunfordSchwartz}) and the absolute continuity of the Lebesgue integral
for any $\varepsilon > 0$ one can find $\delta(\varepsilon) > 0$ such that for any
Lebesgue measurable set $E \subset [0, T]$ with $\mu(E) < \delta(\varepsilon)$ (here $\mu$ is the Lebesgue measure) one
has $\int_E h_{n_k} d \mu < \varepsilon / 2 C$ and $\int_E h d \mu < \varepsilon / 2 C$. Consequently,
$\int_E \omega_k d \mu < \varepsilon$, provided $\mu(E) < \delta(\varepsilon)$. Hence bearing in mind the fact
$\omega_k(t) \to 0$ for a.e. $t \in (0, T)$ and passing to the limit with the use of ``if'' part of the Vitali theorem
one obtains that $\lim_{k \to \infty} \int_0^T |\omega_k(t)| \, dt = 0$, which contradicts \eqref{PenTerm_SC_not_HDD}.
Thus, \eqref{PenTerm_SC_in_pth_Hadamard_DD} holds true, and the proof of the case $p < + \infty$ is complete.

\textbf{The case $p = + \infty$.} The proof of this case coincides with the derivation of the inequality 
$\varphi^{\downarrow}_A(x, u) \le - a$ within the proof of Theorem~\ref{Theorem_StateConstrainedProblem_Nonlinear}.
\end{proof}

\begin{remark}
{(i)~Let us note that one can define
$$
  \varphi(x) = \bigg( \sum_{j = 1}^l \int_0^T \max\{ g_j(x(t), t), 0 \}^p \, dt \bigg)^{1/p}
  \quad \text{or} \quad
  \varphi(x) = \sum_{j = 1}^l \bigg(\int_0^T \max\{ g_j(x(t), t), 0 \}^p \, dt \bigg)^{1/p},
  \quad 1 < p < + \infty,
$$
and easily obtain corresponding sufficient conditions for the complete exactness of the penalty function
$\Phi_{\lambda}$, which are very similar, but not identical, in all three cases.
}

\noindent{(ii)~It should be mentioned that the term $|x(T) - y(T)|$ was introduced into the definition of the metric
$d_X(x, y) = \| x - y \|_p + |x(T) - y(T)|$ in $X$ to ensure that the functional $\mathcal{I}$ is Lipschitz continuous
on $A$. In the case of problems with the cost functional of the form $\mathcal{I}(x) = \int_0^T \theta(x(t), t) \, dt$
one can define $d_X(x, y) = \| x - y \|_p$ and drop the inequality $|x_n(T) - x(T)| \le \eta \| x_n - x \|_p$ from
assumption~\ref{Assumpt_StateConstr_Decay_Finite_p} of Theorem~\ref{Theorem_StateConstr_Nonlinear_Alternative}. Note
that the closedness of the set $A = \{ x \in X \mid x(T) \in S_T \}$ in the case when $X$ is equipped with the metric
$d_X(x, y) = \| x - y \|_p$ can be easily proved under the assumption that $X$ is compact in $(C[0, T])^d$, since in
this case the topologies on $X$ generated by the metrics $d_X(x, y) = \| x - y \|_p$ and 
$d_X(x, y) = \| x - y \|_{\infty}$ coincide.
}
\end{remark}

At first glance, assumption~\ref{Assumpt_StateConstr_Decay_Finite_p} of the
Theorem~\ref{Theorem_StateConstr_Nonlinear_Alternative} might seem very similar to
assumption~\ref{Assumpt_StateConstr_Decay} of Theorem~\ref{Theorem_StateConstrainedProblem_Nonlinear} and inequality
\eqref{StateConstr_Decay_Impossible}. In particular, arguing in the same way as in
Example~\ref{Example_LTV_SlaterImpliesExactPen} one can check that in the case 
$p = + \infty$ inequality \eqref{StateConstr_Decay_Weakened_P_Infty} is satisfied, provided the system is linear, the
state constraints are convex, and Slater's condition holds true. However, there is one important difference. 
In assumption~\ref{Assumpt_StateConstr_Decay_Finite_p} of Theorem~\ref{Theorem_StateConstr_Nonlinear_Alternative} one
does not need to care about control inputs corresponding to the sequence $\{ x_n \}$, as well as the derivatives
$\dot{x}_n$, which makes this assumption significantly less restrictive, then
assumption~\ref{Assumpt_StateConstr_Decay} of Theorem~\ref{Theorem_StateConstrainedProblem_Nonlinear}. 

\begin{remark} \label{Remark_ShiftTowardsFeasibleRegion}
Let us point out a particular case in which assumption~\ref{Assumpt_StateConstr_Decay_Finite_p} of
Theorem~\ref{Theorem_StateConstr_Nonlinear_Alternative} can be reformulated in a more convenient form. Suppose that 
$p < + \infty$, $l = 1$ (i.e. there is only one state constraint), and there exist $a_1, a_2 > 0$ such that
$a_1 \le |\nabla_x g_1(x, t)| \le a_2$ for all $t \in [0, T]$ and $x \in \mathbb{R}^d$ satisfying the inequality 
$g_1(x, t) > 0$. In this case assumption~\ref{Assumpt_StateConstr_Decay_Finite_p} of
Theorem~\ref{Theorem_StateConstr_Nonlinear_Alternative} is satisfied, if there exists $\eta > 0$ such that 
for any $x \in A \setminus \Omega$ one can find a sequence of trajectories $\{ x_n \} \subset A$ converging to $x$
such that $|x_n(T) - x(T)| \le \eta \| x_n - x \|_p$ for all $n \in \mathbb{N}$, and the sequence
$\{ (x_n - x) / \| x_n - x \|_p \}$ converges to $h = y / \| y \|_p$ with 
$y(t) = - \phi(x(t), t) \nabla_x g_1(x(t), t)$ for all $t \in [0, T]$. Indeed, by applying the inequalities
$a_1 \le |\nabla_x g_1(x, t)| \le a_2$ and the fact that $\varphi(x) = \| \phi(x(\cdot), \cdot) \|_p$ one obtains
$$
  \int_0^T \phi(x(t), t)^{p - 1} \max_{v \in \partial_x \phi(x(t), t)} \langle v, h(t) \rangle \, dt
  = - \frac{\int_0^T \phi(x(t), t)^p | \nabla_x g_1(x(t), t) |^2 \, dt}{\left( \int_0^T \phi(x(t), t)^p 
  | \nabla_x g_1(x(t), t) |^p \, dt \right)^{1/p}} 
  \le - \frac{a_1^2 \varphi(x)^p }{a_2 \varphi(x)} = - \frac{a_1^2}{a_2} \varphi(x)^{p - 1},
$$
i.e. inequality \eqref{StateConstr_Decay_Finite_P} holds true. This assumption, in essence, means that for any
trajectory $x$ violating the state constraint $g_1(x(t), t) \le 0$ one has to be able to find a sequence of control
inputs that shift the trajectory $x$ along the ray 
$x_{\alpha}(t) = x(t) - \alpha \phi(x(t), t) \nabla_x g_1(x(t), t)$, $\alpha \ge 0$ 
(recall that $\phi(x(t), t) = \max\{ g_1(x(t), t), 0 \}$, i.e. the trajectory is shifted only at those points where the
state constraint is violated). It is easily seen that for any $t \in [0, T]$ satisfying the inequality
$g_1(x(t), t) > 0$ and for any sufficiently small $\alpha > 0$ one has 
$g_1(x_{\alpha}(t), t) < g_1(x(t), t)$, i.e. $x$ is shifted towards the feasible region. Thus, one can say that
assumption~\ref{Assumpt_StateConstr_Decay_Finite_p} of Theorem~\ref{Theorem_StateConstr_Nonlinear_Alternative} is an
assumption on the controllability of the system $\dot{x} = f(x, u, t)$ with respect to the state constraints.
\end{remark}

It should be noted that Theorem~\ref{Theorem_StateConstr_Nonlinear_Alternative} is mainly of theoretical interest,
since it does not seem possible to verify assumption~\ref{Assumpt_StateConstr_Decay_Finite_p} for any particular classes
of optimal control problems appearing in practice. Nevertheless, let us give a simple and illuminating example of a
problem in which this assumptions is satisfied.

\begin{example} \label{Example_StateInequalConstr_Exact}
Let $d = 2$ and $m = 1$. Define $U = \{ u \in L^{\infty}(0, T) \mid \| u \|_{\infty} \le 1 \}$, and
consider the following variable-endpoint optimal control problem with the state inequality constraint:
\begin{equation} \label{Ex_StateInequalConstr_Exact}
\begin{split}
  &\min \: \mathcal{I}(x) = \int_0^T \theta(x(t), t) dt + \zeta(x(T)) \\
  &\text{s.t.} \quad \begin{cases} \dot{x}^1 = 1 \\ \dot{x}^2 = u \end{cases} \quad
  x(0) = \begin{pmatrix} 0 \\ 0 \end{pmatrix}, \quad x(T) \in S_T, \quad u \in U, \quad g(x(t)) \le 0,
\end{split}
\end{equation}
where $g(x^1, x^2) = x^2$, the functions $\theta$ and $\zeta$ satisfy the assumptions of
Theorem~\ref{Theorem_StateConstr_Nonlinear_Alternative}, and $S_T = \{ T \} \times [ - \beta, 0 ]$ for some 
$\beta \ge 0$. Then $(x, u)$ with $x(t) \equiv (t, 0)^T$ and $u(t) \equiv 0$ is a feasible point of this problem.
Furthermore, by Proposition~\ref{Prop_CompleteMetricSpace} the space $X$ of trajectories of the system under
consideration is compact in $(C[0, T])^d$. Thus, it remains to check that
assumption~\ref{Assumpt_StateConstr_Decay_Finite_p} of Theorem~\ref{Theorem_StateConstr_Nonlinear_Alternative} holds
true. We will verify this assumptions with the use of the idea discussed in
Remark~\ref{Remark_ShiftTowardsFeasibleRegion}.

Note that $g(x(0)) = 0$, i.e. Slater's condition is not satisifed. Fix any $x \in A \setminus \Omega$, and let $x$
corresponds to a control input $u \in U$. For any $n \in \mathbb{N}$ define 
$$
  u_n(t) = \begin{cases}
    u(t), & \text{if } x^2(t) \le 0, \\
    \left(1 - \frac{1}{n} \right) u(t), & \text{if } x^2(t) > 0,
  \end{cases}
  \qquad
  x_n(t) = \begin{pmatrix} t \\ x^2(t) - \frac{1}{n} \max\{ x^2(t), 0 \}  \end{pmatrix}
$$
Observe that $x_n$ is a trajectory of the system corresponding to the control input $u_n$, and for any
$n \in \mathbb{N}$ one has $u_n \in U$, $x_n(0) = x(0) = (0, 0)^T$, $x(T) \in S_T$ by the definition of $A$, and
$x_n(T) \in S_T$, since by the definition of $S_T$ one has $x^2(T) \le 0$, which implies that $x_n(T) = x(T)$. Hence, in
particular, $x_n \in A$ and $|x_n(T) - x(T)| = 0 \le \| x_n - x \|_p$ for all $n \in \mathbb{N}$. The sequence 
$\{ x_n \}$ obviously converges to $x$ in $X$. Furthermore, note that $(x_n - x) / \| x_n - x \|_p = h$ with 
$h(\cdot) = (0, - \max\{ x^2(\cdot), 0 \} / \varphi(x))^T$ for all $n$, which obviously implies that the sequence 
$\{ (x_n - x) / \| x_n - x \|_p \}$ converges to $h$, and 
$$
  \int_0^T \phi(x(t), t)^{p - 1} \max_{v \in \partial_x \phi(x(t), t)} \langle v, h(t) \rangle \, dt
  = - \frac{1}{\varphi(x)} \int_0^T \max\{ x^2(t), 0 \}^p \, dt = - \varphi(x)^{p - 1},
$$
i.e. assumption~\ref{Assumpt_StateConstr_Decay_Finite_p} of Theorem~\ref{Theorem_StateConstr_Nonlinear_Alternative} is
satisfied with $a = 1$ and any $\eta > 0$. Thus, by Theorem~\ref{Theorem_StateConstr_Nonlinear_Alternative} one can
conclude that for any $1 < p < + \infty$ there exists $\lambda^* \ge 0$ such that for any $\lambda \ge \lambda^*$ the
penalised problem
\begin{align*}
  &\min \: \Phi_{\lambda}(x) = \int_0^T \theta(x(t), t) dt
  + \lambda \bigg( \int_0^T \max\{ x^2(t), 0 \}^p \, dt \bigg)^{1 / p} + \zeta(x(T)) \\
  &\text{s.t.} \quad \begin{cases} \dot{x}^1 = 1 \\ \dot{x}^2 = u \end{cases}
  x(0) = \begin{pmatrix} 0 \\ 0 \end{pmatrix}, \: x(T) \in S_T, \: u \in U
\end{align*}
is equivalent to problem \eqref{Ex_StateInequalConstr_Exact} in the sense that these problems have the same optimal
value, the same globally optimal solutions, as well as the same locally optimal solutions and inf-stationary points
with respect to the pseudometric $d((x_1, u_1), (x_2, u_2)) = \| x_1 - x_2 \|_p + |x_1(T) - x_2(T)|$ in
$W^2_{1, \infty}(0, T) \times L^{\infty}(0, T)$.
\end{example}

Theorem~\ref{Theorem_StateConstr_Nonlinear_Alternative} can be easily extended to the case of problems with
state equality constraints. Namely, suppose that there is a single state equality constraint: $g(x(t), t) = 0$ for all 
$t \in [0, T]$. Then one can define $\varphi(x) = \| g(x(\cdot), \cdot)\|_p$ for $1 < p < + \infty$. Arguing in a
similar way to the proof of Theorem~\ref{Theorem_StateConstr_Nonlinear_Alternative} one can verify that this theorem
remains to hold true for problems with one state equality constraint, if one replaces inequality
\eqref{StateConstr_Decay_Finite_P} with the following one:
\begin{equation} \label{StateEqualConstr_Decay_Finite_P}
  \int_0^T |g(x(t), t)|^{p - 1} \sign(g(x(t), t)) \langle \nabla_x g(x(t), t), h(t) \rangle \, dt 
  \le - a \varphi(x)^{p - 1}.
\end{equation}
As in Remark~\ref{Remark_ShiftTowardsFeasibleRegion}, one can verify that this inequality is satisfied for 
$h = y / \| y \|_p$ with $y = - g(x(t), t) \nabla_x g(x(t), t)$, provided $0 < a_1 \le | \nabla_x g(x, t) | \le a_2$ for
all $x$ and $t$. Let us utilise this result to demonstrate that exact penalisation of state equality constraints is
possible, if the cost functional $\mathcal{I}$ does not depend on the control inputs explicitly
(cf.~Example~\ref{CounterExample_StateEqConstr} with which we started our analysis of state constrained problems).

\begin{example}
Let $d = 2$ and $m = 2$. Define $U = \{ u = (u^1, u^2)^T \in L^2_{\infty}(0, T) \mid \| u \|_{\infty} \le 1 \}$, and
consider the following variable-endpoint optimal control problem with state equality constraint:
\begin{equation} \label{Ex_StateEqualConstr_Exact}
\begin{split}
  &\min \: \mathcal{I}(x) = \int_0^T \theta(x(t), t) dt + \zeta(x(T)) \\
  &\text{s.t.} \quad 
  \begin{cases}
    \dot{x}^1 = u^1 \\
    \dot{x}^2 = u^2
  \end{cases}
  x(0) = 0, \quad x(T) \in S_T, \quad u \in U, \quad g(x(t)) = x^1(t) + x^2(t) = 0 \quad \forall t \in [0, T].
\end{split}
\end{equation}
Here $S_T$ is a closed subset of the set $\{ x \in \mathbb{R}^2 \mid x^1 + x^2 = 0 \}$ such that $0 \in S_T$, while
$\theta$ and $\zeta$ satisfy the assumptions of Theorem~\ref{Theorem_StateConstr_Nonlinear_Alternative}. Note that $(x,
u)$ with $x(t) \equiv 0$ and $u(t) \equiv 0$ is a feasible point of problem \eqref{Ex_StateEqualConstr_Exact}.
Furthermore, by Proposition~\ref{Prop_CompleteMetricSpace} the space $X$ of trajectories of the system under
consideration is compact in
$(C[0, T])^d$. Thus, as one can easily verify, there exists a globally optimal solution of
\eqref{Ex_StateEqualConstr_Exact}.

Let us check that inequality \eqref{StateEqualConstr_Decay_Finite_P} holds true for any $x \in A \setminus \Omega$,
i.e. for any trajectory violating the state equality constraint. Indeed, fix any such $x$, and let $x$ corresponds to a
control input $u \in U$. For any $n \in \mathbb{N}$ define
$$
  u_n = u - \frac{g(u)}{n} \begin{pmatrix} 1 \\ 1 \end{pmatrix}, \qquad
  x_n = x - \frac{g(x)}{n} \begin{pmatrix} 1 \\ 1 \end{pmatrix}.
$$
Clearly, $x_n$ is a trajectory of the system corresponding to $u_n$, and for any $n \in \mathbb{N}$ one has 
$x_n(0) = 0$ and $x_n(T) = x(T) \in S_T$, since by our assumptions $x \in A = \{ x \in X \mid x(T) \in S_T \}$ and
$g(\xi) = 0$ for any $\xi \in S_T$. Note also that $u_n \in U$ for any $n \in \mathbb{N}$ due to the facts that
$$
  |u_n^1(t)| = \left| u^1(t) - \frac{u^1(t) + u^2(t)}{n} \right| 
  \le \frac{n - 1}{n} |u^1(t)| + \frac{1}{n} |u^2(t)| \le \frac{n - 1}{n} + \frac{1}{n} = 1
  \quad \text{for a.e. } t \in (0, T),
$$
and the same inequality holds true for $u_n^2$. Thus, $x_n \in A$ and $|x_n(T) - x(T)| = 0 \le \| x_n - x \|_p$
for all $n$. Moreover, for any $n \in \mathbb{N}$ one has $(x_n - x) / \| x_n - x \|_p = h$ with
$h = ( - g(x) / \sqrt{2} \varphi(x), - g(x) / \sqrt{2} \varphi(x))^T$, which obviously implies that the sequence 
$\{ (x_n - x) / \| x_n - x \|_p \}$ converges to $h$, and
$$
  \int_0^T |g(x(t))|^{p - 1} \sign(g(x(t))) \langle \nabla_x g(x(t)), h(t) \rangle \, dt
  = - \frac{2}{\sqrt{2} \varphi(x)} \int_0^T |g(x(t))|^p \, dt = - \sqrt{2} \varphi(x)^{p - 1}
$$
i.e. \eqref{StateEqualConstr_Decay_Finite_P} is satisfied with $a = \sqrt{2}$. Thus, one can conclude that for any 
$1 < p < + \infty$ the penalised problem
\begin{align*}
  &\min \: \Phi_{\lambda}(x) = \int_0^T \theta(x(t), t) dt
  + \lambda \bigg( \int_0^T |x^1(t) + x^2(t)|^p \, dt \bigg)^{1 / p} + \zeta(x(T)) \\
  &\text{s.t.} \quad
  \begin{cases}
    \dot{x}^1 = u^1 \\ 
    \dot{x}^2 = u^2
  \end{cases}
  x(0) = 0, \quad x(T) \in S_T, \quad u \in U
\end{align*}
is equivalent to problem \eqref{Ex_StateEqualConstr_Exact} for any sufficiently large $\lambda$.
\end{example}

\section{Conclusions}

In the second paper of our study we analysed the exactness of penalty functions for optimal control problems with
terminal and pointwise state constraints. We proved that penalty functions for fixed-endpoint optimal control problems
for linear time-varying systems and linear evolution equations are completely exact, if the terminal state belongs to
the relative interior of the reachable set. In the nonlinear case, the local exactness of the penalty function can be
ensured under the assumption that the linearised system is completely controllable, while the complete exactness of the
penalty function can be achieved under certain assumptions on the reachable set and the controllability of the system,
which require further investigation. 

We also proved that penalty functions for variable-endpoint optimal control problems for linear time-varying systems
with convex terminal constraints are completely exact, if Slater's condition holds true. In the case of nonlinear
variable-endpoint problems, the local exactness of a penalty function  was proved under the assumptions that the
linearised system is completely controllable, and the well-known Mangasarian-Fromovitz constraint qualification (MFCQ)
holds true for terminal constraints.

In the case of problems with pointwise state inequality constraints, we showed that penalty functions for such problems
for linear time-varying systems and linear evolution equations with convex state constraints are completely exact, if
the $L^{\infty}$ penalty term is used, and Slater's condition holds true. In the nonlinear case we proved the local
exactness of the $L^{\infty}$ penalty function under the assumption that a suitable constraint qualification is
satisfied, which resembles MFCQ. We also proved that the exact $L^p$ penalisation of pointwise state constraints with
finite $p$ is possible for convex problems, if Lagrange multipliers corresponding to state constraints belong to
$L^{p'}(0, T)$, and for nonlinear problems, if the cost functional does not depend on the control inputs explicitly and
some additional assumptions are satisfied.

A reason that the exact $L^p$ penalisation of state constraints with finite $p$ requires more restrictive assumption is
indirectly connected to the Pontryagin maximum principle. Indeed, if the penalty function with $L^p$ penalty term is
locally exact at a locally optimal solution $(x^*, u^*)$, then by definition for any sufficiently large $\lambda \ge 0$
the pair $(x^*, u^*)$ is a locally optimal solution of the penalised problem without state constraints:
\begin{align*}
  {}&\min \: \Phi_{\lambda}(x) = \int_0^T \theta(x(t), u(t), t) \, dt +
  \lambda \bigg( \int_0^T \max\{ g_1(x(t), t), \ldots, g_l(x(t), t), 0 \}^p \, dt \bigg)^{1/p} \\
  {}&\text{subject to } \dot{x}(t) = f(x(t), u(t), t) \, d \tau, \quad t \in [0, T], \quad 
  x(0) = x_0, \quad x(T) = x_T, \quad u \in U.
\end{align*}
It is possible to derive optimality conditions for this problem in the form of the Pontryagin maximum principle for the
original problem, in which Lagrange multipliers corresponding to state constraints necessarily belong to 
$L^{p'}[0, T]$, if $p < + \infty$. Therefore, for the exactness of the $L^p$ penalty function for state constraints with
finite $p$ it is necessary that there exists Lagrange multipliers corresponding to state constraints that belong to
$L^{p'}[0, T]$. If no such multipliers exist, then the exact $L^p$-penalisation with finite $p$ is impossible. 

Although we obtained a number of results on exact penalty functions for optimal control problems with terminal and
pointwise state constraints, a further research in this area is needed. In particular, it is interesting to find
verifiable sufficient conditions under which assumptions of Theorems~\ref{Theorem_FixedEndPointProblem_NonLinear},
\ref{Theorem_StateConstrainedProblem_Nonlinear}, and \ref{Theorem_StateConstr_Nonlinear_Alternative} on the complete
exactness of corresponding penalty functions hold true in the nonlinear case. Moreover, the main results of our study
can be easily extended to nonsmooth optimal control problems. In particular, one can suppose that the integrand
$\theta$ is only locally Lipschitz continuous in $x$ and $u$, and impose the same growth conditions on the Clarke
subdifferential (or some other suitable subdifferential), as we did on the derivatives of this functions. Also, it
seems worthwhile to analyse connections between necessary/sufficient optimality conditions and the local exactness of
penalty functions (cf.~the papers of Xing et al. \cite{Xing89,Xing94}, and Sections~4.6.2 and 4.7.2 in 
monograph~\cite{Polak_book}).

It should be noted that the general results on exact penalty functions that we utilised throughout our study are based
on completely independent assumptions on the cost functional and constraints
(see~Theorems~\ref{Theorem_CompleteExactness} and \ref{THEOREM_COMPLETEEXACTNESS_GLOBAL}). This approach allowed us to
consider counterexamples in which the cost functionals were unrealistic from a practical point of view
(cf.~Example~\ref{CounterExample_StateEqConstr}). Therefore, it seems profitable to obtain new general results on the
exactness of penalty functions with the use of assumptions that are based on the interplay between the cost functional
and constraints (cf.~such conditions for Huyer and Neumaier's penalty function in the finite dimensional case in
Reference\cite{WangMaZhou}).

Finally, for obvious reasons in this two-part study we restricted our consideration to several ``classical'' problems.
Our goal was not to apply the theory of exact penalty functions to as many optimal control problems as possible, but to
demonstrate the main tools, as well as merits and limitations, of this theory on several standard problems, in the hope
that it will help the interested reader to apply the exact penalty function method to the optimal control problem at
hand.

\section*{Acknowledgments}

The author wishes to express his thanks to the coauthor of the first part of this study A.V. Fominyh for many useful
discussions on exact penalty functions and optimal control problems that led to the development of this paper.

\bibliographystyle{abbrv}  
\bibliography{ExactPen_OptControl}

\section*{Appendix A. Proof of Theorem~\ref{THEOREM_COMPLETEEXACTNESS_GLOBAL}}

Observe that under the assumptions of Theorem~\ref{THEOREM_COMPLETEEXACTNESS_GLOBAL} assumptions of
Theorem~\ref{Theorem_CompleteExactness} are satisfied for any $c \in \mathbb{R}$ and $\delta > 0$. Therefore, by this
theorem there exists $\lambda^* \ge 0$ such that for any $\lambda \ge \lambda^*$ the optimal values and globally optimal
solutions of the problems $(\mathcal{P})$ and \eqref{PenalizedProblem} coincide.

Let $L > 0$ be a Lipschitz constant of $\mathcal{I}$ on $A$, and fix any $x \in A \setminus \Omega$. By our assumption 
$\varphi^{\downarrow}_A(x) \le - a < 0$. By the definition of the rate of steepest descent there exists a sequence 
$\{ x_n \} \subset A$ converging to $x$ and such that $\varphi(x_n) - \varphi(x) < - a d(x_n, x) / 2$ for all 
$n \in \mathbb{N}$. Therefore
$$
  \Phi_{\lambda}(x_n) - \Phi_{\lambda}(x) = \mathcal{I}(x_n) - \mathcal{I}(x) 
  + \lambda\big( \varphi(x_n) - \varphi(x) \big) \le \left( L  - \lambda \frac{a}{2} \right) d(x_n, x)
$$
for any $n \in \mathbb{N}$, which implies that $(\Phi_{\lambda})^{\downarrow}_A(x) < 0$ for all $\lambda > 2 L / a$ and
$x \in A \setminus \Omega$. Thus, if $x^* \in A$ is an inf-stationary point/point of local minimum of $\Phi_{\lambda}$
on $A$ and $\lambda > 2 L / a$, then $x^* \in \Omega$. Here we used the fact that any point of local minimum of
$\Phi_{\lambda}$ on $A$ is also an inf-stationary point of $\Phi_{\lambda}$ on $A$, since
$(\Phi_{\lambda})^{\downarrow}_A(x) \ge 0$ is a necessary condition for local minimum.

Fix any $\lambda > 2 L / a$. Let $x^* \in A$ be a point of local minimum of the penalised problem
\eqref{PenalizedProblem}. Then $x^* \in \Omega$. Hence bearing in mind the fact that by definition 
$\Phi_{\lambda}(x) = \mathcal{I}(x)$ for any $x \in \Omega$ one obtains that $x^*$ is a locally optimal solution of 
the problem $(\mathcal{P})$.

Let now $x^* \in \Omega$ be a locally optimal solution of $(\mathcal{P})$. Clearly, $x^* \in S_{\lambda}(c)$ for any 
$c > \Phi_{\lambda}(x^*)$. Hence by \cite[Lemma~1]{DolgopolikFominyh} there exists $r_1 > 0$ such that
$\varphi(x) \ge a \dist(x, \Omega)$ for all $x \in B(x^*, r_1) \cap A$. Furthermore, by 
\cite[Lemma~2 and Remark~11]{DolgopolikFominyh} there exists $r_2 > 0$ such that 
$\mathcal{I}(x) - \mathcal{I}(x^*) \ge - L \dist(x, \Omega)$ for any $x \in B(x^*, r_2) \cap A$. Consequently, for any
$x \in B(x^*, r) \cap A$ with $r = \min\{ r_1, r_2 \}$ one has
$$
  \Phi_{\lambda}(x) - \Phi_{\lambda}(x^*) = \mathcal{I}(x) - \mathcal{I}(x^*) 
  + \lambda \big( \varphi(x) - \varphi(x^*) \big) 
  \ge \big( - L + \lambda a \big) \dist(x, \Omega) \ge 0,
$$
i.e. $x^*$ is a locally optimal solution of the penalised problem \eqref{PenalizedProblem}. Thus, locally optimal
solutions of the problems $(\mathcal{P})$ and \eqref{PenalizedProblem} coincide for any $\lambda > 2 L / a$.

Let now $x^* \in A$ be an inf-stationary point of $\Phi_{\lambda}$ on $A$. Then $x^* \in \Omega$. By definition
$\Phi_{\lambda}(x) = \mathcal{I}(x)$ for any $x \in \Omega$, which yields
$\mathcal{I}^{\downarrow}_{\Omega}(x^*) = (\Phi_{\lambda})^{\downarrow}_{\Omega}(x^*) \ge
(\Phi_{\lambda})^{\downarrow}_A(x^*) \ge 0$, i.e. $x^*$ is an inf-stationary point of $\mathcal{I}$ on $\Omega$.

Let finally $x^* \in \Omega$ be an inf-stationary point of $\mathcal{I}$ on $\Omega$. By the definition of the rate of
steepest descent there exists a sequence $\{ x_n \} \subset A$ converging to $x^*$ such that
$$
  (\Phi_{\lambda})^{\downarrow}_A(x^*) 
  = \lim_{n \to \infty} \frac{\Phi_{\lambda}(x_n) - \Phi_{\lambda}(x^*)}{d(x_n, x^*)}.
$$
If there exists a subsequence $\{ x_{n_k} \} \subset \Omega$, then by the fact that $\varphi(x) = 0$ for all 
$x \in \Omega$ one gets that
$$
  (\Phi_{\lambda})^{\downarrow}_A(x^*) 
  = \lim_{k \to \infty} \frac{\Phi_{\lambda}(x_{n_k}) - \Phi_{\lambda}(x^*)}{d(x_{n_k}, x^*)}
  = \lim_{k \to \infty} \frac{\mathcal{I}(x_{n_k}) - \mathcal{I}(x^*)}{d(x_{n_k}, x^*)} 
  \ge \mathcal{I}^{\downarrow}_{\Omega}(x^*) \ge 0.
$$
Thus, one can suppose that $\{ x_n \} \subset A \setminus \Omega$.

Choose any $L' \in (L, \lambda a)$. By applying \cite[Lemmas~1 and 2]{DolgopolikFominyh} one obtains that
\begin{align*}
  \Phi_{\lambda}(x_n) - \Phi_{\lambda}(x^*) 
  &= \mathcal{I}(x_n) - \mathcal{I}(x^*) + \lambda \big( \varphi(x_n) - \varphi(x^*) \big) \\
  &\ge - L' \dist(x_n, \Omega) - (L' - L)d(x_n, x^*) + \lambda a \dist(x_n, \Omega)
  \ge - (L' - L)d(x_n, x^*)
\end{align*}
for any sufficiently large $n$. Dividing this inequality by $d(x_n, x^*)$, and passing to the limit as $n \to \infty$
one obtains that $(\Phi_{\lambda})^{\downarrow}_A(x^*) \ge - (L' - L)$, which implies that
$(\Phi_{\lambda})^{\downarrow}_A(x^*) \ge 0$ due to the fact that $L' \in (L, \lambda a)$ was chosen arbitrarily.
Consequently, $x^*$ is an inf-stationary point of $\Phi_{\lambda}$ on $A$. Thus, inf-stationary points of
$\Phi_{\lambda}$ on $A$ coincide with inf-stationary points of $\mathcal{I}$ on $\Omega$ for any $\lambda > 2 L / a$,
and the proof is complete.

\section*{Appendix B. Some Properties of Nemytskii Operators}

For the sake of completeness, in this appendix we give complete proofs of several well-known results on continuity and 
differentiability of Nemytskii operators (cf.~monograph~\cite{AppellZabrejko}). Firstly, we prove some auxiliary
results related to state constraints of optimal control problems.

\begin{proposition} \label{Prop_ContNonlinearMap_in_C}
Let $(Y, d)$ be a metric space, and $g \colon Y \times [0, T] \to \mathbb{R}$ be a continuous function. Then the
operator $G(x)(\cdot) = g(x(\cdot), \cdot)$ continuously maps $C([0, T]; Y)$ to $C[0, T]$.
\end{proposition}

\begin{proof}
Choose any $x \in C([0, T]; Y)$. Due to the continuity of $g$, for any $t \in [0, T]$ and $\varepsilon > 0$ there
exists $\delta(t) > 0$ such that for all $y \in Y$ and $\tau \in [0, T]$ with $d(y, x(t)) + |t - \tau| < \delta(t)$ one
has $|g(y, \tau) - g(x(t), t)| < \varepsilon / 2$. The set 
$K = \{ (x(t), t) \in Y \times \mathbb{R} \mid t \in [0, T] \}$ is compact as the image of the compact set 
$[0, T]$ under a continuous map. Therefore, there exist $N \in \mathbb{N}$ and $\{ t_1, \ldots, t_N \} \subset [0, T]$
such that $K \subset \cup_{k = 1}^N B( (x(t_k), t_k), \delta(t_k) / 2 )$. Define $\delta = \min_k \delta(t_k) / 2$.

Now, choose any $t \in [0, T]$ and $\overline{x} \in C([0, T]; Y)$ such that 
$\| \overline{x} - x \|_{C([0, T]; Y)} < \delta$. By definition one has $d(\overline{x}(t), x(t)) < \delta$.
Furthermore, there exists $k \in \{ 1, \ldots, N \}$ such that $(x(t), t) \in B( (x(t_k), t_k), \delta(t_k) / 2 )$,
which due to the definition of $\delta$ implies that $(\overline{x}(t), t) \in B( (x(t_k), t_k), \delta(t_k) )$. Hence
by the definition of $\delta(t_k)$ one has
$$
  \big| g(\overline{x}(t), t) - g(x(t), t) \big| \le 
  \big| g(\overline{x}(t), t) - g(x(t_k), t_k) \big| + \big| g(x(t_k), t_k) - g(x(t), t) \big|
  < \frac{\varepsilon}{2} + \frac{\varepsilon}{2} = \varepsilon,
$$
which yields $\| g(\overline{x}(\cdot), \cdot) - g(x(\cdot), \cdot) \|_{\infty} < \varepsilon$ due to the fact
that $t \in [0, T]$ is arbitrary. Thus, the operator $G$ continuously maps $C([0, T]; Y)$ to $C[0, T]$.
\end{proof}

\begin{corollary} \label{Corollary_StateConstrPenTerm_Contin}
Let $(Y, d)$ be a metric space, and $g_j \colon Y \times [0, T] \to \mathbb{R}$, $j \in J = \{ 1, \ldots, l \}$, be
continuous functions. Then the function $\varphi \colon C([0, T]; Y) \to \mathbb{R}$,
$\varphi(x) = \max_{t \in [0, T]} \max_{j \in J} g_j(x(t), t)$ is continuous.
\end{corollary}

\begin{proof}
Fix any $x \in C([0, T]; Y)$. By Proposition~\ref{Prop_ContNonlinearMap_in_C} for any $\varepsilon > 0$ there exists 
$\delta > 0$ such that for any $\overline{x} \in B(x, \delta)$ one has 
$\| g_j(\overline{x}(\cdot), \cdot) - g_j(x(\cdot), \cdot) \|_{\infty} < \varepsilon$ for all $j \in J$.
Consequently, for any such $\overline{x}$ one has
$$
  g_j(\overline{x}(t), t) \le g_j(x(t), t) + \varepsilon \le \varphi(x) + \varepsilon \quad 
  \forall t \in [0, T], j \in J.
$$
Taking the supremum, at first, over all $j \in J$, and then over all $t \in [0, T]$ one obtains that 
$\varphi(\overline{x}) \le \varphi(x) + \varepsilon$. Arguing in the same way but swapping $\overline{x}$ with $x$ one
obtains that $\varphi(x) \le \varphi(\overline{x}) + \varepsilon$. Therefore, 
$|\varphi(x) - \varphi(\overline{x})| < \varepsilon$, provided $\| x - \overline{x} \|_{C([0, T]; Y)} \le \delta$, i.e.
$\varphi$ is continuous.
\end{proof}

\begin{theorem} \label{Theorem_DiffStateConstr}
Let a function $g \colon \mathbb{R}^d \times [0, T] \to \mathbb{R}$, $g = g(x, t)$, be continuous, differentiable in
$x$, and the function $\nabla_x g$ be continuous. Then for any $p \in [1, + \infty]$ the Nemytskii operator 
$G(x) = g(x(\cdot), \cdot)$ maps $W^d_{1, p}(0, T)$ to $C[0, T]$, is continuously Fr\'{e}chet differentiable on 
$W^d_{1, p}(0, T)$, and its Fr\'{e}chet derivative has the form $D G(x)[h] = \nabla_x g(x(\cdot), \cdot) h(\cdot)$ for
all $x, h \in W^{1, p}(0, T)$.
\end{theorem}
 
\begin{proof}
Recall that we identify $W^d_{1, p}(0, T)$ with the space of all those absolutely continuous function 
$x \colon [0, T] \to \mathbb{R}^d$ for which $\dot{x} \in L^d_p(0, T)$ (see, e.g. \cite{Leoni}). Hence bearing
in mind the fact the function $g$ is continuous one obtains that for any $x \in W^d_{1, p}(0, T)$ one has
$g(x(\cdot), \cdot) \in C[0, T]$, i.e. the operator $G$ maps $W^d_{1, p}(0, T)$ to $C[0, T]$. Let us check that this
operator is Fr\'{e}chet differentiable.

Fix any $x, h \in W^d_{1, p}(0, T)$. By the mean value theorem for any $t \in [0, T]$ one has
\begin{equation} \label{StateConstr_MeanValue}
  \Big| \frac{1}{\alpha} \big( g(x(t) + \alpha h(t), t) - g(x(t), t) \big) 
  - \langle \nabla_x g(x(t), t), h(t) \rangle \Big| \le
  \sup_{\eta \in (0, \alpha)} | \nabla_x g(x(t) + \eta h(t), t) - \nabla_x g(x(t), t) | \| h \|_{\infty}.
\end{equation}
By Proposition~\ref{Prop_ContNonlinearMap_in_C} the function $x \mapsto \nabla_x g(x(\cdot), \cdot)$ continuously maps
$C[0, T]$ to $(C[0, T])^d$, which by inequality \eqref{SobolevImbedding} implies that it continuously maps 
$W^d_{1, p}(0, T)$ to $(C[0, T])^d$. Consequently, the right-hand side of \eqref{StateConstr_MeanValue} converges to
zero uniformly on $[0, T]$ as $\alpha \to +0$. Thus, one has
$$
  \lim_{\alpha \to + 0} 
  \left\| \frac{G(x + \alpha h) - G(x)}{\alpha} - \nabla_x g(x(\cdot), \cdot)) h(\cdot) \right\|_{\infty} = 0,
$$
i.e. the operator $G$ is G\^{a}teaux differentiable, and its G\^{a}teaux derivative has the form
$D G(x)[h] = \nabla_x g(x(\cdot), \cdot) h(\cdot)$. Note that the map $D G(\cdot)$ is continuous, since the nonlinear
operator $x \mapsto \nabla_x g(x(\cdot), \cdot)$ continuously maps $W^d_{1, p}(0, T)$ to $(C[0, T])^d$ by
Proposition~\ref{Prop_ContNonlinearMap_in_C} and inequality \eqref{SobolevImbedding}
Hence, as is well known, the operator $G$ is continuously Fr\'{e}chet differentiable, and its Fr\'{e}chet derivative
coincides with the G\^{a}teaux one.
\end{proof}

Let us also prove the differentiability of the Nemytskii operator 
$F(x, u) = \dot{x}(\cdot) - f(x(\cdot), u(\cdot), \cdot)$ associated with the nonlinear differential equation
$\dot{x} = f(x, u, t)$.

\begin{theorem} \label{Theorem_DiffNemytskiiOperator}
Let a function $f \colon \mathbb{R}^d \times \mathbb{R}^m \times [0, T] \to \mathbb{R}^d$, $f = f(x, u, t)$, be
continuous, differentiable in $x$ in $u$, and the functions $\nabla_x f$ and $\nabla_u f$ be continuous. Suppose also
that $q \ge p \ge 1$, and either $q = + \infty$ or $f$ and $\nabla_x f$ satisfy the growth 
condition of order $(q / p, p)$, while $\nabla_u f$ satisfies the growth condition of order $(q / s, s)$ with 
$s = qp / (q - p)$ in the case $q > p$, and $\nabla_u f$ does not depend on $u$ in the case $q = p$. Then the nonlinear
operator $F(x, u) = (\dot{x}(\cdot) - f(x(\cdot), u(\cdot), \cdot), x(T))$ maps 
$X = W^d_{1, p}(0, T) \times L^m_q(0, T)$ to $L^d_p(0, T) \times \mathbb{R}^d$, is continuously Fr\'{e}chet
differentiable, and its Fr\'{e}chet derivative has the form
$$
  D F(x, u)[h, v] = \begin{pmatrix} \dot{h}(\cdot) - A(\cdot) h(\cdot) - B(\cdot) v(\cdot) \\ h(T) \end{pmatrix},
  \quad A(t) = \nabla_x f(x(t), u(t), t), \quad B(t) = \nabla_u f(x(t), u(t), t)
$$
for any $(x, u) \in X$.
\end{theorem}

\begin{proof}
Let us prove that the Nemytskii operator $F_0(x, u) = f(x(\cdot), u(\cdot), \cdot)$ maps $X$ to $L^p_d(0, T)$, is
continuously Fr\'{e}chet differentiable, and its Fr\'{e}chet derivative has the form 
\begin{equation} \label{NemytskiiOperatorDerivative}
  D F_0(x, u)[h, v] = A(\cdot) h(\cdot) + B(\cdot) v(\cdot),
  \quad A(t) = \nabla_x f(x(t), u(t), t), \quad B(t) = \nabla_u f(x(t), u(t), t)
\end{equation}
for any $(x, u) \in X$ and $(h, v) \in X$. With the use of this result one can easily prove that the conclusion of the
theorem holds true.

Fix any $(x, u) \in X$. By inequality \eqref{SobolevImbedding} there exists $R > 0$ such that $\| x \|_{\infty} \le R$.
Then by the growth condition on the function $f$ there exist $C_R > 0$ and an a.e. nonnegative function 
$\omega_R \in L^p(0, T)$ such that 
$$
  |f(x(t), u(t), t)|^p \le \big( C_R |u(t)|^{q/p} + \omega_R(t) \big)^p 
  \le 2^p \big( C_R^p |u(t)|^q + \omega_R(t)^p \big).
$$
for a.e. $t \in (0, T)$. Observe that the right-hand side of this inequality belongs to $L^1(0, T)$. Therefore, 
$F_0(x, u) = f(x(\cdot), u(\cdot), \cdot) \in L_p^d(0, T)$, i.e. the operator $F_0$ maps $X$ to $L^p_d(0, T)$.
Now we turn to the proof of the Fr\'{e}chet differentiability of this operator. Let us consider two cases.

\textbf{Case $q = + \infty$.} Fix any $(x, u) \in X$, $(h, v) \in X$ and $\alpha \in (0, 1]$. By the mean value theorem
for a.e. $t \in (0, T)$ one has
\begin{multline} \label{MeanValue_NemytskiiOperator_Infty}
  \frac{1}{\alpha} \big| f(x(t) + \alpha h(t), u(t) + \alpha v(t), t) - f(x(t), u(t), t) 
  - \alpha \nabla_x f(x(t), u(t), t) h(t) - \alpha \nabla_u f(x(t), u(t), t) v(t) \big| \\
  \le \sup_{\eta \in (0, \alpha)} \esssup_{t \in [0, T]}
  \big| \nabla_x f(x(t) + \eta h(t), u(t) + \eta v(t), t) - \nabla_x f(x(t), u(t), t) \big| |h(t)| \\
  + \sup_{\eta \in (0, \alpha)} \esssup_{t \in [0, T]}
  \big| \nabla_u f(x(t) + \eta h(t), u(t) + \eta v(t), t) - \nabla_u f(x(t), u(t), t) \big| |v(t)|.
\end{multline}
With the use of the facts that all functions $x$, $h$, $u$, and $v$ are essentially bounded on
$[0, T]$, and the functions $\nabla_x f$ and $\nabla_u f$ are uniformly continuous one the compact set 
$B(\mathbf{0}_d, \| x \|_{\infty} + \| h \|_{\infty}) \times B(\mathbf{0}_m, \| u \|_{\infty} + \| v \|_{\infty})
\times [0, T]$ (here $\mathbf{0}_d$ is the zero vector from $\mathbb{R}^d$) one can verify that the right-hand
side of \eqref{MeanValue_NemytskiiOperator_Infty} converges to zero as $\alpha \to +0$. Observe also that
$A(\cdot) = \nabla_x f(x(\cdot), u(\cdot), \cdot) \in L^{d \times d}_{\infty}(0, T)$ and 
$B(\cdot) = \nabla_u f(x(\cdot), u(\cdot), \cdot) \in L^{d \times m}_{\infty}(0, T)$ due to the continuity of
$\nabla_x f$ and $\nabla_u f$ and the essential boundedness of $x$ and $u$. Hence, as is easy to check, the mapping 
$(h, v) \mapsto A(\cdot) h(\cdot) + B(\cdot) v(\cdot)$ is a bounded linear operator from $X$ to $L^d_{\infty}(0, T)$
(and, therefore, to $L^d_p(0, T)$). Thus, one has
\begin{align*}
  \lim_{\alpha \to 0} 
  &\left\| \frac{1}{\alpha} \big( F_0(x + \alpha h, u + \alpha v) - F_0(x, u) \big) - D F_0(x, u) [h, v] \right\|_p \\
  &\le T^{1/p} \lim_{\alpha \to 0} 
  \left\| \frac{1}{\alpha} \big( F_0(x + \alpha h, u + \alpha v) - F_0(x, u) \big) - D F_0(x, u) [h, v]
  \right\|_{\infty} = 0,
\end{align*}
where $D F_0(x, u)[h, v]$ is defined as in \eqref{NemytskiiOperatorDerivative} (here $1 / p = 0$, if $p = + \infty$).
Consequently, the Nemytskii operator $F_0$ is G\^{a}teaux differentiable at every point  $(x, u) \in X$, and its
G\^{a}teaux derivative has the form \eqref{NemytskiiOperatorDerivative}.

Let us check that the G\^{a}teaux derivative $D F_0(\cdot)$ is continuous on $X$. Then, as is well-known, $F_0$ is
continuously Fr\'{e}chet differentiable on $X$, and its Fr\'{e}chet derivative coincides with $D F_0(\cdot)$. Fix any
$(x, u) \in X$ and $(x', u') \in X$. For any $(h, v) \in X$ one has
\begin{align*}
  \| D F_0(x, u)[h, v] - D F_0(x', u')[h, v] \|_p 
  &\le T^{1/p} \| D F_0(x, u)[h, v] - D F_0(x', u')[h, v] \|_{\infty} \\
  &\le T^{1/p} \esssup_{t \in [0, T]}
  \big| \nabla_x f(x(t), u(t), t) - \nabla_x f(x'(t), u'(t), t) \big| \| h \|_{\infty} \\
  &+ T^{1/p} \esssup_{t \in [0, T]}
  \big| \nabla_u f(x(t), u(t), t) - \nabla_u f(x'(t), u'(t), t) \big| \| v \|_{\infty}.
\end{align*}
Hence with the use of \eqref{SobolevImbedding} one obtains that there exists $C_p > 0$ (depending only on $p$ and $T$)
such that
\begin{align*}
  \| D F_0(x, u) - D F_0(x', u') \| &\le T^{1/p} C_p \esssup_{t \in [0, T]}
  \big| \nabla_x f(x(t), u(t), t) - \nabla_x f(x'(t), u'(t), t) \big| \\
  &+ T^{1/p} \esssup_{t \in [0, T]} \big| \nabla_u f(x(t), u(t), t) - \nabla_u f(x'(t), u'(t), t) \big|.
\end{align*}
Utilising this inequality and taking into account the fact that the functions $\nabla_x f$ and $\nabla_u f$ are
continuous one can verify via a simple $\varepsilon-\delta$ argument that $\| D F_0(x, u) - D F_0(x', u') \| \to 0$ as 
$(x', u') \to (x, u)$ in $X$ (cf. the proof of Proposition~\ref{Prop_ContNonlinearMap_in_C}). Thus, the mapping 
$D F_0(\cdot)$ is continuous, and the proof of the case $q = + \infty$ is complete.

\textbf{Case $q < + \infty$}. Fix any $(x, u) \in X$, $(h, v) \in X$ and $\alpha \in (0, 1]$. By the mean value theorem
\begin{multline} \label{MeanValue_NemytskiiOperator}
  \frac{1}{\alpha} \big| f(x(t) + \alpha h(t), u(t) + \alpha v(t), t) - f(x(t), u(t), t) 
  - \alpha \nabla_x f(x(t), u(t), t) h(t) - \alpha \nabla_u f(x(t), u(t), t) v(t) \big|^p \\
  \le 2^p \sup_{\eta \in (0, \alpha)}
  \big| \nabla_x f(x(t) + \eta h(t), u(t) + \eta v(t), t) - \nabla_x f(x(t), u(t), t) \big|^p |h(t)|^p \\
  + 2^p \sup_{\eta \in (0, \alpha)}
  \big| \nabla_u f(x(t) + \eta h(t), u(t) + \eta v(t), t) - \nabla_u f(x(t), u(t), t) \big|^p |v(t)|^p
\end{multline}
for a.e. $t \in (0, T)$. Our aim is to apply Lebesgue's dominated convergence theorem.

The right-hand side of \eqref{MeanValue_NemytskiiOperator} converges to zero as $\alpha \to 0$ for a.e. 
$t \in (0, T)$ due to the continuity of $\nabla_x f$ and $\nabla_u f$. By applying \eqref{SobolevImbedding}, and the
facts that $\alpha \in (0, 1]$ and $\nabla_x f$ satisfies the growth condition of order $(q / p, p)$ one obtains that
there exist $C_R > 0$ and an a.e. nonnegative function $\omega_R \in L^p(0, T)$ such that
\begin{multline*}
  \sup_{\eta \in (0, \alpha)} 
  \big| \nabla_x f(x(t) + \eta h(t), u(t) + \eta v(t), t) - \nabla_x f(x(t), u(t), t) \big|^p |h(t)|^p \\
  \le 2^p \sup_{\eta \in (0, \alpha)} \big| \nabla_x f(x(t) + \eta h(t), u(t) + \eta v(t), t) \big|^p 
  C_p \| h \|_{1, p}^p
  + 2^p \sup_{\eta \in (0, \alpha)} \big| \nabla_x f(x(t), u(t), t) \big|^p C_p \| h \|_{1, p}^p \\
  \le 2^{2p} \Big( \big( C_R^p 2^q (|u(t)|^q + |v(t)|^q) + \omega_R(t)^p \big) 
  + \big( C_R^p |u(t)|^q + \omega_R(t)^p \big) \Big) C_p \| h \|_{1, p}^p
\end{multline*}
for a.e. $t \in (0, T)$. Observe that the right-hand side of this inequality belongs to $L^1(0, T)$ and does not depend
on $\alpha$, i.e. the first term in the right-hand side of \eqref{MeanValue_NemytskiiOperator} can be bounded above by a
function from $L^1(0, T)$ that is independent of $\alpha$.

Let us now estimate the second term in the right-hand side of \eqref{MeanValue_NemytskiiOperator}. Let $q > p$. Bearing
in mind the fact that $\nabla_u f$ satisfies the growth condition of order $(q/s, s)$ one obtains that there exists 
$C_R > 0$ and an a.e. nonnegative function $\omega_R \in L^s(0, T)$ such that
\begin{multline*}
  \sup_{\eta \in (0, \alpha)}
  \big| \nabla_u f(x(t) + \eta h(t), u(t) + \eta v(t), t) - \nabla_u f(x(t), u(t), t) \big|^p |v(t)|^p \\
  \le 2^p \Big( \big| C_R 2^{q/s} (|u(t)|^{q/s} + |v(t)|^{q/s}) + \omega_R(t) \big|^p 
  + \big| C_R |u(t)|^{q/s} + \omega_R(t) \big|^p \Big) |v(t)|^p
\end{multline*}
for a.e. $t \in (0, T)$. Let us check that the right-hand side of this inequality belongs to $L^1(0, T)$. Indeed, 
by applying H\"{o}lder's inequality of the form 
\begin{equation} \label{HolderInequality_p_to_s_q}
  \Big(\int_0^T |y_1(t)|^p |y_2(t)|^p \, dt \Big)^{1/p} \le \| y_1 \|_s \| y_2 \|_q
\end{equation}
(here we used the fact that $(q/p)' = s/p$) one gets that 
\begin{multline*}
  \Big( \int_0^T \Big| C_R 2^{q/s} (|u(t)|^{q/s} + |v(t)|^{q/s}) + \omega_R(t) \Big|^p |v(t)|^p \, dt \Big)^{1/p} \\
  \le \Big\| C_R 2^{q/s} (|u(\cdot)|^{q/s} + |v(\cdot)|^{q/s}) + \omega_R(\cdot) \Big\|_s \| v \|_q
  \le \Big( C_R 2^{q/s} \big( \| u \|_q^{q/s} + \| v \|_q^{q/s} \big) + \| \omega_R \|_s \Big) \| v \|_q < + \infty.
\end{multline*}
Thus, the last term in the right-hand side of inequality \eqref{MeanValue_NemytskiiOperator} can also be bounded above
by a function from $L^1(0, T)$ that does not depend on $\alpha$.

Finally, recall that in the case $q = p$ the function $\nabla_u f$ does not depend on $u$, which implies that it
satisfies the growth condition of order $(0, + \infty)$, i.e. for any $R > 0$ there exists $C_R > 0$ such that
$|\nabla_u f(x, u, t)| \le C_R$ for a.e. $t \in (0, T)$ and for all $(x, u) \in \mathbb{R}^d \times \mathbb{R}^m$ with
$|x| \le R$. Therefore, as is easy to check, in this case there exists $C > 0$ (that does not depend on $\alpha$)
such that
$$
  \sup_{\eta \in (0, \alpha)}
  \big| \nabla_u f(x(t) + \eta h(t), u(t) + \eta v(t), t) - \nabla_u f(x(t), u(t), t) \big|^p |v(t)|^p
  \le C |v(t)|^q
$$
for a.e. $t \in (0, T)$. The right-hand side of this inequality obviously belongs to $L^1(0, T)$.

Thus, the right-hand side of \eqref{MeanValue_NemytskiiOperator} can be bounded above by a function from $L^1(0, T)$
that does not depend on $\alpha$. Furthermore, from the growth conditions on $\nabla_x f$ and $\nabla_u f$ it follows
that $A(\cdot) = \nabla_x f(x(\cdot), u(\cdot), \cdot) \in L^{d \times d}_p(0, T)$ and 
$B(\cdot) = \nabla_u f(x(\cdot), u(\cdot), \cdot) \in L^{d \times m}_s(0, T)$, which, as is easily seen, implies that
the mapping $(h, v) \mapsto A(\cdot) h(\cdot) + B(\cdot) v(\cdot)$ is a bounded linear operator from 
$X$ to $L^d_p(0, T)$ (here $s = + \infty$ in the case $p = q$). Therefore, integrating
\eqref{MeanValue_NemytskiiOperator} from $0$ to $T$ and passing to the limit as $\alpha \to 0$ with the use of
Lebesgue's dominated convergence theorem one obtains that
$$
  \lim_{\alpha \to 0} 
  \left\| \frac{1}{\alpha} \big( F_0(x + \alpha h, u + \alpha v) - F_0(x, u) \big) - D F_0(x, u) [h, v] \right\|_p = 0,
$$ 
where $D F_0(x, u)[h, v]$ is defined as in \eqref{NemytskiiOperatorDerivative}. Thus, the Nemytskii operator $F_0$ is
G\^{a}teaux differentiable at every point $(x, u) \in X$, and its G\^{a}teaux derivative has the form
\eqref{NemytskiiOperatorDerivative}. Let us check that this derivative is continuous. Then one can conclude that $F_0$
is continuously Fr\'{e}chet differentiable on $X$, and its Fr\'{e}chet derivative coincides with $D F_0(\cdot)$.

Indeed, choose any $(x, u) \in X$ and $(x', u') \in X$. With the use of \eqref{NemytskiiOperatorDerivative} and
H\"{o}lder's inequality of the form \eqref{HolderInequality_p_to_s_q} one obtains
\begin{align*}
  \| D F_0(x, u)[h, v] - D F_0(x', u')[h, v] \|_p 
  &\le \| \nabla_x f(x(\cdot), u(\cdot), \cdot) - \nabla_x f(x'(\cdot), u'(\cdot), \cdot) \|_p \| h \|_{\infty} \\
  &+ \| \nabla_u f(x(\cdot), u(\cdot), \cdot) - \nabla_u f(x'(\cdot), u'(\cdot), \cdot) \|_s \| v \|_q
\end{align*}
for any $(h, v) \in X$ (in the case $q = p$ we put $s = \infty$). Hence taking into account \eqref{SobolevImbedding} one
gets
$$
  \| D F_0(x, u) - D F_0(x', u') \| 
  \le C_p \| \nabla_x f(x(\cdot), u(\cdot), \cdot) - \nabla_x f(x'(\cdot), u'(\cdot), \cdot) \|_p
  + \| \nabla_u f(x(\cdot), u(\cdot), \cdot) - \nabla_u f(x'(\cdot), u'(\cdot), \cdot) \|_s
$$
Therefore, the mapping $(x, u) \mapsto D F_0(x, u)$ is continuous in the operator norm, if for any sequence
$\{ (x_n, u_n) \}$ in $X$ converging to $(x, u)$ the sequence $\{ \nabla_x f(x_n(\cdot), u_n(\cdot), \cdot) \}$
converges to $\nabla_x f(x(\cdot), u(\cdot), \cdot)$ in $L^{d \times d}_p(0, T)$, while the sequence 
$\{ \nabla_u f(x_n(\cdot), u_n(\cdot), \cdot) \}$ converges to $\nabla_u f(x(\cdot), u(\cdot), \cdot)$ in 
$L^{d \times m}_s(0, T)$.\footnote{Note that in the case $p = q < + \infty$ one must prove that the sequence 
$\{ \nabla_u f(x_n(\cdot), u_n(\cdot), \cdot) \}$ converges in $L^{d \times m}_{\infty}(0, T)$, while $\{ u_n \}$
converges only in $L^m_q(0, T)$ with $q < + \infty$. That is why in this case one must assume that $\nabla_u f$ does not
depend on $u$, i.e. $f$ is affine in control} 

Let us prove the convergence of the sequence $\{ \nabla_x f(x_n(\cdot), u_n(\cdot), \cdot) \}$. The convergence of the
sequence $\{ \nabla_u f(x_n(\cdot), u_n(\cdot), \cdot) \}$ can be proved in the same way. Arguing by reductio ad
absurdum, suppose that there exists a sequence $\{ (x_n, u_n) \} \subset X$ converging to $(x, u)$ such that the
sequence $\{ \nabla_x f(x_n(\cdot), u_n(\cdot), \cdot) \}$ does not converge to $\nabla_x f(x(\cdot), u(\cdot), \cdot)$
in $L^{d \times d}_p(0, T)$. Then there exist $\varepsilon > 0$ and a subsequence $\{ (x_{n_k}, u_{n_k}) \}$ such that 
\begin{equation} \label{NonConvergenceInLp}
  \big\| \nabla_x f(x_{n_k}(\cdot), u_{n_k}(\cdot), \cdot) - \nabla_x f(x(\cdot), u(\cdot), \cdot) \big\|_p 
  \ge \varepsilon
  \quad \forall k \in \mathbb{N}
\end{equation}
It should be noted that all functions $\{ \nabla_x f(x_n(\cdot), u_n(\cdot), \cdot) \}$ belong
to $L^{d \times d}_p(0, T)$ due to the fact that $\nabla_x f$ satisfies the growth condition of order $(q/p, p)$.

By \eqref{SobolevImbedding} the sequence $\{ x_{n_k} \}$ converges to $x$ uniformly on $[0, T]$, which implies that 
$\| x_{n_k} \|_{\infty} \le R$ for all $k \in \mathbb{N}$ and some $R > 0$. The sequence $\{ u_{n_k} \}$ converges to
$u$ in $L^m_q(0, T)$. Hence, as is well-known, there exists a subsequence, which we denote again by $\{ u_{n_k} \}$,
that converges to $u$ almost everywhere. Consequently, by the continuity of $\nabla_x f$ the subsequence 
$\{ \nabla_x f(x_{n_k}(t), u_{n_k}(t), t) \}$ converges to $\nabla_x f(x(t), u(t), t)$ for a.e. $t \in (0, T)$.

The sequence $\{ u_{n_k} \}$ converges to $u$ in $L^m_q(0, T)$. Therefore, by the ``only if'' part of Vitali's theorem
characterising convergence in $L^p$ spaces (see, e.g. \cite[Theorem~III.6.15]{DunfordSchwartz}) for any 
$\varepsilon > 0$ there exists $\delta(\varepsilon) > 0$ such that for any Lebesgue measurable set $E \subset (0, T)$
with $\mu(E) < \delta(\varepsilon)$ (here $\mu$ is the Lebesgue measure) one has 
$\int_E |u_{n_k}|^q \, d \mu < \varepsilon$ for all $k \in \mathbb{N}$. Hence by applying the fact that $\nabla_x f$
satisfies the growth condition of order $(q/p, p)$ one obtains that there exist $C_R > 0$ and an a.e. nonnegative
function $\omega_R \in L^p(0, T)$ such that for any measurable set $E \subset (0, T)$ with 
$\mu(E) < \delta(\varepsilon)$ one has
$$
  \int_E |\nabla_x f(x_{n_k}(t), u_{n_k}(t), t)|^p d \mu(t) \le \int_E | C_R |u|^{q/p} + \omega_R |^p d \mu
  \le 2^p \Big( C_R^p \varepsilon + \int_E \omega_R^p d \mu \Big).
$$
Taking into account the absolute continuity of the Lebesgue integral and the fact that $\omega_R \in L^p(0, T)$, and
decreasing $\delta(\varepsilon) > 0$, if necessary, one can suppose that $\int_E \omega_R^p d \mu < \varepsilon$.
Therefore, choosing a sufficiently small $\varepsilon > 0$ one can make the integral
$\int_E |\nabla_x f(x_{n_k}(t), u_{n_k}(t), t)|^p d \mu(t)$ arbitrarily small for all $k \in \mathbb{N}$ and measurable
sets $E \subset (0, T)$ with $\mu(E) < \delta(\varepsilon)$. Consequently, by the ``if'' part of Vitali's theorem on
convergence in $L^p$ spaces the sequence $\{ \nabla_x f(x_{n_k}(\cdot), u_{n_k}(\cdot), \cdot) \}$ converges to 
$\nabla_x f(x(\cdot), u(\cdot), \cdot)$ in $L^{d \times d}_p(0, T)$, which contradicts \eqref{NonConvergenceInLp}.
\end{proof}

\end{document}